\documentclass[twoside,DIV18,abstract]{scrartcl}
\usepackage{latexsym,amsmath,amssymb,graphics,graphicx,amscd,amsfonts}
\usepackage[english]{babel}
\usepackage[thmmarks,amsmath]{ntheorem}
\usepackage[arrow, matrix, curve]{xy}
\usepackage{eucal}
\usepackage{abstract}
\usepackage{hyperref}
\newtheorem{theorem}{Theorem}[section]
\newtheorem{definition}[theorem]{Definition}
\newtheorem{lemma}[theorem]{Lemma}
\newtheorem{proposition}[theorem]{Proposition}
\newtheorem{corollary}[theorem]{Corollary}
\theorembodyfont{\normalfont}
\theoremsymbol{\ensuremath{\lozenge}}
\newtheorem{example}[theorem]{Example}
\theoremsymbol{\ensuremath{\triangle}} 
\newtheorem{remark}[theorem]{Remark}
\theoremstyle{nonumberplain}
\theoremsymbol{\ensuremath{\square}}
\theoremheaderfont{\scshape}
\theoremseparator{:}
\theorembodyfont{\normalfont}
\newtheorem{proof}{Proof}
\usepackage{dsfont}

\DeclareMathOperator{\pr}{pr}

\newcommand{\V}{{\mathcal{V}}}
\newcommand{\TT}{{\mathbb{T}}}

\newcommand{\tg}{{\mathsf{t}}}
\newcommand{\s}{{\mathsf{s}}}
\newcommand{\m}{{\mathsf{m}}}
\newcommand{\rr}{{\rightrightarrows}}

\newcommand{\R}{\mathbb{R}}

\newcommand{\inv}{^{-1}}

\newcommand{\mx}{\mathfrak{X}}
\newcommand{\lie}[1]{\mathfrak{#1}}
\newcommand{\dr}{\mathbf{d}}

\newcommand{\ldr}[1]{{{\pounds}}_{#1}}
\newcommand{\ip}[1]{{\mathbf{i}}_{#1}}
\newcommand{\an}[1]{\arrowvert_{#1}}

\DeclareMathOperator{\Id}{Id}

\DeclareMathOperator{\dom}{Dom}

\DeclareMathOperator{\Ad}{Ad}

\DeclareMathOperator{\graph}{Graph}
\DeclareMathOperator{\Exp}{Exp}

\allowdisplaybreaks[1]

\begin{document}
\title{Infinitesimal objects associated to
 {D}irac groupoids and  their homogeneous spaces}
\author{M. Jotz
}
\date{}

\maketitle

\begin{abstract}
Let $(G\rr P, \mathsf D_G)$ be a  Dirac groupoid.
We show that there are  natural Lie algebroid structures on the units 
$\lie A(\mathsf D_G)$
and on the \emph{core} $I^\tg(\mathsf D_G)$ of the multiplicative Dirac structure.
In the Poisson case, the Lie algebroid $A^*G$ is isomorphic
to $\lie A(\mathsf D_G)$ and in the  case of a closed $2$-form, the $IM$-2-form 
\cite{BuCrWeZh04} is equivalent to the  core algebroid
that we find.
We construct a vector bundle $\lie B(\mathsf D_G)\to P$ associated to any 
(almost) Dirac structure.
In the  Dirac case, $\lie B(\mathsf D_G)$
 has the structure of a Courant algebroid that generalizes 
the Courant algebroid defined by the Lie bialgebroid of a Poisson groupoid 
\cite{LiWeXu97}. 
This Courant algebroid structure is induced
in a natural way by the ambient Courant algebroid $TG\oplus T^*G$. 

The theorems in \cite{Drinfeld93}, \cite{LiWeXu98} and \cite{Jotz11a} 
 about one-one correspondence between the homogeneous 
spaces of a Poisson Lie group (respectively Poisson groupoid, Dirac Lie group)
and suitable Lagrangian subspaces of the 
Lie bialgebra or  Lie bialgebroid
are generalized to a classification of the Dirac homogeneous spaces 
of a Dirac groupoid. $\mathsf D_G$-homogeneous Dirac structures 
on $G/H$ are related to
 suitable Dirac structures in $\lie B(\mathsf D_G)$.
In the case of almost Dirac structures, we find 
Lagrangian subspaces of $\lie B(D_G)$, that are invariant under an 
induced action of the bisections of $H$ on $\lie B(\mathsf D_G)$.


\medskip


\end{abstract}

\tableofcontents

\section{Introduction}
A \emph{Poisson groupoid}  is a Lie groupoid  endowed with a
 Poisson structure that is compatible with the partial multiplication.
 Poisson Lie groups were introduced in \cite{Weinstein88b}, as a common generalization
of the symplectic groupoids and the Poisson Lie groups \cite{Drinfeld83a},
 and studied 
in \cite{Xu95}, \cite{MaXu94} among others.

A Poisson structure on a homogeneous space  of a Poisson groupoid 
is \emph{homogeneous} if the action of the Lie groupoid on the homogeneous space is 
compatible with the Poisson structures on the Lie groupoid
and on the homogeneous space. Poisson homogeneous spaces of a Poisson groupoid
are in correspondence with suitable Dirac structures in the direct sum of the Lie
algebroid with its dual \cite{LiWeXu98}.
Hence, the homogeneous spaces are encoded in terms 
of the infinitesimal data of the Poisson groupoid.
 We show that this correspondence result
fits into a  more general and natural context: the one of \emph{Dirac groupoids},
which are objects generalizing Poisson groupoids
and multiplicative closed $2$-forms on groupoids \cite{BuCrWeZh04}.
This result gives some insight on the problem of finding the 
infinitesimal data of Dirac groupoids,
something that is not fully understood yet.

\bigskip

Let $G\rr P$ be a  Lie groupoid endowed with a  Poisson 
bivector field $\pi_G\in\Gamma\left(\bigwedge^2TG\right)$.
The bivector field $\pi_G$ is 
\emph{multiplicative}
if the vector bundle map 
$\pi_G^\sharp:T^*G\to TG$
is a Lie groupoid morphism over some map $A^*G\to TP$ \cite{MaXu94}, 
where 
$A^*G$ is the dual of the Lie algebroid $AG$ of $G\rr P$ and
$TG\rr TP$ and $T^*G\rr A^*G$ are endowed with the tangent and cotangent
Lie groupoid structures (see \cite{CoDaWe87}, \cite{Pradines88}, \cite{Mackenzie05}).
Equivalently, the graph  of 
$\pi_G^\sharp$,  $\graph(\pi_G^\sharp)\subseteq TG\oplus T^*G$, is a subgroupoid of the 
\emph{Pontryagin groupoid}
$(TG\oplus T^*G)\rr(TP\oplus A^*G)$. The pair
 $(G\rr P,\pi_G)$ is then a \emph{Poisson groupoid}.

A Poisson groupoid $(G\rr P,\pi_G)$ induces a Lie algebroid structure
on the dual $A^*G$ of $AG$ and a Courant algebroid structure on the direct
sum  $AG\oplus A^*G$.
This was shown by
\cite{MaXu94} and \cite{LiWeXu97}, the pair $(AG,A^*G)$ is the \emph{Lie bialgebroid}
associated to  $(G\rr P, \pi_G)$. 
If
$G\rightrightarrows P$ is a target-simply
    connected Lie groupoid, and if $(AG, A^*G)$ is a Lie bialgebroid,
 then there exists a unique multiplicative Poisson structure
$\pi_G$
    on $G$ such that $(G\rr P, \pi_G)$ is a Poisson groupoid with Lie bialgebroid 
$(AG,
    A^*G)$ \cite{MaXu00}. This generalizes a theorem in \cite{Drinfeld83a} 
about one-one correspondence between \emph{Lie bialgebra} structures on 
$(\lie g, \lie g^*)$
 over 
the Lie algebra $\lie g$ of a connected and simply connected Lie group $G$,
and multiplicative Poisson structures on $G$ (see for instance \cite{Lu90}).
 
\medskip 

In the same spirit, a closed $2$-form $\omega_G$ on a Lie groupoid 
$G\rr P$ is multiplicative 
if $\m^*\omega_G=\pr_1^*\omega_G+\pr_2^*\omega_G$, where 
$\m:G\times _PG \to G$ is the multiplication map of the groupoid and 
$\pr_1,\pr_2:G\times_P G\to G$ 
are the projections. Equivalently, 
the map $\omega_G^\flat:TG\to T^*G$ associated to $\omega_G$ is a Lie groupoid morphism 
over a map $\lambda:TP\to A^*G$.
It has been shown in \cite{BuCrWeZh04} and 
\cite{BuCaOr09} that multiplicative 
closed $2$-forms on a Lie groupoid $G\rr P$ are in one-one correspondence 
with IM-$2$-forms; special maps $\sigma:AG\to T^*P$ satisfying
some algebraic and differential conditions. 
The correspondence is given by $\sigma=-\lambda^t$.

\medskip 

 Dirac structures generalize simultaneously 
Poisson brackets and closed $2$-forms  
in the
sense that the graphs of the vector bundle homomorphisms $\pi^\sharp:T^*M\to
TM$ and  $\omega^\flat:TM\to T^*M$
associated to a Poisson bivector $\pi$ on $M$ 
and  a closed $2$-form $\omega\in\Omega^2(M)$ define  Dirac
structures on the manifold $M$.
Hence, it is natural to ask how to recover the
 two results above on classification of multiplicative Poisson bivectors and 
closed $2$-forms
on a Lie groupoid in terms of data on its algebroid, which are by nature very different,
as special cases of a more general result about the infinitesimal data of 
\emph{Dirac groupoids}.
These objects have been defined in \cite{Ortiz08t}; a Dirac groupoid 
is a  groupoid endowed with 
a Dirac structure that is a subgroupoid of the Pontryagin groupoid $(TG\oplus
T^*G)\,\rr \,(TP\oplus A^*G)$.

\medskip 

It has been shown in \cite{Ortiz08t}
that multiplicative  Dirac structures
on a Lie groupoid $G\rr P$ are in one-one correspondence 
with Dirac structures on the Lie algebroid $AG\to P$,
which are at the same time 
subalgebroids of the natural Lie algebroid $T(AG)\oplus T^*(AG)\to TP\oplus A^*G$
defined by the Lie algebroid structure on $AG$.

Yet, this result does not generalize the results given above in the 
Poisson and closed $2$-form cases, 
but relates, modulo canonical identifications, the multiplicative 
Poisson bivectors and closed $2$-forms to 
the associated Lie algebroid maps $A(\pi_G^\sharp):A(TG)\to A(T^*G)$ and 
$A(\omega_G^\flat):A(T^*G)\to A(TG)$. In \cite{MaXu00} and 
\cite{BuCaOr09}, the construction of these maps from the infinitesimal data is 
an intermediate step in the reconstruction of $\pi_G$ and $\omega_G$ from the
Lie bialgebroid $(AG,A^*G)$ and, respectively, the IM-$2$-form $\sigma:AG\to T^*P$.

\medskip

We  show in this paper that, given a Dirac  groupoid $(G\rr P,\mathsf D_G)$,
 there is an induced Lie 
algebroid structure on the units $\lie A( \mathsf D_G)=\mathsf D_G\cap(TP\oplus A^*G)$ 
of the multiplicative Dirac structure (Theorem \ref{lie_algebroid_dual}).
This  was predicted by \cite{Ortiz08t}
and, since $\lie A( \mathsf D_G)$ 
is the graph of the anchor map of $A^*G$ in the Poisson case, 
generalizes the fact that a multiplicative Poisson structure on $G\rr P$ defines
a Lie algebroid structure on $A^*G$.
Since $\mathsf D_G$ has then 
the structure of an \emph{$\mathcal L$A-groupoid}, we recover 
 the natural Lie algebroid structure on the core
$I^\tg(\mathsf D_G)=\mathsf D_G\cap(AG\oplus T^*P)$
of the multiplicative Dirac structure (Proposition \ref{other_algebroid}), see  
\cite{Mackenzie00}. 
In the case of a closed $2$-form, 
this Lie algebroid is just the graph of the IM-$2$-form, hence completely 
equivalent to it.
We  show then  that the Courant algebroid structure
on $TG\oplus T^*G$ defines naturally a Courant algebroid $\lie B(\mathsf D_G)$ over $P$, 
on a vector bundle that is isomorphic to $\lie A(\mathsf D_G)
\oplus \lie A(\mathsf D_G)^*$
(Theorem \ref{theorem_courant_algebroid}).
In the Poisson case, we recover exactly the Courant algebroid $AG\oplus A^*G\to P$
defined by the Lie bialgebroid $(AG,A^*G)$. This new approach shows hence how to
see the Courant algebroid $AG\oplus A^*G$ defined 
by the Lie bialgebroid $(AG, A^*G)$ of a Poisson  groupoid
as a suitable restriction of the ambient Courant algebroid structure on $TG\oplus T^*G$.

We show also that the integrability properties of the Dirac structure are completely 
encoded in the pair of algebroids $(\lie A(\mathsf D_G), I^\tg(\mathsf D_G))$ (Theorem \ref{int_crit_thm}).

\bigskip

In the second part of this paper, we 
focus on Dirac homogeneous spaces of Dirac groupoids.
 A \emph{Poisson homogeneous space} $(X,\pi_X)$ of a Poisson groupoid
$(G\rr P,\pi_G)$ 
is a homogeneous space $X$ of $G\rr P$ endowed with a Poisson structure
$\pi_X$  that is compatible with
the  action of $G\rr P$ on $J:X\to P$ (see \cite{LiWeXu98} for more
details).

It has been shown in   \cite{Drinfeld93}
that the Poisson homogeneous spaces of a Poisson Lie group
are classified by a special class 
of Lagrangian subalgebras of the double Lie algebra $\lie g\times \lie g^*$
defined by the Lie bialgebra $(\lie g, \lie g^*)$ of the 
Poisson Lie group.
This result has been generalized to \emph{Dirac homogeneous spaces} of \emph{Dirac Lie groups}
in \cite{Jotz11a}, and to Poisson homogeneous spaces of Poisson groupoids in
\cite{LiWeXu97}. Poisson homogeneous spaces of a Poisson groupoid
are in one-one correspondence with 
a special class of Dirac structures in the Courant algebroid $AG\oplus A^*G$.

Here, we prove a classification 
theorem for Dirac homogeneous spaces of Dirac groupoids, 
that generalizes 
the theorems in \cite{Drinfeld93}, \cite{LiWeXu97} 
and \cite{Jotz11a}. Dirac homogeneous 
spaces of a Dirac groupoid are related to Dirac structures 
in $\lie B(\mathsf D_G)$ (Theorem \ref{Drinfeld}). 
In the case of almost Dirac structures,  we classify the homogeneous spaces in terms of an action
of the bisections of $G\rr P$ on the vector bundle $\lie B(\mathsf D_G)$. This action is found in Theorem
\ref{action}, and is already interesting independently since it generalizes  the adjoint actions
of a Poisson Lie group on its Lie bialgebra (\cite{Drinfeld93}).

\bigskip

The geometry is more involved in the Lie groupoid setting than in the 
Lie group case, where the Lie bialgebra
of the Dirac Lie group can be defined 
using the theory that is already known 
about Poisson Lie groups:
multiplicative Dirac structures on a Lie group are only a
slight generalization of the graphs of multiplicative 
bivector fields \cite{Ortiz08,Jotz11a}.
Here, we  need  first to construct in 
the first part of the paper the object $\lie B(\mathsf D_G)$
 that will play the
role of the Lie bialgebroid in this more general setting.
Since we find the right object for the classification 
of the homogeneous spaces, 
our classification theorem suggests that a lot of 
information about the  Dirac groupoid is contained in the data 
$(AG, \lie A(\mathsf D_G), I^\tg(\mathsf D_G), \lie B(\mathsf D_G))$.

The natural question that arises is then in what sense these
infinitesimal objects determine 
infinitesimal invariants 
of a multiplicative Dirac structure, and how to reconstruct the Dirac structure 
and its counterpart on $AG$ as in \cite{Ortiz08t}
from these Lie algebroids. 
The case of multiplicative foliations on Lie groupoids, i.e.,
multiplicative Dirac 
structures $D_G=F_G\oplus F_G^\circ$, with $F_G\subseteq TG$ 
an involutive, multiplicative subbundle of $TG$,
has been solved in \cite{JoOr11}. We show there
that the two Lie algebroids and a partial connection
related to the Courant algebroid encode completely
the multiplicative foliation. 

The general case is  the subject of a work in progress with C. Ortiz and T. Drummond \cite{DrJoOr11}.
In  the special
case of Poisson groupoids, we will recover the result of
 \cite{MaXu00}. In the  case of  multiplicative closed $2$-forms,
we will recover  \cite{BuCrWeZh04}, see also \cite{BuCaOr09}, and in the setting 
of a multiplicative foliation, we will find \cite{JoOr11} as a corollary.
Hence, this paper is the first part of a series of articles
showing how to understand all these descriptions of the infinitesimal
data of multiplicative structures, which are by nature very different, in a common framework.

\paragraph{Outline of the paper}
Backgrounds about Lie groupoids and their Lie algebroids 
are recalled in \S \ref{background1}
and generalities about 
Dirac manifolds  are recalled in
\S 
\ref{background2}. 

The definition
of a Dirac groupoid
is given in \S
\ref{def_ex} together with examples.
In \S \ref{induced}, 
we give (under certain hypotheses)
a generalization to this setting of a theorem
in \cite{Weinstein88b}
about the induced Poisson structure
on the units of a Poisson groupoid.
We also explain shortly how  the situation
is  more complicated in the general case than in the group case.

In \S \ref{units}, we study the 
set $\lie A(\mathsf D_G)$ of units of the Dirac structure, seen as a subgroupoid
of $(TG\oplus T^*G)\,\rr \,(TP\oplus A^*G)$.
We show that  there is a Lie algebroid structure on this vector bundle
over $P$. The Lie algebroid structure on $I^\tg(\mathsf D_G)$ is then a
consequence  and 
we show in \S \ref{int_crit} that 
the integrability properties  of the Dirac structure are completely encoded in these two algebroids.
Then, in \S \ref{Courant}, we 
define a vector bundle over $P$ that is associated to the 
Dirac structure $\mathsf D_G$.
We prove  the existence of a natural Courant algebroid structure 
on  this vector bundle. We compute 
this Courant bracket in three standard examples. In the case of 
a Poisson groupoid, we recover the Courant algebroid structure
on $AG\oplus A^*G$, and in the case of a Lie groupoid endowed with a 
closed multiplicative $2$-form,
we find simply the standard Courant bracket on $TP\oplus T^*P$.
In \S \ref{subsec_action}, we prove 
that there is an induced action of the bisections 
of $G\rr P$ on the vector bundle defined in \S \ref{Courant}.
In Section \ref{section_geometry_dirac}, each one of the main results
is illustrated by the three special examples 
of Poisson groupoids, 
 multiplicative closed $2$-forms 
and pair Dirac groupoids.

Dirac homogeneous spaces of Dirac groupoids are defined in 
Section \ref{homogeneous}. 
For this, we use the fact that if
$X\to P$ is a homogeneous space of a Lie groupoid 
$G\rr P$, the action $\phi:G\times_P X\to X$ of
$G\rr P$ on $J:X\to P$ induces
an action of $(TG\oplus T^*G)\rr (TP\oplus A^*G)$ on
some momentum map $(TX\oplus T^*X)\to  (TP\oplus A^*G)$ (this is proved in \cite{thesis}). 
The pair $(X,\mathsf D_X)$ is then defined to be $(G\rr P, \mathsf D_G)$-homogeneous
if this action restricts to an action of $\mathsf D_G\rr \,\lie A(\mathsf D_G)$
on $\mathsf D_X$.  
Our main theorem (Theorem \ref{Drinfeld})
about the correspondence between (almost) Dirac homogeneous
spaces of an (almost) Dirac groupoid  and Lagrangian subspaces (subalgebroids)
of the Courant algebroid (vector bundle) $\lie B(\mathsf D_G)$ is then proved.

\paragraph{Notations and conventions}
Let $M$ be a smooth manifold. We will denote by $\mx(M)$ and $\Omega^1(M)$ the
spaces of (local) smooth sections of the tangent and the cotangent bundle,
respectively. For an arbitrary vector bundle $\mathsf E\to M$, the space of
(local) sections of $\mathsf E$ will be written $\Gamma(\mathsf E)$. 
We will write $\dom(\sigma)$ for the open subset of the smooth manifold $M$
where the local section $\sigma\in\Gamma(\mathsf E)$ is defined.

The \emph{Pontryagin bundle} of $M$ is the direct sum $TM\oplus T^*M\to M$.
The zero section in $TM$ will be considered as a trivial vector bundle
over $M$ and  written $0_M$, and the zero section in $T^*M$
will be written $0_M^*$. The pullback or restriction of a vector bundle 
$\mathsf E\to M$ to an embedded
 submanifold $N$ of $M$  will be written $\mathsf E\an{N}$. In the special
case of the tangent and cotangent spaces of $M$, we will write $T_NM$ and $T^*_NM$.
The annihilator in $T^*M$ of a smooth subbundle $F\subseteq TM$ will be written 
$F^\circ\subseteq T^*M$.

\paragraph{Acknowledgements}
The author would like to thank Cristian Ortiz
for many helpful comments and suggestions that have improved the paper.
Many thanks go also to Tudor Ratiu for his good advice.

\section{Review of necessary backgrounds}

\subsection{Generalities on Lie groupoids and Lie algebroids}\label{background1}
The general theory of Lie groupoids and their Lie algebroids can be found 
in \cite{Mackenzie05}, \cite{MoMr03}. We fix here first of all some notations and conventions.

A groupoid $G$ with base $P$ will be written  $G\rr P$. 
The set $P$ will be considered  most of the time as a subset of $G$, that is, the unity
$1_p$ will be identified 
 with $p$ for all $p\in P$.
A \emph{Lie groupoid}
is a groupoid $G$ on base $P$ together with the structures of \emph{smooth Hausdorff manifolds}
on $G$ and $P$ such that
the \emph{source and target maps} $\s,\tg:G\to P$ are \emph{surjective submersions}, and such that
the object inclusion 
map $\epsilon:P\hookrightarrow G$, $p\mapsto 1_p$  and the partial multiplication
$\mathsf m:G\times_PG:=\{(g,h)\in G\times G\mid \s(g)=\tg(h)\}\to G$ are all smooth.

Since $\tg$ and $\s$ are smooth 
surjective submersions, the kernels
$\ker(T\tg)$ and $\ker(T\s)$ are smooth subbundles of 
$TG$. These two vector bundles over $G$ are written
 $T^\s G:=\ker(T\s)$ and $T^\tg G:=\ker(T\tg)$.

\medskip

\bigskip
Let $g\in G$, then   the \emph{right translation} by $g$ is
\[R_g:\s\inv(\tg(g))\to\s\inv(\s(g)), \qquad h \mapsto R_g(h)=h\star g. 
\]
The  \emph{left translation} by $L_g:\tg\inv(\s(g))\to\tg\inv(\tg(g))$  is defined in an analogous manner.

Let $G\rr P$ be a Lie groupoid. A \emph{right translation on $G$} is a pair of
diffeomorphisms
$\Phi:G\to G$, $\phi:P\to P$ such that $\s\circ
\Phi=\phi\circ\s$, $\tg\circ\Phi=\tg$
and, for all $p\in P$, the map
 $\Phi\an{\s\inv(p)}:\s\inv(p)\to \s\inv(\phi(p))$ is $R_g$ for some $g\in G$.  
A \emph{bisection} of $G\rr P$ is a smooth map $K:P\to G$ which is right-inverse 
to $\tg:G\to P$ and is such that $\s\circ K$
diffeomorphism. The set of bisections of $G$ is denoted by $\mathcal B(G)$.
If $K:P\to G$ is a bisection of $G\rr P$, 
then  the \emph{right translations
by $K$} is a right translation:
\[R_K:G \to G, \qquad g \mapsto R_{K(\s(g))}(g)=g\star K(\s(g)). 
\]
We will also use the \emph{left translation by $K$},
\[L_K:G\to G, \qquad g\mapsto L_{K\left((s\circ K)\inv(\tg(g))\right)}(g).
\]

The set $\mathcal B(G)$ of bisections of $G$ has the structure of a group. For 
$K, L\in\mathcal B(G)$, the product $L\star K$ is given by
$$L\star K:P\to G, 
\quad (L\star K)(p)=L(p)\star K((\s\circ L)(p))\quad \forall p\in P.$$
The composition 
$\tg\circ (L\star K)$ is equal to $\Id_P$ and the composition
$\s\circ (L\star K)$ is equal to $(\s\circ K)\circ (\s\circ L)$, 
which is a diffeomorphism of $P$.
The identity element in $\mathcal B(G)$ 
is the identity section $\epsilon: P\hookrightarrow G$.
The inverse $K\inv:P\to G$ of $K\in\mathcal B(G)$ is given 
by $$K\inv(p)=\left( K\left((\s\circ K)\inv(p)\right)\right)\inv $$
for all $p\in P$.
Finally, we have $R_{L\star K}=R_K\circ R_L$ 
for all $K,L\in\mathcal B(G)$.
Since $R_\epsilon=\Id_G$, we have then also the equality
$R_{K\inv}=R_K\inv$ for all $K\in \mathcal B(G)$.

We will also consider \emph{local bisections} of $G$ in the following,
without saying it always explicitly. A local bisection
of $G\rr P$ is a map $K:U\to G$ defined on an open set $U\subseteq P$ such that
$\tg\circ K=\Id_U$ and $\s\circ K$ is a diffeomorphism on its image. We will write 
$\mathcal B_U(G)$ for the set of local bisections of 
$G\rr P$ with the domain of definition $U\subseteq P$. The local
right translation induced by the local bisection $K: U\to G$ is the map
$R_K:\s\inv(U)\to \s\inv((\s\circ K)(U))$, $g\mapsto g\star K(\s(g))$.

\paragraph{The Lie algebroid of a Lie groupoid}
In this paper, the Lie algebroid of the Lie groupoid $G\rr P$ 
is $AG:=T^\tg_PM$, equipped with the anchor map
$T\s\an{AG}$ and the Lie bracket defined by the left invariant vector fields.
We write $(AG, \mathsf a, [\cdot\,,\cdot]_{AG})$ for the \emph{Lie algebroid
of the Lie groupoid $G$}. 

Note that the vector field 
$X^l$, for $X\in\Gamma\left(AG\right)$, satisfies 
$X^l\sim_\s \mathsf a(X)\in\mx(P)$ since 
we have 
$T_g\s X^l(g)=T_g\s(T_{\s(g)}L_g X(\s(g)))=T_{\s(g)}\s X(\s(g))$ for all $g\in G$.

\bigskip

We recall here the definition of the \emph{exponential map for a Lie groupoid}, see
\cite{Mackenzie05} (note that the Lie algebroid is defined there with the 
right-invariant vector fields, and the bisections of $G\rr P$ are defined 
to be  inverse to the source map).

Let $G\rr P$ be a Lie groupoid and choose $X\in\Gamma(AG)$. Let 
$\{\phi_t^X:U\to U_t\}$ be a local flow 
for $X^l\in\mx(G)$. Since $T_g\tg X^l(g)=0$ for all $g\in G$, we have 
$(\tg\circ\phi^X_t)(g)=\tg(g)$ for all $t\in \R$ and $g\in G$ where this makes sense.
For each $t\in \R$ where this is defined and $p\in P$, the map
$\phi_t^X$ restricts to $\phi_t^X:\tg\inv(p)\to \tg\inv(p)$. Choose
$h\in G$ such that $\s(h)=p$. We have then 
$L_h:\tg\inv(p)\to \tg\inv(\tg(h))$ and $L_h\circ \phi_t^X=\phi_t^X\circ L_h$
since the vector field $X^l$ satisfies $X^l(h\star g)=T_gL_hX^l(g)$ for all
$g\in\tg\inv(p)$. Recall that  $\bar X:=\mathsf a(X)\in\Gamma(P)$
is defined on  $\s(U):=V\subseteq P$ and is such that
$X^l\sim_\s\bar X$. Let $\bar\phi^X$ be the flow of $\bar X$. Then we have 
$\{\bar\phi^X_t:V\to V_t\}$, where  $V_t=\s(U_t)$, and 
$\bar\phi^X_t\circ \s=\s\circ\phi^X_t$ for all $t$ where this makes sense.

Each $\phi^X_t$ is the restriction to $U$ of a unique local right translation 
$R_{\Exp(tX)}$ with $\Exp(tX)\in\mathcal B_V(G)$.
The local bisection $\Exp(tX)$ is defined 
by $\Exp(tX)(p)=g\inv\star\phi^X_t(g)$ for any $g\in U\cap\s\inv(p)$.
We have then $\tg\circ\Exp(tX)=\Id_V$
and $\s\circ\Exp(tX)=\bar\phi^X_t$  is a local diffeomorphism on its image 
$V_t$.
For any $g\in U$, we have 
$$\phi^X_t(g)=g\star \Exp(tX)(\s(g))=R_{\Exp(tX)}(g)$$
and the flow of $X^l$ is hence the right translation 
by $\Exp(\cdot X)$.
This is summarized in Proposition 3.6.1 of \cite{Mackenzie05}:
\begin{proposition}\cite{Mackenzie05}\label{exp_map}
Let $G\rr P$ be a Lie groupoid, choose  
$X\in\Gamma(AG)$ and set $W=\dom(X)$.
For all $p\in W$ there exists an open neighborhood 
$U $ of $p$ in $W$, a flow neighborhood for $X$, an $\varepsilon>0$ and 
a unique smooth family of local bisections $\Exp(tX)\in\mathcal B_U(G)$,
$|t|<\varepsilon$, such that:
\begin{enumerate}
\item $\left.\frac{d}{dt}\right\an{t=0}\Exp(tX)=X$,
\item $\Exp(0X)=\Id_U$,
\item $\Exp((t+s)X)=\Exp(tX)\star\Exp(sX)$, if $|t|,|s|,|s+t|<\varepsilon$,
\item $\Exp(-tX)=(\Exp(tX))\inv$,
\item $\{\s\circ\Exp(tX):U\to U_t\}$
is a local $1$-parameter group of transformations for $\mathsf a(X)\in\mx(P)$.
\end{enumerate} 
\end{proposition}

Let  $G\rr P$ be a Lie groupoid and let
$C_p$ be the connectedness component of $p$ in $\tg\inv(p)$.
Then the union
$$C(G):=\bigcup_{p\in P}C_p$$
is a wide Lie subgroupoid of $G\rr P$ (see \cite{Mackenzie05}), the 
\emph{identity-component subgroupoid} of $G\rr P$.
The set of values $\Exp(tX)(p)$, for all $X\in\Gamma(AG)$, $p\in P$  and $t\in
\R$ where this makes sense, is the identity-component subgroupoid $C(G)$
of $G\rr P$
(see \cite{MaXu00},\cite{Mackenzie05}). Hence, if $G\rr P$ is $\tg$-connected, that is, if all 
the $\tg$-fibers of $G$ are connected, then 
$G=C(G)$ is the set of values of $\Exp(tX)(p)$,  $X\in\Gamma(AG)$, $p\in P$  and $t\in
\R$ where defined.

\medskip

Note that we can show in the same manner 
that the flow of a right invariant vector field $Y^r$ is the left translation
by a family of bisections $\{L_t\}$ of $G$ satisfying $\s\circ L_t=\Id$
on their domains of definition and 
such that $\tg\circ L_t$ are diffeomorphisms on their images. Hence, the 
flow of $Y^r$ commutes with the flow of $X^l$ for any 
left invariant vector field $X^l$
and we get the fact that $\left[Y^r, X^l\right]=0$ for all $Y\in\Gamma(T^\s_PG)$ and 
$X\in \Gamma(AG)$. 

\paragraph{The tangent prolongation of a Lie groupoid}

Let $G\rr P$ be a Lie groupoid. Applying the tangent functor to each of the
maps defining $G$ yields a Lie groupoid structure on $TG$ with base $TP$,
source $T\s$, target $T\tg$ (these maps will be written $\s$ and $\tg$ in
the following) and multiplication $T\m:T(G\times _PG)\to TG$.
The identity at $v_p\in T_pP$ is $1_{v_p}=T_p\epsilon v_p$.
This defines  the \emph{tangent prolongation of $G\rr P$}
or the \emph{tangent groupoid associated to $G\rr P$}.

\paragraph{The cotangent Lie groupoid defined by a Lie groupoid}
If $G\rr P$ is a Lie groupoid, then there is an induced 
Lie groupoid structure on $T^*G\rr\, A^*G= (TP)^\circ$. The source map
$\hat\s:T^*G\to A^*G$ is given by 
\[\hat\s(\alpha_g)\in A_{\s(g)}^*G \text{ for } \alpha_g\in T_g^*G,\qquad 
\hat\s(\alpha_g)(u_{\s(g)})=\alpha_g(T_{\s(g)}L_gu_{\s(g)})\]
for all $u_{\s(g)}\in A_{\s(g)}G$, 
and the target map $\hat\tg:T^*G\to A^*G$
is given by \[\hat\tg(\alpha_g)\in A_{\tg(g)}^*G,
\qquad \hat\tg(\alpha_g)(u_{\tg(g)})
=\alpha_g\bigl(T_{\tg(g)}R_g(u_{\tg(g)}-T_{\tg(g)}\s u_{\tg(g)})\bigr)\]
for all $u_{\tg(g)}\in A_{\tg(g)}G$.
If $\hat\s(\alpha_g)=\hat\tg(\alpha_h)$, then the product $\alpha_g\star\alpha_h$
is defined
by 
\[(\alpha_g\star\alpha_h)(v_g\star v_h)=\alpha_g(v_g)+\alpha_h(v_h)
\]
for all composable pairs $(v_g,v_h)\in T_{(g,h)}(G\times_P G)$.

This Lie groupoid structure was introduced in \cite{CoDaWe87}
and is explained in \cite{CoDaWe87}, 
\cite{Pradines88} and \cite{Mackenzie05}.

\paragraph{The Pontryagin groupoid defined by a Lie groupoid}
If $G\rr P$ is a Lie groupoid, there is hence 
an induced Lie groupoid structure on $\mathsf P_G=TG\oplus T^*G$
over $TP\oplus A^*G$. 
We will write $\mathbb T\tg$ for the target map $\mathsf P_G\to TP\oplus A^*G$,
and in the same manner 
$\mathbb T\s:\mathsf P_G\to TP\oplus A^*G$ for the source map.

\paragraph{Lie groupoid actions}
Let $G\rr P$ be a Lie groupoid and $M$ a set
with a map $J:M\to P$.
Consider the set $G\times_PM=\{(g,m)\in G\times M\mid \s(g)=J(m)\}$.

A groupoid action of $G\rr P$ on $J:M\to P$
is a map
$\Phi:G\times_P M\to M$, $\Phi(g,m)=g\cdot m=gm$
such that
\begin{itemize}
\item $J(g\cdot m)=\tg(g)$ for all $(g,m)\in G\times_P M$,
\item $g\cdot (h\cdot m)=(g\star h)\cdot m$
for all $(h,m)\in G\times_P M$, and $g\in G$ such that 
$\s(g)=\tg(h)$,
\item $1_{J(m)}\cdot m=m$ for all $m\in M$.
\end{itemize}

\begin{example}\label{extreme_actions}
Let $G\rr P$ be a groupoid.
\begin{enumerate}
\item $G\rr P$ acts  on $\tg:G\to P$ via the multiplication.
\item $G\rr P$ acts on $\Id_P:P\to P$ via 
$\Phi: G\times_P P\to P$, $(g,p)\mapsto \tg(g\star p)=\tg(g)$.
\end{enumerate} 
\end{example}

\paragraph{Homogeneous spaces}
Let $G\rr P$ be a Lie groupoid and $H\rr P$ a wide 
subgroupoid of $G$.
Define the equivalence relation 
\[g\sim_H g'\iff \exists\ h\in H \text{ such that } g\star h=g'\]
on $G$ and $G/H:=G/\sim_H=\{gH\mid g\in G\}$,
where $gH=\{g\star h\mid \s(g)=\tg(h) \text{ and } h\in H\}$.
Since $\tg(g\star h)=\tg(g)$
for all $g\star h\in gH$, the map $\tg$
factors to a map $J:G/H\to P$, 
$J(gH)=\tg(g)$ for all $gH\in G/H$. The multiplication $\mathsf m:G\times_P G\to G$
factors to a groupoid action $\Phi$ of $G\rr P$ on $J:G/H\to P$, 
$\Phi(g,g'H)=(g\star g')H$ for all 
$(g,g'H)\in G\times_P(G/H)=\{(g,g'H)\mid \s(g)=J(g'H)=\tg(g')\}$.

\begin{definition}
A $G$-space $X$ over $P$ is \emph{homogeneous} if there is
a section $\sigma$ of the moment map $J:X\to P$
which is \emph{saturating} for the action  in the sense that $G\star \sigma(P)=X$.
The \emph{isotropy subgroupoid} of the section  $\sigma$
consists of those $g\in G$ for which $g\star \sigma(P)\subseteq \sigma(P)$.
\end{definition}
It is shown in \cite{LiWeXu98} that a $G$-space is homogeneous
if and only if it is isomorphic to $G/H$ for some wide subgroupoid $H\subseteq G$.


\begin{example}\label{extreme_homogeneous}
Let $G\rr P$ be a groupoid. The two extreme examples of homogeneous spaces 
of $G$ are the following. 
\begin{enumerate}
\item In the case where the wide subgroupoid is $P$, 
the equivalence classes are $gP=\{g\star p\mid p\in P, p=\s(g)\}=\{g\}$
and the quotient is just $G/P=G$ with the first action of  Example \ref{extreme_actions}.
\item If the wide subgroupoid is $G$ itself, then the 
equivalence classes are $gG=\{g\star h\mid h\in G, \tg(h)=\s(g)\}=\tg\inv(\tg(g))$
and the quotient is $G/G=P$, with projection 
equal to the target map $\tg: G\to G/G\simeq P$ and with the second action in Example
\ref{extreme_actions}.
\end{enumerate}
\end{example}

Assume that $H$ \emph{is a $\tg$-connected wide Lie subgroupoid of $G$} and that $G/H$
is a smooth manifold such that the projection $q:G\to G/H$
is a smooth surjective submersion.

Consider the vector bundle  $AH=T^\tg_P H\subseteq T_PH\subseteq T_PG$
over $P$ 
and the subbundle $\mathcal H\subseteq TG$ defined 
as the left invariant image of $AH$, i.e., $\mathcal H(g)=T_{\s(g)}L_gA_{\s(g)}H$
for all $g\in G$. We show that $\mathcal H=\ker Tq$ and 
$G/H$ is the leaf space of the foliation on $G$ defined by the 
\emph{involutive subbundle } $\mathcal H \subseteq TG$.

The vector bundle $\mathcal H$ is spanned by the left invariant vector fields
$X^l$, for $X\in\Gamma\left(AH\right)\subseteq \Gamma(AG)$.
Since $H$ is an immersed submanifold of $G$, $AH$
is a subalgebra of $AG$, and $\Gamma(\mathcal H)$
is hence closed under the Lie bracket.

It is easy to check that $\mathcal H=\ker(Tq)$.
Hence, if $g$ and $g'$ are in the same leaf of $\mathcal H$, we have $g'H=gH$. Conversely, if 
$g'H=gH$, it is easy to show, using the fact that 
$H$ is $\tg$-connected,
and hence $H=C(H)=\{\Exp(tX)\mid t\in \R,
X\in\Gamma(AH)\}$ that $g$ and $g'$ are in the same leaf of $\mathcal H$.

Consider the set $\mathcal B(H)$ of (local) bisections $K:U\subseteq P\to H$
of $H$ such that $\tg\circ K=\Id_U$ and $\s\circ K$ is a diffeomorphism.
We have $gH=\{R_K(g)\mid K\in\mathcal B(H)\}$
and $G/H$ is the quotient of $G$ by the right action of $\mathcal B(H)$
on $G$.  A function $f\in C^\infty(G)$ pushes forward to the quotient $G/H$
if and only if it is invariant under $R_K$ for all bisections $K\in\mathcal B(H)$.

\paragraph
{Lie bialgebroids, associated Courant algebroids, the special case of the Pontryagin bundle}
Let $M$ be a smooth manifold and $(A\to M, a, [\cdot\,,\cdot])$
a Lie algebroid on $M$. Assume that the dual
$A^*\to M$ of $A$ is endowed with a Lie algebroid
structure $(A^*\to M, a_*, [\cdot\,,\cdot]_*)$
such that 
\[\dr [\cdot\,,\cdot]_*=[\dr\cdot\,,\cdot]_*+[\cdot\,,\dr\cdot]_*\quad 
\text{ or equivalently 
}\quad \dr_* [\cdot\,,\cdot]=[\dr_*\cdot\,,\cdot]+[\cdot\,,\dr_*\cdot]\]
hold for the induced 
maps $\dr: \Gamma\left(\bigwedge^\bullet A^*\right)\to\Gamma
\left(\bigwedge^{\bullet} A^*\right)$
and $\dr_*: \Gamma\left(\bigwedge^\bullet A\right)\to
\Gamma\left(\bigwedge^{\bullet} A\right)$
and the brackets $[\cdot\,,\cdot]$ (respectively $[\cdot\,,\cdot]_*$)
induced on $\Gamma\left(\bigwedge^\bullet A\right)$ (respectively 
$\Gamma\left(\bigwedge^\bullet A^*\right)$),
see for instance \cite{Kosmann95}. We refer to
 \cite{MaXu94} for a quick review of the definitions
of these objects since they will not be needed explicitly here.
We just recall that if $(A\to M, a, [\cdot\,,\cdot])$
is a Lie algebroid, then the
map $\dr:\Gamma\left(\bigwedge^\bullet A^*\right)\to\Gamma
\left(\bigwedge^{\bullet+1} A^*\right)$
is defined on $C^\infty(M)=\Gamma\left(\bigwedge^0 A^*\right)$
by $(\dr f)(X)=a(X)(f)$ for all $f\in C^\infty(M)$ and $X\in\Gamma(A)$.

For instance, if $(G\rr P,\pi_G)$ is a Poisson  groupoid, the dual
$A^*G$ of the Lie algebroid $AG\to P$ of $G\rr P$ 
inherits the structure of a Lie algebroid such that 
$(AG, A^*G) $ is a Lie bialgebroid (see \cite{MaXu94}, \cite{Mackenzie05}).

The direct sum vector bundle $A\oplus A^*\to P$ endowed with 
the  map $\rho=a\oplus a_*$, the symmetric non degenerate bilinear form 
$\langle\cdot\,,\cdot\rangle$ given by 
$\langle (x_p,\alpha_p), (y_p,\beta_p)\rangle
=\alpha_p(y_p)+\beta_p(x_p)$
for all $(x_p,\alpha_p), (y_p,\beta_p)\in (A\oplus A^*)(p)$
and the bracket on its sections given by
\begin{align*}
&\left[(X,\alpha), (Y,\beta)\right]\\
=&\left([X,Y]+\ldr{\alpha}Y-\ldr{\beta}X-\frac{1}{2}\dr_*(\alpha(Y)-\beta(X)),
[\alpha, \beta]_*+\ldr{X}\beta-\ldr{Y}\alpha+\frac{1}{2}\dr(\alpha(Y)-\beta(X))\right)
\end{align*}
for all $(X,\alpha), (Y,\beta)\in\Gamma(A\oplus A^*)$,
is then a Courant algebroid in the sense of the definition below, 
see \cite{LiWeXu97}.

 A Courant algebroid  over a manifold $M$ 
is a vector bundle $\mathsf E\to  M$ equipped with a 
fiberwise non degenerate symmetric bilinear form 
$\langle\cdot\,,\cdot\rangle$, a skew-symmetric bracket 
$[\cdot\,,\cdot]$ on the smooth sections $\Gamma(\mathsf E)$,
and a vector bundle map  $\rho: \mathsf E\to TM$ 
called the anchor, which satisfy the following conditions 
for all $e_1, e_2, e_3\in\Gamma(\mathsf E)$ and 
$f\in C^\infty(M)$:
\begin{enumerate}
\item $[[e_1, e_2], e_3]] + [[e_2, e_3], e_1] + 
[[e_3, e_1], e_2]=\frac{1}{3}\mathcal D \bigl(\langle[e_1, e_2], e_3\rangle
+\langle[e_2, e_3], e_1\rangle+\langle[e_3, e_1], e_2\rangle
\bigr)$,
\item   $\rho([e_1, e_2]) = [\rho(e_1), \rho(e_2)]$,
\item    $[e_1, f e_2] = f [e_1 , e_2] + (\rho(e_1 )f )e_2
-\langle e_1, e_2\rangle\mathcal D f$,
\item $\rho\circ \mathcal D=0$, i.e., for any $f,g\in C^\infty(M)$,
$\langle \mathcal D f, \mathcal D g\rangle =0$,
\item    $\rho(e_1)\langle e_2, e_3\rangle = 
\left\langle [e_1, e_2]+\mathcal D\langle e_1, e_2\rangle, e_3\right\rangle
+\left\langle e_2, [e_1, e_3]+\mathcal D\langle e_1, e_3\rangle\right\rangle
$,
\end{enumerate}
where $\mathcal D: C^\infty(M)\to \Gamma(\mathsf E)$
is defined by $$\langle \mathcal D f, e\rangle
=\frac{1}{2}\rho(e)(f)
$$
for all $f\in C^\infty(M)$ and  $e\in\Gamma(\mathsf E)$, that 
is,
$ \mathcal D= \frac{1}{2}\beta\inv\circ \rho^*\circ\dr : C^\infty(M)\to
\Gamma(\mathsf E)$. Here, $\beta:\mathsf E\to \mathsf E^*$
is the isomorphism defined by the non degenerate bilinear form
$\langle\cdot\,,\cdot\rangle$. 

\begin{example}\label{ex_pontryagin}
Consider a smooth manifold $M$, the Lie algebroid 
$(TM, [\cdot\,,\cdot], a=\Id_{TM})$
and its dual, the cotangent space  $T^*M$
endowed with the trivial bracket 
$[\cdot\,,\cdot]_*=0$ and the trivial anchor map $a_*=0$.
The map $\dr$ induced 
by $TM$ on the sections 
of $\bigwedge^{\bullet} T^*M$ is here simply the 
usual de Rham derivative.
The map $\dr_*$ induced by 
$(T^*M, 0,0 )$ on the sections of 
$\bigwedge^{\bullet} TM$ 
is trivial. The 
equation 
$\dr[\cdot\,,\cdot]_*=[\dr\cdot\,,\cdot]_*+[\cdot\,,\dr\cdot]_*$
is obviously satisfied
and the pair $(TM, T^*M)$ is a Lie bialgebroid.

\medskip

The direct sum $\mathsf P_M= TM\oplus T^*M$ 
endowed with the projection on $TM$ as anchor map, $\rho=\pr_{TM}$, 
the symmetric bracket 
$\langle\cdot\,,\cdot\rangle$
given by 
\begin{equation}
\langle(v_m,\alpha_m), (w_m,\beta_m)\rangle=\alpha_m(w_m)+\beta_m(v_m)
\label{sym_bracket}
\end{equation}
for all $m\in M$, $v_m,w_m\in T_mM$ and $\alpha_m,\beta_m\in T_m^*M$
and the \emph{Courant bracket} 
given by 
\begin{align}
[(X,\alpha), (Y,\beta)]&=\left([X,Y], \ldr{X}\beta-\ldr{Y}\alpha+\frac{1}{2}
\dr(\alpha(Y)-\beta(X))\right)\label{wrong_bracket}\\
&= \left([X,Y], \ldr{X}\beta-\ip{Y}\dr\alpha-\frac{1}{2}
\dr\langle (X,\alpha), (Y,\beta)\rangle\right)\nonumber
\end{align}
for all $(X,\alpha), (Y, \beta)\in\Gamma(\mathsf P_M)$,
is then a Courant algebroid.
The map $\mathcal D: C^\infty(M)\to \Gamma(\mathsf P_M)$
is given by $\mathcal D f=\frac{1}{2}(0, \dr f)$.
\end{example}

\subsection{Generalities on Dirac structures}\label{background2}
As we have seen in Example \ref{ex_pontryagin}, the \emph{Pontryagin bundle} $\mathsf{P}_M:=TM
\oplus T^* M$ of a smooth manifold $M$ is
endowed with the  non-degenerate symmetric fiberwise bilinear form of signature $(\dim M, \dim M)$ given by
\eqref{sym_bracket}. An
\emph{almost Dirac structure} (see \cite{Courant90a}) on $M $ is a Lagrangian vector 
subbundle $\mathsf{D} \subset \mathsf{P}_M $. That is, $ \mathsf{D}$ coincides with its
orthogonal relative to \eqref{sym_bracket} and so its fibers are necessarily $\dim M $-dimensional.

\medskip

Let $(M,\mathsf D)$ be a Dirac manifold. 
For each $m\in M$, the Dirac structure $\mathsf{D}$ 
defines two subspaces 
$\mathsf{G_0}(m), \mathsf{G_1}(m) \subset T_mM $ by 
\[
\mathsf{G_0}(m):= \{v_m \in T_mM \mid (v_m, 0)\in\mathsf D(m) \}\]
  and
\[\mathsf{G_1}(m):= \left\{v_m \in T_mM \mid \exists\,
\alpha_m \in T_m^*M\; : 
(v_m, \alpha_m) \in \mathsf{D}(m)\right\},\]
 and two subspaces 
$\mathsf{P_0}(m), \mathsf{P_1}(m) \subset T_m^*M$ defined in an analogous manner. 
The distributions $\mathsf{G_0}=\cup_{m\in M}\mathsf{G_0}(m)$
and $\mathsf{P_0}=\cup_{m\in M}\mathsf{P_0}(m)$
are not necessarily smooth. 
The distributions
$\mathsf{G_1}=\cup_{m\in M}\mathsf{G_1}(m)$
(respectively $\mathsf{P_1}=\cup_{m\in M}\mathsf{P_1}(m)$)
are smooth since they are the projections on 
$TM$ (respectively $T^*M$) of $\mathsf D$.

\medskip

The almost Dirac structure $\mathsf D$ is 
a \emph{Dirac structure}
 if 
\begin{equation}\label{closedness}
[ \Gamma(\mathsf{D}), \Gamma(\mathsf{D}) ] \subset \Gamma(\mathsf{D}).
\end{equation}
Since $\left\langle (X, \alpha), (Y, \beta) \right\rangle = 0$ if 
$(X, \alpha), (Y, \beta) \in \Gamma(\mathsf{D})$, this integrability property 
of the Dirac structure 
is  expressed  relative to a non-skew-symmetric bracket that 
differs from 
\eqref{wrong_bracket} by eliminating in the second line the third term of
the second component. This truncated expression 
is called the 
\emph{Courant-Dorfman bracket} in the literature:
\begin{equation}\label{Courant_bracket}
[(X, \alpha), (Y, \beta) ]  
= \left( [X, Y],  \boldsymbol{\pounds}_{X} \beta - \ip{Y} \dr\alpha \right)
\end{equation} 
for all $(X,\alpha), (Y,\beta)\in\Gamma(\mathsf D)$.
The restriction of the Courant bracket to the sections of a Dirac bundle
is skew-symmetric and satisfies the Jacobi identity.
It satisfies also the Leibniz rule:
\begin{equation}\label{leibniz} 
[(X,\alpha), f(Y,\beta)]=f[(X,\alpha), (Y,\beta)]+ X(f)\cdot (Y,\beta)
\end{equation}
for all $(X,\alpha), (Y,\beta)\in\Gamma(\mathsf D)$ and $f\in C^\infty(M)$.

\textbf{Note that in the following, we  
work in the general setting of  almost Dirac structures. To simplify the notation, we will simply
call almost Dirac structures ``Dirac structures'' and always state it explicitely if the integrability condition
\eqref{closedness}
is assumed to be satisfied. We will say in this case that the Dirac 
structure is \emph{closed} or \emph{integrable}.\footnote{We prefer the terminology ``closed'' 
because integrability of a Dirac structure can also significate that it is integrable 
as a Lie algebroid, i.e., it integrates to a presymplectic groupoid as in \cite{BuCr05}.}}
\medskip

The class of Dirac structures given in the next example 
will be very important in the following.
\begin{example}\label{exPoisson}
Let $M$ be a smooth manifold endowed with a globally defined bivector field
$\pi\in\Gamma\left(\bigwedge^2 TM\right)$.  Then the subdistribution $\mathsf
D_\pi\subseteq \mathsf P_M$ defined by
\[\mathsf{D}_\pi(m)=\left\{(\pi^\sharp(\alpha_m),\alpha_m)\mid
  \alpha_m\in T_m^*M \right\}\quad \text{ for all }\, m\in M,
\]
where $\pi^\sharp:T^*M\to TM$ is defined by 
$\pi^\sharp(\alpha)=\pi(\alpha,\cdot)\in\mx(M)$ for all
$\alpha\in\Omega^1(M)$,
is a Dirac structure on $M$. It is closed if and only if the bivector field
satisfies
$[\pi,\pi]=0$, that is, if and only if $(M,\pi)$ is a Poisson manifold.
\end{example}
\paragraph{Dirac maps and Dirac reduction}
Let $(M, \mathsf D_M)$ and $(N, \mathsf D_N)$ be
two Dirac manifolds and $F:M\to N$ a smooth map.
Then $F$ is a \emph{forward Dirac map}
if for all  $n\in N$, $m\in F\inv(n)$ and 
$(v_n, \alpha_n)\in\mathsf D_N(n)$ 
there exists $(v_m, \alpha_m)\in\mathsf D_M(m)$
such that $T_mF v_m=v_n$ and $\alpha_m=(T_mF)^*\alpha_n$.
The map  $F$ is a \emph{backward Dirac map}
if for all  $m\in M$, $n=F(m)$ and 
$(v_m, \alpha_m)\in\mathsf D_M(m)$ 
there exists $(v_n, \alpha_n)\in\mathsf D_N(n)$
such that $T_mF v_m=v_n$ and $\alpha_m=(T_mF)^*\alpha_n$.
\bigskip

Assume that $G\rr P$ is a Lie groupoid, and that 
$G$ is endowed with a Dirac structure. Let $H$ be a $\tg$-connected, wide Lie 
subgroupoid of 
$G$ such that $G/H$ has a smooth manifold structure and 
 $q:G\to G/H$ is a smooth surjective submersion.
Since $G/H$ is the leaf space of $\mathcal H$, where $\mathcal H$ is the left invariant 
image of $AH$
(see  the previous section),
we can apply the results 
 in \cite{Zambon08} (see also \cite{JoRaZa11}) for Dirac reduction. Assume that 
the Dirac structure $\mathsf D$ on $G$ is such that
$\mathsf D\cap (TG\oplus\mathcal H^\circ)$ has constant rank on $G$
and 
\begin{equation}\label{Dirac_red_condition}
[\Gamma(\mathsf D),\Gamma(\mathcal H\oplus\{0\})]
\subseteq \Gamma(\mathsf D+(\mathcal H\oplus\{0\})),
\end{equation}
then $\mathsf D$ induces a Dirac structure $q(\mathsf D)$
on the quotient $G/H$. The Dirac structure
$q(\mathsf D)$ on $G/H$ is given by 
$$\Gamma(q(\mathsf D))=\{(\bar X,\bar \alpha)\in\Gamma(\mathsf P_{G/H})
\mid  \exists X\in\mx(G) \text{ such that }
X\sim_q\bar X \text{ and } (X,q^*\bar\alpha)\in\Gamma(\mathsf D)\}.
$$ 
In other words, $q(\mathsf D)$ is the forward Dirac image
of $\mathsf D$ under $q:G\to G/H$.
If the Dirac structure $\mathsf D$ is closed, then $q(\mathsf D)$ is closed.

If $\mathcal H\oplus\{0\}\subseteq \mathsf D$, 
then $\mathsf D=q^*(q(\mathsf D))$, where for any 
Dirac structure $\bar{\mathsf D}$ on $G/H$, its 
\emph{pullback}
$q^*(\bar{\mathsf D})$ to $G$ is the Dirac structure on $G$
defined 
by
\[q^*(\bar{\mathsf D})(g)=\{(v_g,(T_gq)^*\alpha_{gH})\in \mathsf P_G(g)\mid
(T_gq v_g, \alpha_{gH})\in\bar{\mathsf D}(gH)\}
\]
for all $g\in G$. (The bundle $q^*(\mathsf D)$ is the backward Dirac image of
$\mathsf D$ under $q$.)

Note that if we can verify that
\begin{equation}\label{condition_reduction}
(R_K^*X,R_K^*\alpha)\in\Gamma(\mathsf D)\text{ for all }
(X,\alpha)\in\Gamma(\mathsf D)\text{ and } K\in\mathcal B(H),
\end{equation}
then condition \eqref{Dirac_red_condition} is satisfied.

\section{The geometry of  Dirac groupoids.}\label{section_geometry_dirac}
\subsection{Definition and examples}\label{def_ex}
\begin{definition}[\cite{Ortiz08t}]
A Dirac groupoid is a Lie groupoid $G\rightrightarrows P$
endowed with a Dirac structure $\mathsf D_G$ such that
$\mathsf D_G\subseteq TG\oplus T^*G$ is a Lie subgroupoid.
The Dirac structure $\mathsf D_G$ is then said to be \emph{multiplicative}.
\end{definition}

Note that in \cite{Ortiz08t}, Dirac manifolds are always closed
by definition.
\begin{example}\label{Poisson_groupoid_example}
Consider a Poisson groupoid, that is, a Lie groupoid $G\rr P$ endowed 
with a Poisson structure $\pi_G$ such that 
the graph  $\Gamma\subseteq G\times G\times G$
of the multiplication map is a 
coisotropic submanifold of 
$(G\times G\times G, \pi_G\oplus \pi_G \oplus(-\pi_G))$.
Poisson groupoids were introduced in \cite{Weinstein88b}
and studied in \cite{Weinstein88b}, \cite{Xu95}, \cite{MaXu94}
among other, see also \cite{Mackenzie05}.

It is shown in \cite{MaXu94} that  $(G\rr P,\pi_G)$ 
is a Poisson groupoid if and only if 
the vector bundle map $\pi_G^\sharp:T^*G\to TG$ associated 
to $\pi_G$ is a morphism of Lie groupoids over some map
$a_*:A^*G\to TP$ (the restriction of $\pi_G^\sharp$ to $A^*G$). 
Using this, it is easy to see that 
$(G\rr P, \pi_G)$ is a Poisson groupoid 
if and only if 
$(G\rr P, \mathsf D_{\pi_G})$ is a closed Dirac groupoid
(recall the definition of $ \mathsf D_{\pi_G}$ from Example \ref{exPoisson}).
\end{example}

\begin{example}\label{ex_presymplectic_groupoid}
Let $G\rr P$ be a Lie groupoid.
A $2$-form $\omega_G$ on $G$  is \emph{multiplicative} if 
 the partial multiplication map $\m:G\times_P G\to G$ satisfies 
$\m^*\omega_G=\operatorname{pr}_1^*\omega_G+\operatorname{pr}_2^*\omega_G$.
The 
 graph $\mathsf D_{\omega_G}=\graph(\omega_G^\flat:TG\to T^*G)\subseteq \mathsf P_G$
is then  multiplicative, and $(G\rr P,\mathsf D_{\omega_G})$ is a
Dirac groupoid, see \cite{Ortiz08t}, \cite{BuCaOr09}. If the $2$-form is closed, 
then the 
Dirac groupoid is closed.

Conversely, if a Dirac groupoid $(G\rr P, \mathsf D_G)$ is such that
 $\mathsf{G_1}=TG$, then 
$\mathsf D_G$ is the graph of the vector bundle homomorphism 
$TG\to T^*G$
induced by a multiplicative $2$-form. If the set of smooth sections of 
$\mathsf D_G$ is closed under the Courant bracket, then the $2$-form is closed.

Note that \emph{presymplectic groupoids} have been studied in 
\cite{BuCrWeZh04}, \cite{BuCr05}. These are Lie groupoids endowed with closed, 
multiplicative
$2$-forms satisfying some additional non degeneracy properties that 
will be recalled in Example \ref{ex_presymplectic_groupoid2}.
\end{example}

\begin{example}\label{Dirac_lie_group_example}
Let $(G,\mathsf D_G)$ be a Dirac Lie group in the sense of \cite{Jotz11a}. We
have seen there that it is a Dirac Lie group in the sense of \cite{Ortiz08t},
that is, $\mathsf D_G$ is a subgroupoid of the Pontryagin groupoid $TG\oplus
TG^*\rr \{0\}\oplus\lie g^*$. The set of units is here
$\mathsf D_G(e)\cap( \{0\}\oplus\lie g^*)=\{0\}\oplus\lie p_1$ since 
we know that $\mathsf D_G(e)$ is equal to a direct sum $\lie g_0\oplus\lie
p_1\subseteq \lie g\oplus \lie g^*$, with $\lie g_0$ an ideal in $\lie g$
and $\lie p_1\subseteq \lie g^*$ its annihilator.
\end{example}

\begin{example}\label{ex_pair_Dirac_Lie}
Consider a smooth Dirac manifold $(M,\mathsf D_M)$ and the \emph{pair
Lie groupoid $(M\times M)\rr M$ associated to $M$}, 
that is
$\tg,\s:M\times M\to M$, $\s=\pr_2$, $\tg=\pr_1$ and  $(m,n)\star(n,p)=(m,p)$ for all $m,n,p\in M$.
The unit $1_m$  associated to $m\in M$ is the pair
$(m,m)$ and the set of units $\{1_m\mid m\in M\}$
is equal to $\Delta_M$, which is an embedded submanifold 
of $M\times M$, via the smooth map $\epsilon: M\to \Delta_M$, 
$m\mapsto (m,m)$.
The tangent groupoid $T(M\times M)\,\rr\, TM$ of 
$M\times M\,\rr\, M$ is easily seen to be 
$(TM\times TM)\,\rr\, TM$, the pair groupoid associated to $TM$.

The Lie algebroid $A(M\times M)$ of $M\times M\,\rr\, M$
is the set $T^\tg_{\Delta_M}(M\times M)$. A vector $(v_m,w_m)\in T_{(m,m)}(M\times M)$
lies in $T^\tg_{(m,m)}(M\times M)$ if $0=T_{(m,m)}\tg(v_m,w_m)=v_m$. Hence, we have
$A(M\times M)=(0_M\oplus TM)\an{\Delta_M}$
and its dual
$A^*(M\times M)\simeq (T\Delta_M)^\circ\subseteq T^*(M\times M)\an{\Delta_M}$
is given by
$A^*_{(m,m)}(M\times M)=\{(-\alpha_m,\alpha_m)\mid \alpha_m\in T^*_mM\}$
for all $m\in M$.
Hence, we can give  the structure of the cotangent groupoid 
$T^*(M\times M)\,\rr\, A^*(M\times M)$.
If $(\alpha_m,\alpha_n)\in T_{(m,n)}^*(M\times M)$, then it is easy to check that
$\tg(\alpha_m,\alpha_n)\in (\{0_m\}\times T_mM)^*=A^*_{(m,m)}(M\times M)$,
\begin{align*}
\tg(\alpha_m,\alpha_n)(0_m,v_m)&=-\alpha_m(v_m)
\end{align*}
for all $v_m\in T_mM$,
and hence $\tg(\alpha_m,\alpha_n)=(\alpha_m,-\alpha_m)$.
In the same manner, we show that $\s(\alpha_m,\alpha_n)=(-\alpha_n,\alpha_n)$.
The product of  $(\alpha_m,\alpha_n)$
and $(-\alpha_n,\alpha_p)$ is then given by
$$(\alpha_m,\alpha_n)\star(-\alpha_n,\alpha_p)=(\alpha_m,\alpha_p).$$

It is easy to check that the Dirac structure 
$\mathsf D_M\ominus\mathsf D_M$, defined by 
\[(\mathsf D_M\ominus\mathsf D_M)(m,n)
=\left\{((v_m,-v_n),(\alpha_m,\alpha_n))\in \mathsf P_{M\times M}(m,n)\left|
\begin{array}{c} (v_m,\alpha_m)\in\mathsf D_M(m)\\
\text{ and }
(v_n,\alpha_n)\in\mathsf D_M(n)
\end{array}\right.\right\}\]
for all $(m,n)\in M\times M$, is a multiplicative Dirac structure on $M\times M\,\rr\, M$.
This generalizes the fact that if $(M,\pi_M)$ is a Poisson manifold, then 
$M\times M\,\rr\, M$ endowed with $\pi_M\oplus(-\pi_M)$ is
a Poisson groupoid.

\medskip

We call the Dirac groupoid $(M\times M\,\rr\, M, \mathsf D_M\ominus\mathsf D_M)$
the \emph{pair Dirac groupoid} associated to $(M,\mathsf D_M)$.
It is closed   if and only if
 $(M,\mathsf D_M)$ is closed.
\end{example}

\begin{remark}
In the Poisson case, it is known by results in \cite{Weinstein88b} that 
any multiplicative Poisson structure on a pair groupoid is $\pi_M\oplus(-\pi_M)$
for some Poisson bivector $\pi_M$ on $M$.
This is not true in general. For instance, let $M$ be a smooth manifold 
with a smooth free action of a Lie group $H$ with Lie algebra $\lie h$.
Then the diagonal action
of $H$ on $M\times M$ is by Lie groupoid morphisms, and its vertical 
space $\V\subseteq T(M\times M)$,
$\V(m,n)=\{(\xi_M(m),\xi_M(n))\mid \xi\in\lie h\}$
for all $m,n\in M$,  is multiplicative (see for instance \cite{Jotz11b}).
The Dirac structure $\V\oplus\V^\circ$ is then multiplicative,
but cannot be written as a pair Dirac structure on $M\times M$.
\end{remark}

\subsection{General properties of Dirac groupoids}\label{induced}
First, we study the characteristic distribution of an arbitrary Dirac groupoid.
The results here illustrate how the situation
in the case of Dirac  groupoids is different 
from the case of Dirac Lie groups.

The proof of the first proposition is straightforward.
\begin{proposition}\label{G_0_multiplicative}
Let $(G\rr P,\mathsf D_G)$ be a Dirac groupoid. Then the
subbundle $\mathsf{G_0}\subseteq TG$ is a (set) subgroupoid over $TP\cap\mathsf{G_0}$. 
\end{proposition}

In the group case, $\mathsf{G_0}$ is automatically the biinvariant image of an ideal in the Lie algebra
\cite{Ortiz08,Jotz11a}, hence involutive of constant rank. 
In general, the distribution $\mathsf{G_0}$ does not even need
to be smooth. 
Since each manifold 
can be seen as a (trivial) groupoid over itself 
(i.e., with $\tg=\s=\Id_M$), any Dirac manifold 
can be seen as a Dirac groupoid, which will, 
in general not satisfy these conditions. 
Thus,  trivial Dirac groupoids and pair Dirac groupoids 
yield already many examples of Dirac groupoids that do not have these properties.

If $\mathsf{G_0}$ associated to a \emph{closed}
Dirac groupoid $(G\rr P, \mathsf D_G)$ is \emph{assumed} to be a vector bundle on $G$, 
then we are in the same situation as in the group case.
Yet, we know by the considerations in \cite{Jotz11b} that, even 
if it is regular, the quotient $G/{\mathsf{G_0}}$ does 
not necessarily inherit a groupoid structure. 
If it does, we have the following result, which is shown in \cite{Jotz11b}.

\begin{theorem}\label{main}
Let $(G\rr P, \mathsf D_G)$ be a closed Dirac groupoid. 
Assume that $\mathsf{G_0}$ is a  subbundle of $TG$ and that it is complete \cite{Jotz11b}.
If the leaf spaces $G/\mathsf{G_0}$ and 
$P/\mathsf{G_0}$
have smooth manifold structures 
such that the projections are submersions,
then there is an induced multiplicative   Poisson  structure
on the Lie groupoid $G/\mathsf{G_0}\rightrightarrows P/\mathsf{G_0}$, such that
the projection $\pr:G\to G/\mathsf{G_0}$ is a forward Dirac map.
\end{theorem}

\begin{remark}
In the Lie group case, the Poisson Lie group $(G/N, q(\mathsf{D}_G))$ associated to a closed Dirac Lie group
$(G, \mathsf D_G)$ satisfying the necessary regularity assumptions 
was also a \emph{Poisson homogeneous space}
of the Dirac Lie group. Here, the Poisson groupoid 
associated to the Dirac groupoid is, 
in general, \emph{not} a Poisson homogeneous space of the Dirac groupoid
since the quotient $G/\mathsf{G_0}$ is not a homogeneous space of the Lie groupoid $G\rr P$.
\end{remark}
\bigskip

For the sake of completeness, we show next
how the result in \cite{Weinstein88b} 
about the induced Poisson structure on the units 
of a Poisson groupoid can be generalized 
to the situation of Dirac groupoids. For that, we need to study the units of the Dirac groupoid.
It is natural to ask what the set of units of $\mathsf D_G$ is, when seen as a
subgroupoid of $(TG\oplus T^*G)\rr(TP\oplus A^*G)$. 
It is easy to see that  $\mathsf D_G$ is a Lie groupoid over 
$\mathsf D_G\cap(TP\oplus A^*G)$.
This intersection will be written $\mathfrak A(\mathsf D_G):=\mathsf D_G\cap(TP\oplus A^*G)$.
 Here, we  will
show  that it is a vector bundle over $P$.

\begin{definition}\label{def_I}
\begin{enumerate}
\item Let $(G\rr P, \mathsf D_G)$ be a  Dirac groupoid
and  $\mathfrak A(\mathsf D_G)$ the set of units of $\mathsf D_G$, i.e., the subdistribution
$\mathsf D_G\cap( TP\oplus A^*G)$ of $ TP\oplus A^*G$.
We write $\mathsf a_\star: \mathfrak A(\mathsf D_G)\to TP$ for the map defined by
$\mathsf a_\star(v_p,\alpha_p)=v_p$ for all $p\in P$, 
$(v_p,\alpha_p)\in \lie A_p(\mathsf D_G)$.
\item We write $\ker\TT\s$, respectively $\ker\TT\tg$  for the kernel
$T^\s G\oplus (T^\tg G)^\circ$ (respectively $T^\tg G\oplus (T^\s G)^\circ$)
of the source  map $\TT\s:\mathsf P_G\to TP\oplus A^*G$ 
(respectively the target map $\TT\tg:\mathsf P_G\to TP\oplus A^*G$).
We denote by $I^\s(\mathsf D_G)$ 
the restriction to $P$ of $\mathsf D_G\cap\ker\TT\s$, i.e.,
 $$I^\s(\mathsf D_G):=\mathsf D_G\cap(T_P^\s G\oplus(T_P^\tg G)^\circ)
=(\mathsf D_G\cap\ker\TT\s)\an{P}.$$
In the same manner, we write
 $I^\tg(\mathsf D_G):=\mathsf D_G\cap(T_P^\tg G\oplus(T_P^\s
G)^\circ)=(\mathsf D_G\cap\ker\TT\tg)\an{P}$.
\end{enumerate}
\end{definition}

\begin{theorem}\label{theorem_intersections}
Let $(G\rr P,\mathsf{D}_G)$ be a Dirac groupoid. Then the 
Dirac subspace $\mathsf D_G\an{P}$ splits as a direct sum
\begin{align*}
\mathsf D_G\an{P}&=\lie A(\mathsf D_G)\oplus  I^\tg(\mathsf D_G)
\end{align*}
and in the same manner
\begin{align*}
\mathsf D_G\an{P}&=\lie A(\mathsf D_G)\oplus I^\s(\mathsf D_G).
\end{align*}
The three intersections are smooth and have constant rank on $P$.
\end{theorem}
\begin{proof}
Choose $p\in P$ and $(v_p, \alpha_p)\in\mathsf D_G(p)$. Then we have 
$\TT\tg(v_p,\alpha_p)\in\mathsf D_G(p)$ and hence also
$(v_p,\alpha_p)-\TT\tg(v_p,\alpha_p)\in\mathsf D_G(p)$.
We find that $v_p-T_p\tg v_p\in T^\tg_p G$ and 
$T_p\tg(v_p)\in T_pP$, and in the same manner
$\hat\tg(\alpha_p)\in A^*_pG=(T_pP)^\circ$, by definition, and 
$\alpha_p-\hat\tg(\alpha_p)\in (T_p^\s G)^\circ$.

Since 
$$(v_p,\alpha_p)=\TT\tg(v_p,\alpha_p)+\left( (v_p,\alpha_p)-\TT\tg(v_p,\alpha_p)\right),$$
we have shown the first equality. The second formula can be shown in the same manner,
using the map $\TT\s:\mathsf D_G(p)\to \mathsf D_G(p)\cap(T_pP\times A^*_pG)$.

\medskip

Next, we show that the intersection of $\mathsf D_G$ with
$TP\oplus A^*G$ is smooth. Choose 
$p\in P$ and $(v_p, \alpha_p)\in \mathsf D_G(p)\cap (T_pP\times A^*_pG)$.
Since $\mathsf D_G$ is a smooth vector bundle
on $G$, we find a section $(X,\alpha)\in\Gamma(\mathsf D_G)$ defined
on a neighborhood of $p$ such that 
$(X,\alpha)(p)=(v_p, \alpha_p)$. The restriction 
$(X,\alpha)\an{P}$ is then a smooth section of $\mathsf D_G\an{P}$.
We have $\TT\s((X,\alpha)\an{P})\in\Gamma(\mathsf D_G\cap(TP\oplus A^*G))$ and 
$\TT\s(X,\alpha)(p)=(T_p\s v_p, \alpha_p\an{T^\tg_pG})=(v_p,\alpha_p)$
since $v_p\in T_pP$ and $\alpha_p\in A^*_pG=(T_pP)^\circ$.

Thus, we have found a smooth section of 
$\mathsf D_G\cap(TP\oplus A^*G)$ defined on a neighborhood of $p$ in $P$
and taking value $(v_p,\alpha_p)$ at $p$.

\medskip

Since  $\left(\mathsf D_G\an{P}\right)^\perp=\mathsf D_G\an{P}$ and 
$TP\oplus A^*G=\left(TP\oplus A^*G\right)^\perp$ are smooth  subbundles 
of $\mathsf P_G\an{P}$, we get from
Proposition 4.4
 in  \cite{JoRaSn11} that 
$\mathsf D_G\cap(TP\oplus A^*G)$ has constant rank on $P$.
By the splittings shown above and the fact that 
$\mathsf D_G\an{P}$ has constant rank on $P$, we find that the two other 
intersections have constant rank on $P$, and are thus smooth.
\end{proof}

In the case of a Dirac Lie group, the bundle $I^\s(\mathsf D_G)\to P$ 
is $\lie g_0\to \{e\}$, as is shown in  the next example. We will see later
that $I^\s(\mathsf D_G)$ has a crucial role in the construction of the Courant algebroid
associated to a Dirac groupoid $(G\rr P,\mathsf D_G)$. The fact that 
the left and right invariant images of this subspace are exactly the characteristic distribution of the 
Dirac structure is a very special and convenient feature in the group case, that makes the Dirac Lie groups
much easier to understand than arbitrary Dirac groupoids (see \cite{Jotz11a}).
\begin{example}
Let $(G,\mathsf D_G)$ be a Dirac Lie group (Example \ref{Dirac_lie_group_example}) and set
$\lie p_1=\mathsf P_1(e)\subseteq \lie g^*$ and 
$\lie g_0=\mathsf G_0(e)\subseteq \lie g$.
We have 
$P=\{e\}$ (the neutral element of $G$), 
$$\mathsf D_G(e)\cap(T_eP\times (T_eP)^\circ)
=\mathsf D_G(e)\cap(\{0\}\times\lie g^*)=\{0\}\times\lie p_1$$
and 
$$\mathsf D_G(e)\cap(T^\s_eG\times (T_e^\tg G)^\circ)
=\mathsf D_G(e)\cap(\lie g\times\{0\})=\lie g_0\times\{0\}.$$
We recover hence the equality $\mathsf D_G(e)=\lie g_0\times\lie p_1$ \cite{Jotz11a}.

In this particular case, $\mathsf D_G$ is a Poisson structure if and only if 
$\mathsf D_G(e)$ is equal to the set of units of $TG\oplus T^*G\rr\{0\}\times\lie g^*$, i.e.
$\lie g_0=\{0\}$ and $\lie p_1=\lie g^*$.
In the general case, this is not true. 
\end{example}

\begin{lemma}\label{const_rk_int_kern}
Let $(G\rr P,\mathsf D_G)$ be a Dirac groupoid.
For all $g\in G$, we have 
$$\mathsf D_G(g)\cap\ker\TT\tg=(0_g,0_g)
\star I_{\s(g)}^\tg(\mathsf D_G)$$
and 
$$\mathsf D_G(g)\cap\ker\TT\s=I_{\tg(g)}^\s(\mathsf D_G)
\star (0_g,0_g).$$
The intersections 
$\mathsf D_G\cap\ker\TT\tg$
and $\mathsf D_G\cap\ker\TT\s$
have consequently constant rank on $G$.
\end{lemma}

\begin{proof}
Choose $g\in G$ and $v_{\s(g)}\in T_{\s(g)}^\tg G$, $\alpha_{\s(g)}\in T_{\s(g)}^*P$
such that $(v_{\s(g)},(T_{\s(g)}\s)^*\alpha_{\s(g)})\in \mathsf D_G(\s(g))$.
Then we have $\TT\tg(u_{\s(g)},(T_{\s(g)}\s)^*\alpha_{\s(g)})=(0_{\s(g)}, 0_{\s(g)})$ 
and $(0_g, 0_g)\in\mathsf D_G(g)$ with
$\TT\s(0_g, 0_g)=(0_{\s(g)}, 0_{\s(g)})$. Thus, the product 
$$(0_g, 0_g)\star (u_{\s(g)},(T_{\s(g)}\s)^*\alpha_{\s(g)})$$
makes sense and is an element of $\mathsf D_G(g)\cap\ker\TT\tg$. 
Conversely, it is easy to see that 
$$ (0_{g\inv}, 0_{g\inv})\star (v_g,\alpha_g)
$$
makes sense and is an element of $\mathsf D_G(\s(g))\cap\ker\TT\tg=I_{\s(g)}^\tg(\mathsf D_G)$
for all
 $(v_g,\alpha_g)\in \mathsf D_G(g)\cap (T_g^\tg G\oplus(T_g^\s G)^\circ)$.
Since $(0_{g\inv},0_{g\inv})=(0_g,0_g)\inv$, 
this shows $(v_g,\alpha_g)\in (0_g,0_g)\star I_{\s(g)}^\tg(\mathsf D_G)$.

\medskip

There is hence an isomorphism 
$$\mathsf D_G(\s(g))\cap\left(T_{\s(g)}^\tg G\times (T_{\s(g)}^\s G)^\circ\right)
\leftrightarrow 
\mathsf D_G(g)\cap\left(T_{g}^\tg G\times (T_{g}^\s G)^\circ\right).$$
As a consequence, $\mathsf D_G\cap\ker\TT\tg$ has constant rank along 
$\s$-fibers. Since $\mathsf D_G\cap\ker\TT\tg$
has constant rank on $P$
by Theorem \ref{theorem_intersections}, it has hence constant rank on the whole of $G$.
\end{proof}

\begin{example}
If $(G\rr P, \pi_G)$ is a Poisson groupoid, 
then $\pi_G^\sharp(\dr(\s^* f))\in\Gamma(T^\tg G)$ for 
all $f\in C^\infty(P)$ (see \cite{Weinstein88b}).
The intersection
$\mathsf D_{\pi_G}\cap\ker\TT\tg$ is hence spanned 
by the sections 
$(\pi_G^\sharp(\dr(\s^* f)), \dr(\s^* f))$,
with $f\in C^\infty(P)$, and has constant rank.
The intersection 
 $\mathsf D_{\pi_G}\cap\ker\TT\s$ is spanned 
by the sections\linebreak
$(\pi_G^\sharp(\dr(\tg^* f)), \dr(\tg^* f))$
with $f\in C^\infty(P)$.
\end{example}

\bigskip

Using this, we will show the next main theorem of this section. We will need the following lemma.
\begin{lemma}\label{inverse}
Let $G\rr P$ be a Lie groupoid.
Choose $g\in G$ and set $p=\tg(g)$. Then, for all  $\alpha_p\in T_p^*P$, we have 
$$-(T_{g\inv}\s)^*\alpha_p
=\left((T_g\tg)^*\alpha_p\right)\inv.$$
\end{lemma}

\begin{proof}
For any $u_p\in A_pG$,  one gets easily
\begin{align*}
\hat\tg((T_g\tg)^*\alpha_p)(u_p)&
=-((T_p\s)^*\alpha_p)(u_p)
\end{align*}
and 
\begin{align*}
\hat\s(-(T_{g\inv}\s)^*\alpha_p)(u_p)&
=-((T_p\s)^*\alpha_p)(u_p).
\end{align*}
In the same manner,
\begin{align*}
\hat\tg(-(T_{g\inv}\s)^*\alpha_p)=0, \qquad 
\hat\s((T_g\tg)^*\alpha_p)=0.
\end{align*}

Hence, we can compute $((T_g\tg)^*\alpha_p)\star(-(T_{g\inv}\s)^*\alpha_p)$
and $(-(T_{g\inv}\s)^*\alpha_p)\star((T_g\tg)^*\alpha_p)$. By choosing
for any $u_q\in T_qG$ two vectors
$u_{g\inv}\in T_{g\inv}G$ and $u_{g}\in T_gG$
such that $u_q=u_{g\inv}\star u_{g}$, one gets immediately
\begin{align*}
\left((-(T_{g\inv}\s)^*\alpha_p)\star((T_g\tg)^*\alpha_p)\right)(u_q)
&=(-(T_{g\inv}\s)^*\alpha_p)(u_{g\inv})+((T_g\tg)^*\alpha_p)(u_g)=0,
\end{align*}
which shows that
$(-(T_{g\inv}\s)^*\alpha_p)\star((T_g\tg)^*\alpha_p)=0_q=\hat\s((T_g\tg)^*\alpha_p)$.
For any $w_p=w_{g}\star w_{g\inv}\in T_pG$, 
\begin{align*}
\left(((T_g\tg)^*\alpha_p)\star(-(T_{g\inv}\s)^*\alpha_p)\right)(w_p)
&=\alpha_p(T_g\tg w_g)-\alpha_p(T_{g\inv}\s w_{g\inv})
=\alpha_p\circ (T_p\tg-T_p\s)(w_p).
\end{align*}
Thus, $((T_g\tg)^*\alpha_p)\star(-(T_{g\inv}\s)^*\alpha_p)
=\hat\tg((T_g\tg)^*\alpha_p)$.
\end{proof}

\begin{remark}\label{computation}
If $(v_p,(T_p\s)^*\alpha_p)$ is such that $T_p\tg v_p=0_p$,
then $\TT\tg(v_p,(T_p\s)^*\alpha_p)=(0_p,0_p)$. If $g\in G$ is such that
$\s(g)=p$, then 
$(0_g, 0_g)\star (v_p,(T_p\s)^*\alpha_p)
=(T_pL_g v_p, (T_g\s)^*\alpha_p)$ for all $g\in\s\inv(p)$.

To see this, let $c:(-\varepsilon, \varepsilon)\to \tg\inv(p)$
be a curve such that
$c(0)=p$ and $\dot c(0)=v_p$.
We can then compute $0_g\star v_p=T_{(g,p)}\m(0_g,v_p)
=\left.\frac{d}{d\sigma}\right\an{\sigma=0}g\star c(\sigma)
=T_{p}L_g v_p$. If $v_g\in T_g G$, the equality
$v_g=v_g\star (T_g\s v_g)$
yields
$\left(0_g\star (T_{p}\s)^*\alpha_{p}\right)(v_g)
=0_g(v_g)+  \left((T_{p}\s)^*\alpha_{p}\right)(T_g\s v_g)
=\alpha_{p}(T_g\s v_g)
=\left((T_g\s)^*\alpha_{p}\right)(v_g)$.
\end{remark}

Now we can prove a generalization of the fact that the units of a Poisson groupoid inherit a Poisson structure 
such that the target map if a Poisson map and the source is anti-Poisson (see \cite{Weinstein88b}).
\begin{theorem}\label{induced_Dirac_on_P}
Assume that $(G\rr P,\mathsf D_G)$ is a  Dirac groupoid such that 
$TP\cap\mathsf{G_0}$ is smooth.
Define the subspace $\mathsf D_P$ of $\mathsf P_P$ by
\begin{equation}\label{Dirac_on_P}
\mathsf D_P(p)=\left\{(v_p,\alpha_p)\in\mathsf P_P(p)\left|\begin{array}{c}
\exists 
(w_p, (T_p\tg)^*\alpha_p)\in \mathsf D_G(p)\cap(T_pG\times (A_pG)^\circ)\\
\text{ such that } v_p=T_p\tg w_p
\end{array}\right.\right\}
\end{equation}
for all $p\in P$.
Then  $\mathsf D_P$ is a  Dirac structure on $P$.

Furthermore, if for all $g\in G$, the restriction to $\mathsf{G_0}(g)$
of the target map
$T_g\tg:\mathsf{G_0}(g)\to \mathsf{G_0}(\tg(g))\cap T_{\tg(g)}P$ 
is surjective, then the maps
 $\tg:(G,\mathsf D_G)\to (P,\mathsf D_P)$ 
  and $\s:(G,\mathsf D_G)\to (P,-\mathsf D_P)$ are
forward Dirac maps, where $-\mathsf D_P$ is the Dirac structure 
defined on $P$ by $\mathsf D_P(p)=\{(-v_p,\alpha_p)\in\mathsf P_P(p)\mid
(v_p,\alpha_p)\in\mathsf D_P(p)\}$.
\end{theorem}
The characteristic distribution $\mathsf{G_0}^P$ of $(P,\mathsf D_P)$ 
is then equal to the intersection $\mathsf{G_0}\cap TP$.
Note that this theorem generalizes Theorem 4.2.3 in \cite{Weinstein88b} (see also
\cite{Weinstein87}, \cite{CoDaWe87} for the special case of symplectic groupoids), since 
in the Poisson case, we have $\mathsf{G_0}=0_{TG}$
and the hypotheses are consequently trivially satisfied.
If all conditions are satisfied, the Dirac structure on $P$ is just the push forward of the Dirac structure
on $G$ under the quotient map $\tg:G\to G/G\simeq P$ (see Example \ref{extreme_homogeneous}).

\begin{proof}
First note that $I^\s(\mathsf D_G)\oplus ((\mathsf{G_0}\cap TP)\oplus0_{A^*G})=
\mathsf D_G\cap (T_PG\oplus AG^\circ)$.
By the hypothesis on $\mathsf{G_0}\cap TP$, the
 intersection $\mathsf D_G\cap (T_PG\oplus AG^\circ)$
is hence smooth and has consequently constant rank on $P$
by a Proposition in \cite{JoRaSn11}.
The space $\mathsf D_P$ is smooth since it is spanned by 
the smooth sections of 
$\mathsf{G_0}\cap TP$ and the smooth sections 
$(T\tg X,\alpha)$ for all $(X^r,\tg^*\alpha)\in\Gamma(\mathsf D_G\cap\ker\TT\s)$.

The inclusion $\mathsf D_P\subseteq \mathsf D_P^\perp$ is obvious.
Conversely, 
if  $(v_p,\alpha_p)\in \mathsf D_P(p)^\perp\subseteq \mathsf P_P(p)$
and $y_p\in T_pG$ is chosen such that $T_p\tg y_p=v_p$,  then we have 
\[ 
\langle (w_p,(T_p\tg)^*\beta_p), (y_p,(T_p\tg)^*\alpha_p\rangle
=\langle (T_p\tg w_p,\beta_p), (v_p,\alpha_p)\rangle=0\]
for all 
$(w_p, (T_p\tg)^*\beta_p)\in \mathsf D_G(p)\cap(T_pG\times A_pG^\circ)$. 
Hence,  we get
\[(y_p, (T_p\tg)^*\alpha_p)\in 
(\mathsf D_G(p)\cap(T_pG\times (A_pG)^\circ))^\perp
=(\mathsf D_G(p)+A_pG\times\{0_p\})\]
and consequently 
$(y_p, (T_p\tg)^*\alpha_p)=(y'_p,(T_p\tg)^*\alpha_p)+(u_p,0)$
for some 
$(y'_p,(T_p\tg)^*\alpha_p)\in\mathsf D_G(p)$ and 
$u_p\in A_pG$. But then 
$T_p\tg y'_p=T_p\tg y_p=v_p$ and 
$(v_p,\alpha_p)\in\mathsf D_P(p)$.

\medskip

Assume that the target map $T_g\tg:\mathsf{G_0}(g)\to \mathsf{G_0}(\tg(g))\cap T_{\tg(g)}P$ 
is surjective for all $g\in G$.
We show  that $\tg:(G,\mathsf D_G)\to (P,\mathsf D_P)$ is a forward Dirac map.
Choose $p\in P$, $g\in \tg\inv(p)$ and 
$(v_p,\alpha_p)\in\mathsf D_P(p)$. We have to prove that there exists 
$(v_g,\alpha_g)\in\mathsf D_G(g)$ such that 
$\alpha_g=(T_g\tg)^*\alpha_p$
and $T_g\tg v_g=v_p$.
By definition of $\mathsf D_P$ and the considerations above, there exists $u_p\in T^\s_pG$ and 
$z_p\in\mathsf{G_0}(p)\cap T_pP$ such that 
$T_p\tg u_p+z_p=v_p$ and $(u_p, (T_p\tg)^*\alpha_p)\in I^\s_p(\mathsf D_G)$.
Then the pair $(T_pR_gu_p, (T_g\tg)^*\alpha_p)
=(u_p, (T_p\tg)^*\alpha_p)\star (0_g,0_g)$
is an element of $\mathsf D_G(g)$ 
and by hypothesis, we find $z_g\in \mathsf{G_0}(g)$ such that
$T_g\tg z_g=z_p$. The pair $(T_pR_gu_p +z_g, (T_g\tg)^*\alpha_p)$
is then an element of $\mathsf D_G(g)$ and 
 $T_g\tg(T_pR_gu_p +z_g)=T_p\tg u_p+z_p=v_p$.

It remains to  
prove that $\s: (G,\mathsf D_G)\to (P,-\mathsf D_P)$ is also a forward Dirac map.
Choose $p\in P$, $g\in\s\inv(p)$ and $(v_p,\alpha_p)\in -\mathsf D
_P(p)$. Then 
$(-v_p,\alpha_p)\in\mathsf D_P(p)$ and, since $\tg(g\inv)=\s(g)=p$,
there exists  by the considerations above
$w_g\in T_{g\inv}G$ such that
 $T_{g\inv}\tg w_{g\inv}=v_p$ and $(-w_{g\inv},(T_{g\inv}\tg)^*\alpha_p)
\in\mathsf D_G(g\inv)$.
But by Lemma \ref{inverse}, we have then 
$((-w_{g\inv})\inv,-(T_{g}\s)^*\alpha_p)$
$\in\mathsf D_G(g)$. This leads to 
$((w_{g\inv})\inv,(T_{g}\s)^*\alpha_p)
\in\mathsf D_G(g)$ and since $T_g\s((w_{g\inv})\inv)$$=T_{g\inv}\tg w_{g\inv}
=v_p$, the proof is finished.
\end{proof}

Note that the hypotheses on the distribution $\mathsf{G_0}$
in Theorem \ref{induced_Dirac_on_P} are rather strong.
The following example shows that this theorem can hold under weaker hypotheses.
\begin{example}
Assume that  $(M,\mathsf D_M)$ is a smooth Dirac manifold such that 
$\mathsf{G_0}$ is a singular distribution. Then, the induced pair Dirac groupoid
$(M\times M\,\rr\, M, \mathsf D_M\ominus\mathsf D_M)$ as in Example
\ref{ex_pair_Dirac_Lie}
does not satisfy the conditions for Theorem \ref{induced_Dirac_on_P}. 
The space $I^\s(\mathsf D_M\ominus\mathsf D_M)$
is here 
given by 
$$I^\s_{(m,m)}(\mathsf D_M\ominus\mathsf D_M)
=\{(v_m,0_m,\alpha_m,0_m)\mid (v_m,\alpha_m)\in\mathsf D_M(m)\}
$$
for all $m\in M$, and the space 
$\mathsf{G_0}\cap T\Delta_M$
is given by
$$\mathsf{G_0}(m,m)\cap T_{(m,m)}\Delta_M
=\{(v_m,v_m)\mid (v_m,0_m)\in\mathsf D_M(m)\}
$$
for all $m\in M$. Hence, we have 
\begin{align*}
&I^\s_{(m,m)}(\mathsf D_M\ominus\mathsf D_M)+(\mathsf{G_0}(m,m)\cap T_{(m,m)}\Delta_M)\times_{M\times M}\{0\}\\
=\,&\{(v_m,w_m,\alpha_m,0_m)\mid (v_m,\alpha_m)\in\mathsf D_M(m),
(w_m, 0_m)\in\mathsf D_M(m)\}
\end{align*}
and we find that 
the same construction as in Theorem \ref{induced_Dirac_on_P}
defines a  Dirac structure on $M\simeq \Delta_M$, which equals the original Dirac structure 
$\mathsf D_M$ on $M$ since its fiber over 
$m\in M$ is 
 $\{(v_m,\alpha_m)\mid (v_m,w_m,\alpha_m,0_m)
\in I^\s_{(m,m)}(\mathsf D_M\ominus\mathsf D_M)
+\mathsf{G_0}(m,m)\cap T_{(m,m)}\Delta_M\}$.
\end{example}

\begin{example}\label{ex_presymplectic_groupoid2}
Let $G\rr P$ be a Lie groupoid and 
 $\omega_G\in\Omega^2(G)$
be a closed $2$-form on $G$.
Then $(G\rr P, \omega_G)$ is a  presymplectic groupoid \cite{BuCrWeZh04} if
$\omega_G$
is multiplicative, 
 $\dim G=2\dim P$ and 
 $(\ker\omega_G)(p)\cap T_p^\tg G\cap T_p^\s G=\{0_p\}$ for all $p\in P$.

Note that $(\ker\omega_G)(p)=\mathsf{G_0}(p)$, if 
$\mathsf{G_0}$ is the characteristic distribution 
associated to the Dirac groupoid 
$(G\rr P, \mathsf D_{\omega_G})$, see Example 
\ref{ex_presymplectic_groupoid}.
In \cite{BuCrWeZh04}, a multiplicative $2$-form is 
said to be \emph{of Dirac type}
if it has a property that is shown to be equivalent to 
 our hypothesis on surjectivity of $T\tg\an{\ker\omega_G}$.
If the bundle of Dirac structures defined as in 
\eqref{Dirac_on_P} by the multiplicative 
Dirac structure $\mathsf D_{\omega_G}$ associated to a multiplicative
$2$-form $\omega_G$ of Dirac type is smooth, 
then  $\mathsf D_P$ is a Dirac structure on $P$ such that
the target map $\mathsf t: (G, \mathsf D_G)\to (P,\mathsf D_P)$ is a forward Dirac map.
Thus, we recover here 
the two conditions in \cite{BuCrWeZh04} since we made the hypothesis 
on smoothness of $\mathsf{G_0}\cap TP$ to ensure the smoothness 
of $\mathsf D_P$.

It is shown in \cite{BuCrWeZh04} that a presymplectic groupoid
 $(G\rr P, \omega_G)$  satisfy automatically these  conditions and hence  that
 there exists 
a  Dirac structure $\mathsf D_P$ on $P$ such that
the target map 
$\mathsf t: (G, \mathsf D_{\omega_G})\to (P,\mathsf D_P)$ is a forward Dirac map.
\end{example}

\begin{remark}
In the situation of Theorem \ref{main},
the multiplicative 
subbundle $\mathsf{G_0}$  of $TG$ has constant rank on $G$. In particular, 
the intersection $TP\cap\mathsf{G_0}$ is a smooth vector bundle over $P$
and for each $g\in G$, the restriction 
to $\mathsf{G_0}(g)$
of the target map, $T_g\tg:\mathsf{G_0}(g)\to \mathsf{G_0}(\tg(g))\cap T_{\tg(g)}P$,
is surjective \cite{Jotz11b}. 
By Theorem \ref{induced_Dirac_on_P}, there exists
then a Dirac structure $\mathsf D_P$ on $P$ such that 
$\tg: (G, \mathsf D_G)\to (P,\mathsf D_P)$
is a forward Dirac map.
Since $(G/\mathsf{G_0}\rr P/\mathsf{G_0}, \pr(\mathsf D_G))$
is a Poisson groupoid,  we know also
by a theorem in \cite{Weinstein88b} that there is   
a Poisson structure $\{\cdot\,,\cdot\}_{P/\mathsf{G_0}}$ on $P/\mathsf{G_0}$ such that
$[\tg]:(G/\mathsf{G_0}, \pr(\mathsf D_G))\to (P/\mathsf{G_0},\{\cdot\,,\cdot\}_{P/\mathsf{G_0}})$
is a Poisson map.
It is easy to check that
the map $\pr_\circ:(P,\mathsf D_P)\to (P/\mathsf{G_0},\{\cdot\,,\cdot\}_{P/\mathsf{G_0}})$
is then also a forward Dirac map, i.e.,
the graph of the vector bundle homomorphism $T^*(P/\mathsf{G_0})\to T(P/\mathsf{G_0})$
defined by the Poisson structure is the forward Dirac image 
of $\mathsf D_P$ under $\pr_\circ$.  
\begin{align*}
\begin{xy}
\xymatrix{
(G,\mathsf D_G)\ar[r]^\tg\ar[d]_{\pr}&(P,\mathsf D_P)\ar[d]^{\pr_\circ}\\
(G/\mathsf{G_0},\pr(D_G))\ar[r]_{[\mathsf \tg]}&(P/\mathsf{G_0}, \{\cdot\,,\cdot\}_{P/\mathsf{G_0}})
}
\end{xy}
\end{align*}
\end{remark}

\subsection{The units of a Dirac groupoid}\label{units}
In this section, we 
discuss further properties of the set of units $\lie A(\mathsf D_G)$
of a multiplicative Dirac structure.

\begin{proposition}\label{sdescending}
Let $\bar\xi=(\bar{X_\xi},\bar{\theta_\xi})$ be a section of 
$\mathsf D_G\cap (TP\oplus A^*G)=\mathfrak A(\mathsf D_G)$.
Then there exists a smooth section $\xi=(X_\xi,\theta_\xi)$ of $\mathsf D_G$
such that 
$\xi\an{P}=\bar\xi$ and 
$\TT\s(\xi(g))=\bar\xi(\s(g))$ for all
$g\in \s\inv\left(\dom\left(\bar\xi\right)\right)$.
\end{proposition}

Following \cite{Mackenzie00}, 
we say that then that the section $\xi$ of $\mathsf D_G$ is a 
\emph{$\s$-star section} or simply \emph{star section}
and we write $\xi\sim_\s\bar\xi$.
Indeed, since $\bar{X_\xi}\in\Gamma(TP)$ and 
$T_g\s X_\xi(g)=\bar{X_\xi}(\s(g))$ for all $g\in G$ where this makes sense, 
the vector fields $X_\xi$ and $\bar{X_\xi}$ are $\s$-related, $X_\xi\sim_\s\bar{X_\xi}$.
Note that outside of $P$, $\xi$ 
is defined modulo sections of $\mathsf D_G\cap\ker\TT\s$.


\begin{proof}
We have shown in Lemma \ref{const_rk_int_kern}
that $\mathsf D_G\cap\ker\TT\s$ is a subbundle of $\mathsf D_G$.
Hence, we can consider the smooth 
vector bundle  $\mathsf D_G/ (\mathsf D_G\cap\ker\TT\s)$ over $G$.
Since $\mathsf D_G$ is a Lie subgroupoid of $\mathsf P_G\rr(TP\oplus A^*G)$,
we can consider the restriction to $\mathsf D_G$ of the source map, 
$\TT\s:\mathsf D_G\to \mathfrak A(\mathsf D_G)$. Since 
$\mathsf D_G\cap\ker\TT\s$ is the kernel of this map, 
we have an induced smooth vector bundle homomorphism
$\overline{\TT\s}:\mathsf D_G/ (\mathsf D_G\cap\ker\TT\s)\to \mathfrak A(\mathsf D_G)$
over the source map $\s:G\to P$, that is bijective in every fiber.
Hence, there exists a unique smooth section $[\xi]$
of $\mathsf D_G/(\mathsf D_G\cap\ker\TT\s)$ 
such that $\overline{\TT\s}([\xi](g))=\bar\xi(\s(g))$
for all $g\in G$.
If $\xi\in\Gamma(\mathsf D_G)$ is a representative of 
$[\xi]$
such that $\xi\an{P}=\bar\xi$, then 
$\TT\s(\xi(g))=\bar\xi(\s(g))$
for all $g\in G$.
\end{proof}

\begin{lemma}
Choose $\bar\xi, \bar\eta\in\Gamma(\mathfrak A(\mathsf D_G))$ and 
star sections $\xi\sim_\s \bar\xi$, 
$\eta\sim_\s\bar\eta$ of $\mathsf D_G$.
Then, if $\xi=(X_\xi,\theta_\xi)$ and $\eta=(X_\eta,\theta_\eta)$, the identity
\begin{equation}\label{key}
\theta_\eta(\ldr{Z^l}X_\xi)+\left(\ldr{Z^l}\theta_\xi\right)(X_\eta)=\s^*\Bigl( \bigl(
\theta_\eta(\ldr{Z^l}X_\xi)+(\ldr{Z^l}\theta_\xi)(X_\eta)\bigr)\an{P}
\Bigr)
\end{equation}
holds for any section $Z\in\Gamma(AG)$.
\end{lemma}

\begin{proof}
Choose $g\in G$ and set $p=\s(g)$. 
For all $t\in (-\varepsilon, \varepsilon)$ for a small $\varepsilon$, 
we have 
\begin{align*}
\xi(g\star \Exp(tZ)(p))&=
\bigl(\xi(g\star \Exp(tZ)(p))\bigr)
\star \bigl(\xi(\Exp(tZ)(p))\bigr)\inv
\star (\xi(\Exp(tZ)(p))\bigr).
\end{align*}
The pair $$\bigl(\xi(g\star \Exp(tZ)(p))\bigr)
\star \bigl(\xi(\Exp(tZ)(p))\bigr)\inv$$
 is an element 
of $\mathsf D_G(g)$ for all $t\in (-\varepsilon, \varepsilon)$
and will be written $\delta_t(g)$ to simplify the notation.
Note that we have 
\begin{align}\label{formula_for_product}
&\bigl(T_{R_{\Exp(tZ)}(g)}R_{\Exp(-tZ)}X_\xi(g\star
\Exp(tZ)(p)), 
\theta_\xi(g\star\Exp(tZ)(p))\circ T_gR_{\Exp(tZ)}\bigr)\nonumber\\
=\,&\delta_t(g)\star\bigl(T_{\Exp(tZ)}R_{\Exp(-tZ)}X_\xi(\Exp(tZ)(p)),
\theta_\xi(\Exp(tZ)(p))\circ T_pR_{\Exp(tZ)}\bigr).
\end{align}
We compute 
\begin{align*}
&\left(\theta_\eta(\ldr{Z^l}X_\xi)+(\ldr{Z^l}\theta_\xi)(X_\eta)\right)(g)\\
=\,\,&\,\left\langle \eta(g), 
\left.\frac{d}{dt}\right\an{t=0}\left(\left(R_{\Exp(tZ)}^*X_\xi\right)(g), 
\left(R_{\Exp(tZ)}^*\theta_\xi\right)(g)\right)\right\rangle\\
\overset{\eqref{formula_for_product}}=&\,\left.\frac{d}{dt}\right\an{t=0}
\biggl\langle \eta(g)\star \bar\eta(p), \, \delta_t(g)\star\left(\left(R_{\Exp(tZ)}^*X_\xi\right)(p), 
\left(R_{\Exp(tZ)}^*\theta_\xi\right)(p)\right)
\biggr\rangle\\
=\,\,&\,\left.\frac{d}{dt}\right\an{t=0}
\bigl\langle \eta(g), \delta_t(g)\bigr\rangle +\left.\frac{d}{dt}\right\an{t=0}
\biggl\langle  \bar\eta(p),
 \left(\left(R_{\Exp(tZ)}^*X_\xi\right)(p), 
\left(R_{\Exp(tZ)}^*\theta_\xi\right)(p)\right)
\biggr\rangle\\
=\,\,&\,\left(\left.\frac{d}{dt}\right\an{t=0}0\right)+\bigl\langle \bar\eta,
 (\ldr{Z^l}X_\xi, \ldr{Z^l}\theta_\xi)
\bigr\rangle(p)=\left(\bar\theta_\eta(\ldr{Z^l}X_\xi)+(\ldr{Z^l}\theta_\xi)(\bar X_\eta)\right)(\s(g)).
\end{align*}
\end{proof}

\begin{proposition}\label{lie_structure_dual}
Let $(G\rr P, \mathsf D_G)$ be a Dirac groupoid.
Choose $\bar\xi, \bar\eta\in\Gamma(\mathfrak A(\mathsf D_G))$ and 
star sections $\xi\sim_\s \bar\xi$, 
$\eta\sim_\s \bar\eta$ of $\mathsf D_G$, as in Proposition \ref{sdescending}.
Then the Courant-Dorfman bracket
$$\left[\xi, \eta\right]
=\left([X_\xi, X_\eta], \ldr{X_\xi}\theta_\eta-\ip{X_\eta}\dr\theta_\xi\right)
$$
is a star section and 
its values on $P$ are elements of $TP\oplus A^*G$.
\end{proposition}

\begin{proof}
Since $X_\xi\sim_\s\bar{X_\xi}$ and $X_\eta\sim_\s\bar{X_\eta}$, we know that
$[X_\xi,X_\eta]\sim_\s[\bar{X_\xi}, \bar{X_\eta}]$. 
Since $X_\xi\an{P}=\bar{X_\xi}$, $X_\eta\an{P}=\bar{X_\eta}$ the value
of $[X_\xi,X_\eta]$ on points in $P$ is equal
to the value of $[\bar{X_\xi}, \bar{X_\eta}]\in\mx(P)$.
We check that for all $p\in P$, we have
$\hat\s\left((\ldr{X_\xi}\theta_\eta-\ip{X_\eta}\dr\theta_\xi)(g)\right)
=(\ldr{X_\xi}\theta_\eta-\ip{X_\eta}\dr\theta_\xi)(p)$ for any $g\in\s\inv(p)$.

We have  for any $Z\in\Gamma(AG)$:
\begin{align*}
\hat\s\left(\left(\ldr{X_\xi}\theta_\eta-\ip{X_\eta}\dr\theta_\xi\right)(g)\right)(Z(p))
&=\left(\ldr{X_\xi}\theta_\eta-\ip{X_\eta}\dr\theta_\xi\right)\left(Z^l\right)(g).
\end{align*}
Hence, we compute with \eqref{key}
\begin{align*}
&\left(\ldr{X_\xi}\theta_\eta-\ip{X_\eta}\dr\theta_\xi\right)(Z^l)\\
=\,&X_\xi(\theta_\eta(Z^l))+\theta_\eta(\ldr{Z^l}X_\xi)-X_\eta(\theta_\xi(Z^l))
+Z^l(\theta_\xi(X_\eta))-\theta_\xi(\ldr{Z^l}X_\eta)\\
=\,&X_\xi(\s^*(\bar{\theta_\eta}(Z)))+\theta_\eta(\ldr{Z^l}X_\xi)-X_\eta(\s^*(\bar{\theta_\xi}(Z)))
+(\ldr{Z^l}\theta_\xi)(X_\eta)\\
=\,&\s^*\Bigl(\bar{X_\xi}(\bar{\theta_\eta}(Z))
+\bar{\theta_\eta}(\ldr{Z^l}X_\xi)-{\bar X_\eta}(\bar{\theta_\xi}(Z))
+(\ldr{Z^l}\theta_\xi)(\bar{X_\eta})
\Bigr).
\end{align*}
We  have then also for $p\in P$:
\begin{align*}
\left(\ldr{X_\xi}\theta_\eta-\ip{X_\eta}\dr\theta_\xi\right)(Z(p))
&=\left(\bar{X_\xi}(\bar{\theta_\eta}(Z))+\bar{\theta_\eta}(\ldr{Z^l}X_\xi)-{\bar X_\eta}(\bar{\theta_\xi}(Z))
+(\ldr{Z^l}\theta_\xi)(\bar{X_\eta})
\right)(p).
\end{align*}
Choose $X\in\Gamma(TP)$, then
\begin{align*}
\left(\ldr{X_\xi}\theta_\eta-\ip{X_\eta}\dr\theta_\xi\right)(X(p))
=\,\,&\bar{X_\xi}(\bar{\theta_\eta}(X))(p)+\bar{\theta_\eta}([X,\bar{X_\xi}])(p)
-\bar{X_\eta}(\bar{\theta_\xi}(X))(p)\\
&+X(\bar{\theta_\xi}(\bar{X_\eta}))(p)-\bar{\theta_\xi}([X, \bar{X_\eta}])(p)=0
\end{align*}
since $\bar{\theta_\eta}, \bar{\theta_\xi}\in\Gamma(TP^\circ)$ and 
$X,  [X,\bar{X_\xi}], \bar{X_\eta}, [X, \bar{X_\eta}]\in\Gamma(TP)$.

Thus, we have shown that
$\left(\ldr{X_\xi}\theta_\eta-\ip{X_\eta}\dr\theta_\xi\right)\an{P}$ is a section 
of $A^*G=TP^\circ$ and
$$\hat\s\left(\left(\ldr{X_\xi}\theta_\eta-\ip{X_\eta}\dr\theta_\xi\right)(g)\right)
=\left(\ldr{X_\xi}\theta_\eta-\ip{X_\eta}\dr\theta_\xi\right)(\s(g))$$
for all $g\in G$.
\end{proof}

\begin{theorem}\label{lie_algebroid_dual}
Let $(G\rr P,\mathsf D_G)$ be a Dirac groupoid.
Then there is an induced antisymmetric bracket 
\begin{equation*}
\begin{array}{cccc}
[\cdot\,,\cdot]_\star&:\Gamma(\mathfrak A(\mathsf D_G))\times
\Gamma(\mathfrak A(\mathsf D_G))&\to &\Gamma( TP\oplus A^*G)
\end{array}
\end{equation*}
defined by $
[\bar\xi, \bar\eta]_\star=[\xi,\eta]\an{P}$
for any choice of star sections $\xi\sim_\s \bar\xi$, 
$\eta\sim_\s \bar\eta$ of $\mathsf D_G$.
If $(G\rr P,\mathsf D_G)$ is closed, then 
$(\mathfrak A(\mathsf D_G), [\cdot\,,\cdot]_\star, \mathsf a_\star)$ is a Lie algebroid
over $P$.
\end{theorem}

\begin{proof}
By Proposition \ref{lie_structure_dual}, if 
$ \xi\sim_\s \bar\xi$, 
$\eta\sim_\s \bar\eta$
then 
$$[(X_\xi,\theta_\xi), (X_\eta,\theta_\eta)]\sim_\s
\left([\bar{X_\xi},\bar{X_\eta}],
  \left(\ldr{X_\xi}\theta_\eta-\ip{X_\eta}\dr\theta_\xi\right)\an{P}\right).$$
Thus, we have first to show that
the right-hand side of this equation does not depend on the choice of the sections
$\xi$ and 
$\eta$. Choose a star section $\nu\sim_\s0$ of $\mathsf D_G$,
i.e., $\nu\in\Gamma(\mathsf D_G\cap\ker\TT\s)$ with $\nu\an{P}=0$.
For any $Z\in\Gamma(AG)$, we find  as in the proof of Proposition 
\ref{lie_structure_dual}
\begin{align*}
\left(\ldr{X_\nu}\theta_\xi-\ip{X_\xi}\dr\theta_\nu\right)\left(Z^l\right)
&= \s^*(\bar{X_\nu}(\bar{\theta_\xi}(Z)))+\theta_\xi(\ldr{Z^l}X_\nu)
-\s^*({\bar X_\xi}(\bar{\theta_\nu}(Z)))+(\ldr{Z^l}\theta_\nu)(X_\xi)\\
&=\theta_\xi(\ldr{Z^l}X_\nu)+(\ldr{Z^l}\theta_\nu)(X_\xi)
\qquad  \text{ since } \bar X_\nu=0 \text{ and } \bar{\theta_\nu}=0\\
&=\ldr{Z^l}\langle (X_\nu,\theta_\nu), (X_\xi, \theta_\xi)\rangle
-\langle\ldr{Z^l}(X_\xi, \theta_\xi), (X_\nu,\theta_\nu)\rangle\\
&= -\langle\ldr{Z^l}(X_\xi, \theta_\xi), (X_\nu,\theta_\nu)\rangle.
\end{align*}
Hence, at any $p\in P$, we find 
\begin{align*}
(\ldr{X_\nu}\theta_\xi-\ip{X_\xi}\dr\theta_\nu)(Z(p))
&= -\langle\ldr{Z^l}(X_\xi, \theta_\xi)(p), (\bar{X_\nu},\bar{\theta_\nu})(p)\rangle\\
&=-\langle\ldr{Z^l}(X_\xi, \theta_\xi)(p), (0_p,0_p)\rangle=0.
\end{align*}
Thus, we find $(\ldr{X_\nu}\theta_\xi-\ip{X_\xi}\dr\theta_\nu)(p)=0_p$
since we know by the previous proposition that 
$(\ldr{X_\nu}\theta_\xi-\ip{X_\xi}\dr\theta_\nu)(p)\in A_p^*G=T_pP^\circ$.
We get hence
\begin{align*}
\left[(X_\nu,\theta_\nu), (X_\xi,\theta_\xi)\right](p)
&=\left([\bar{X_\nu}, \bar{X_\xi}],  
\ldr{X_\nu}\theta_\xi-\ip{X_\xi}\dr\theta_\nu\right)(p)=([0,\bar{X_\xi}]_p, 0_p)=(0_p, 0_p).
\end{align*}
This shows that the bracket on $\Gamma(\mathfrak A(\mathsf D_G))$ is well-defined.
It is antisymmetric because
the Courant-Dorfman bracket on sections  of $\mathsf D_G$ is antisymmetric.

\medskip

If $\mathsf D_G$ is closed, then for all star sections
$\xi,\eta\in\Gamma(\mathsf D_G)$, the bracket
$[\xi,\eta]$ is also a section of $\mathsf D_G$ and 
its restriction to $P$
is a section of $\mathfrak A(\mathsf D_G)$ since it is a section 
of $TP\oplus A^*G$.
The Jacobi identity is satisfied by $[\cdot\,,\cdot]_\star$ because the Courant-Dorfman
 bracket on sections
of $\mathsf D_G$ satisfies the Jacobi identity. 
For any $\bar\xi, \bar\eta\in\Gamma(\mathfrak A(\mathsf D_G))$ and 
$f\in C^\infty(P)$, we have 
$$\mathsf a_\star\left[\bar\xi, \bar\eta\right]_\star
=\left[\bar{X_\xi},\bar{X_\eta}\right]=\left[\mathsf a_\star(\bar\xi), 
\mathsf a_\star(\bar\eta)\right]$$
and
\begin{align*}
\left[\bar\xi, f\cdot\bar\eta\right]_\star(p)
&=\left[ (X_\xi,\theta_\xi), (\s^*f)(X_\eta, \theta_\eta)\right](p)\\
&=X_\xi(\s^*f)(X_\eta, \theta_\eta)(p)+(\s^*f)\left[ (X_\xi,\theta_\xi), (X_\eta, \theta_\eta)\right](p)\\
&=\bar{X_\xi}(f)(p)\cdot (\bar X_\eta,\bar{\theta_\eta})(p)
+ f(p)\cdot\left[ \bar\xi, \bar\eta\right]_\star(p)\\
&=\mathsf a_\star(\bar\xi)(f)(p)\cdot \bar\eta(p)
+ f(p)\cdot\left[ \bar\xi, \bar\eta\right]_\star(p)
 \end{align*} 
for all $p\in P$.
\end{proof}

\medskip

Let $G\rr P$ be a Lie groupoid, $TG\rr TP$ its tangent prolongation and 
$(A\to P, \mathsf a, [\cdot\,,\cdot]_{\mathsf a})$ a Lie algebroid over $P$.
Let $\Omega$ be a smooth manifold. The quadruple $(\Omega; G, A; P)$ is 
a \emph{$\mathcal L A$-groupoid} \cite{Mackenzie00} if 
$\Omega$ has  both a  Lie groupoid structure over $A$ and a Lie algebroid
structure   over $G$ such that the two structures on $\Omega$ commute in the sense 
that the maps defining  the groupoid structure are all Lie algebroid 
morphisms.
(The bracket on sections of $\lie A(\mathsf D_G)$ 
can be defined in the same manner with the target map, and the fact that the multiplication
in $T^*G\oplus TG$ is a Lie algebroid morphism is shown in \cite{Ortiz08t}.) 
 The double source map $(\tilde q, \tilde \s):\Omega\to G\oplus A$ 
has furthermore to be a surjective submersion.

Recall from \cite{Courant90a} that if $\mathsf D_G$
is closed, then $\mathsf D_G\to G$ has the structure of a Lie algebroid
with the Courant-Dorfman bracket
and the projection on $TM$ as anchor. Thus, the previous theorem 
shows that the quadruple $(\mathsf D_G; G, \lie A(\mathsf D_G); P)$ 
is a $\mathcal L A$-groupoid (see also \cite{Ortiz08t}):
\begin{displaymath}\begin{xy}
\xymatrix{ 
 \mathsf D_G\ar@<.6ex>^{\TT\s}[rr]\ar@<-.6ex>_{\TT\tg}[rr]
\ar[dd]_{q}\ar[dr]^{\pi_{TG}}&&\lie A(\mathsf D_G)\ar@{-->}[dd]\ar[dr]^{\mathsf a_\star}&\\
& TG \ar@<.6ex>^{\qquad T\s}[rr]\ar@<-.6ex>_{\qquad T\tg}[rr]\ar[dl]&&TP\ar[dl]\\
G\ar@<.6ex>^{\s}[rr]\ar@<-.6ex>_{\tg}[rr]&&P&
}
\end{xy}
\end{displaymath}

We recover hence our  star sections of $\mathsf D_G$
as  the star sections of  the $\mathcal L$A-groupoid
$(\mathsf D_G; G, \lie A(\mathsf D_G); P)$ as in \cite{Mackenzie00}.
It is shown in  \cite{Mackenzie00}
  (see also \cite{Mackenzie92}), that the bracket
of two star sections is again a star section.
Here, we have shown this fact in Proposition \ref{lie_structure_dual}
and get as a consequence the fact that $\lie A(\mathsf D_G)$ has the structure
of a Lie algebroid over $P$.

The next interesting object in \cite{Mackenzie00} is the \emph{core} $K$ of $\Omega$.
It is defined as the pullback vector bundle across $\epsilon:P\hookrightarrow G$
of the kernel $\ker(\tilde \s:\Omega\to A)$. Hence, it is here 
exactly the vector bundle  $I^\s(\mathsf D_G)$ over $P$. It comes 
equipped with the vector bundle morphisms
$\delta_{\lie A(\mathsf D_G)}: I^\s(\mathsf D_G)\to \lie A(\mathsf D_G)$,
$(v_p,\alpha_p)\mapsto\TT \tg(v_p,\alpha_p)$ and 
$\delta_{\widetilde{AG}}:I^\s(\mathsf D_G)\to \widetilde{AG}$,
$(v_p,\alpha_p)\mapsto v_p$.
We have then $\tilde{\mathsf a}\circ \delta_{\widetilde{AG}}
=\mathsf a_\star\circ \delta_{\lie A(\mathsf D_G)}=:\mathsf k$.
Furthermore,
there is an induced bracket $[\cdot\,,\cdot]_{I^\s(\mathsf D_G)}$ on sections of $I^\s(\mathsf D_G)$
such that $(I^\s(\mathsf D_G), [\cdot\,,\cdot]_{I^\s(\mathsf D_G)}, \mathsf k)$ is a Lie algebroid
over $P$. We prove this fact for our special situation in the following proposition.

Recall that if $(v_p,\alpha_p)$, $p\in P$, is an element of 
$I^\s_p(\mathsf D_G)$, then $\alpha_p$ can be written $(T_p\tg)^*\beta_p$ with some 
$\beta_p\in T_p^*P$. Furthermore, 
if $\sigma$ is a section of $I^\s(\mathsf D_G)\subseteq 
(T^\s G\oplus (T^\tg G)^\circ)\an{P}$, then 
$\sigma^r$ defined
by $\sigma^r(g)=\sigma(\tg(g))\star (0_g,0_g)$ for all
$g\in G$
is a section of $\mathsf D_G\cap\ker\TT\s$
 by Lemma \ref{const_rk_int_kern}
and Remark \ref{computation}.
We write 
$\sigma^r=(X_\sigma^r, \tg^*\alpha_\sigma)$
with some $X_\sigma\in\Gamma\left(\widetilde{AG}\right)$ and $\alpha_\sigma\in\Omega^1(P)$.

\begin{proposition}\label{other_algebroid}
Let $(G\rr P, \mathsf D_G)$ be a Lie groupoid. 
Define $[\cdot\,,\cdot]_{I^\s(\mathsf D_G)}:\Gamma(I^\s(\mathsf D_G))\times \Gamma(I^\s(\mathsf D_G))
\to\Gamma((\ker\TT\s)\an{P})$ by 
$$\left([\sigma,\tau]_{I^\s(\mathsf D_G)}\right)^r=\left[\sigma^r,\tau^r\right]$$
for all sections $\sigma,\tau\in \Gamma(I^\s(\mathsf D_G))$, i.e.,
$$[\sigma,\tau]_{I^\s(\mathsf D_G)}
=\left([X_\sigma,X_\tau]_{\widetilde{AG}}, 
(\tg^*(\ldr{\tilde{\mathsf a}(X_\sigma)}\alpha_\tau
-\ip{\tilde{\mathsf a}(X_\tau)}\dr\alpha_\sigma))\an{P}\right).$$ 
If $\mathsf D_G$ is closed, this bracket has image in $\Gamma(I^\s(\mathsf D_G))$
and $I^\s(\mathsf D_G)$ has the structure of a Lie algebroid over $P$ with 
the anchor map $\mathsf k$
defined by $\mathsf k(v_p,\alpha_p)=T_p\tg v_p$ for all $(v_p,\alpha_p)\in I^\s_p(\mathsf D_G)$, $p\in P$.
\end{proposition}
Note that this bracket on $I^\s(\mathsf D_G)$ is the restriction to
$I^\s(\mathsf D_G)$ of a bracket defined in the same manner
on the sections of $(\ker\TT\s)\an{P}$. 
Note also that, if 
$\mathsf D_G$ is closed, 
the space $I^\tg(\mathsf D_G)$ has in the same manner the structure of an algebroid over $P$.

\begin{proof}
Choose $\sigma,\tau\in \Gamma(I^\s(\mathsf D_G))$
and assume that $\mathsf D_G$ is closed.
 The
 bracket $$[\sigma^r,\tau^r]
=\left[\left(X_\sigma^r, \tg^*\alpha_\sigma\right), 
\left(X_\tau^r, \tg^*\alpha_\tau\right)\right]$$ is then itself
a section of $\mathsf D_G$. The identity 
\begin{align*}
\left[\left(X_\sigma^r, \tg^*\alpha_\sigma\right), \left(X_\tau^r, \tg^*\alpha_\tau\right)\right]
=\left( \left([X_\sigma,X_\tau]_{\widetilde{AG}}\right)^r, \tg^*(\ldr{\tilde{\mathsf a}(X_\sigma)}\alpha_\tau
-\ip{\tilde{\mathsf a}(X_\tau)}\dr\alpha_\sigma)\right)
\end{align*}
shows hence that
$\left[\sigma^r, \tau^r\right]\in\Gamma(\mathsf D_G\cap\ker\TT\s)$
is right invariant 
and consequently
$$[\sigma, \tau]_{I^\s(\mathsf D_G)}
=\left[(X_\sigma^r, \tg^*\alpha_\sigma), (X_\tau^r, \tg^*\alpha_\tau)\right]\an{P}
\in\Gamma(I^\s(\mathsf D_G)).$$
The bracket $[\cdot\,,\cdot]_{I^\s(\mathsf D_G)}$ satisfies then the Jacobi identity because the Courant 
bracket on sections of $\mathsf D_G$ satisfies it.
 The Leibniz rule is easy to check.
\end{proof}

As in \cite{Mackenzie00}, we have thus 
four Lie algebroids over $P$:

\begin{displaymath}\begin{xy}\label{diagram}
\xymatrix{ 
 I^\s(\mathsf D_G)\ar[dr]\ar[dd]_{\delta_{\widetilde{AG}}}\ar[rr]^{\delta_{\lie A(\mathsf D_G)}}
&&\lie A(\mathsf D_G)\ar[dl]\ar[dd]^{\mathsf a_\star}\\
&P&\\
\widetilde{AG}\ar[ur]\ar[rr]_{\tilde{\mathsf a}}&&TP\ar[ul]
}
\end{xy}
\end{displaymath}
The anchors $\tilde{\mathsf a}$, $\mathsf a_\star$ and
  the map $\delta_{\widetilde{AG}}$ are obviously Lie algebroid morphisms
and 
the theory in \cite{Mackenzie92}, 
\cite{Mackenzie00} 
yields that 
 $\delta_{\lie A(\mathsf D_G)}$ is also a Lie algebroid morphism.

\bigskip

Next, we compute the Lie algebroid $\lie A(\mathsf D_G)\to P$ for our three ``standard'' examples.
\begin{example}\label{ex_dual_of_lie_algebroid_poisson}
Let $(G\rr P,\pi_G)$ be a Poisson groupoid and $\mathsf D_{\pi_G}$ the graph
of the vector bundle homomorphism $\pi_G^\sharp:T^*G\to TG$ associated
to $\pi_G$. The pair $(G\rr P,\mathsf D_{\pi_G})$ is
a closed Dirac groupoid. The set of units
$\lie A(\mathsf  D_{\pi_G})$ of $\mathsf D_{\pi_G}$
equals here  $\graph\left(\left.\pi_G^\sharp\right\an{A^*G}:A^*G\to TP\right)$
and is hence isomorphic to $A^*G$ as a vector bundle, 
via the maps 
$\Theta:=\pr_{A^*G}:\lie A(\mathsf  D_{\pi_G})\to A^*G$
and $\Theta\inv=\left(\left.\pi_G^\sharp\right\an{A^*G},\Id_{A^*G}\right): A^*G\to \lie A(\mathsf  D_{\pi_G})$
over $\Id_P$.

The vector bundle $A^*G$ has the structure of a Lie algebroid over $P$ 
with anchor map given by  $A^*G\to TP$, $\alpha_p\mapsto \pi_G^\sharp(\alpha_p)\in T_pP$
and with  bracket
the restriction to $A^*G$ of the bracket $[\cdot\,,\cdot]_{\pi_G}$
on $\Omega^1(G)$ defined 
by $\pi_G$:
$[\alpha,\beta]_{\pi_G}=\ldr{\pi_G^\sharp(\alpha)}\beta
-\ldr{\pi_G^\sharp(\beta)}\alpha-\dr\pi_G(\alpha,\beta)$
for all $\alpha,\beta\in\Omega^1(G)$
(\cite{CoDaWe87}).
Thus, $A^*G$ with this Lie algebroid structure 
and $\lie A(\mathsf  D_{\pi_G})$ 
are isomorphic as Lie algebroids via $\Theta$ and $\Theta\inv$.
\end{example}

\begin{example}
Let $\omega_G$ be a multiplicative closed
$2$-form on a Lie  groupoid $G\rr P$ and consider 
the associated multiplicative Dirac structure $\mathsf D_{\omega_G}$
on $G$. The Lie algebroid 
$\lie A(\mathsf D_{\omega_G})\to P$
 is here equal to 
$$\lie A(\mathsf D_{\omega_G})=\graph\left(\omega_G^\flat\an{TP}:TP\to A^*G\right)$$
with anchor map $\mathsf a_\star:\lie A(\mathsf D_{\omega_G})\to TP$
given by
$\mathsf a_\star\left(v_p,\omega_G^\flat(v_p)\right)=v_p$.
The bracket of two sections 
$\left(\bar X, \omega_G^\flat(\bar X)\right), \left(\bar Y, \omega_G^\flat(\bar Y)\right)
\in\Gamma(\lie A(\mathsf D_{\omega_G}))$
is simply given by 
$$\left[\left(\bar X, \omega_G^\flat(\bar X)\right), \left(\bar Y, \omega_G^\flat(\bar Y)\right)\right]
=\left([\bar X, \bar Y], \omega_G^\flat\left([\bar X, \bar Y]\right)\right).$$

The Lie algebroid $\lie A(\mathsf D_{\omega_G})$
is obviously isomorphic to the tangent Lie algebroid $TP\to P$ of $P$, 
via the maps $\pr_{TP}:\lie A(\mathsf D_{\omega_G})\to TP$ (the anchor map)
and 
$\left(\Id_{TP},\omega_G^\flat\an{TP}\right): TP\to \lie A(\mathsf D_{\omega_G})$.

Note that 
if $(G\rr P, \omega)$ is a Lie groupoid endowed with a multiplicative closed $2$-form,
 then
$\lie A(\mathsf D_{\omega})$
is the graph of the  dual of the map $\sigma_\omega: AG\to T^*P$ in 
\cite{BuCaOr09}. 
\end{example}

\begin{example}\label{algebroid_of_pair_Dirac}
Let $(M,\mathsf D_M)$ be a smooth Dirac manifold and 
$(M\times M\,\rr\, M, \mathsf D_M\ominus\mathsf D_M)$ the associated pair Dirac groupoid
as in Example \ref{ex_pair_Dirac_Lie}. The set 
$\lie A(\mathsf D_M\ominus\mathsf D_M)$ is defined here
by $$\lie A(\mathsf D_M\ominus\mathsf D_M)_{(m,m)}=
\TT\tg((\mathsf D_M\ominus\mathsf D_M)(m,m))
=\{(v_m,v_m,\alpha_m,-\alpha_m)\mid (v_m,\alpha_m)\in\mathsf D_M(m)\}
$$
for all $m\in M$. Hence, we have 
an isomorphism 
$\lie A(\mathsf D_M\ominus\mathsf D_M)\to \mathsf D_M$ 
over the map $\pr_{1}:\Delta_M\to M$. 
Sections of $\lie A(\mathsf D_M\ominus\mathsf D_M)$
are exactly the sections 
$(X,X,\alpha,-\alpha)\an{\Delta_M}$ for sections 
 $(X,\alpha)\in\mathsf D_M$. The section
$(X,X,\alpha,-\alpha)$ of $\mathsf D_M\ominus \mathsf D_M$
defined on $M\times M$ by 
$(X,X,\alpha,-\alpha)(m,n)=(X(m), X(n),\alpha(m),-\alpha(n))$ 
for all $(m,n)\in M\times M$ is then
easily shown to be star sections $\s$-related to 
$(X,X,\alpha,-\alpha)\an{\Delta_M}$.
Using this, one can check that, if $(M,\mathsf D_M)$ is closed,
 the Lie algebroid structure 
on $ \lie A(\mathsf D_M\ominus\mathsf D_M)$
corresponds to the Lie algebroid structure on 
$(M,\mathsf D_M)$ (see \cite{Courant90a}).
\end{example}

\subsection{Integrability criterion}\label{int_crit}
The main  theorem of this section shows that the integrability of a Dirac groupoid is 
completely encoded in its square of Lie algebroids. The proof, 
which is very technical, will be only summarized here. It can be found with more details in \cite{thesis}.
We begin by showing 
a derivation formula  for star sections, that will also be useful later.
\begin{theorem}\label{lie_der_of_xi_section}
Let $(G\rr P,\mathsf D_G)$ be a Dirac groupoid,
$\xi\sim_{\s}\bar{\xi}$ a star section of 
$\mathsf D_G$ and $Z\in\Gamma(AG)$. Then the derivative 
$\ldr{Z^l}(X_\xi,\theta_\xi)$ 
can be written 
as a sum 
\begin{equation}\label{der_of_s_section}
\ldr{Z^l}(X_\xi,\theta_\xi)=\left(X_{\mathcal L_Z\xi}, \theta_{\mathcal L_Z\xi}\right)+
\left(Y_{\xi,Z}^l,\s^*\alpha_{\xi,Z}\right)=:\mathcal L_Z\xi+(\sigma_{\xi,Z})^l
\end{equation}
with 
$Y_{\xi,Z}\in\Gamma(AG)$, $\alpha_{\xi,Z}\in\Omega^1(P)$ 
and $\mathcal L_Z\xi:=\left(X_{\mathcal L_Z\xi}, \theta_{\mathcal L_Z\xi}\right)$
a star section of $\mathsf
D_G$. 
We have $\mathcal L_Z\xi\sim_\s\TT\tg\left(\ldr{Z^l}(X_\xi,\theta_\xi)\an{P}\right)$
in the sense that 
$$\TT\s\left(X_{\mathcal L_Z\xi}(g), \theta_{\mathcal L_Z\xi}(g)\right)
=\TT\tg\left(\ldr{Z^l}(X_\xi,\theta_\xi)(\s(g))\right)$$
for all $g\in G$.

In addition, if $(X_\nu, \theta_\nu)\sim_s(0,0)$, 
then $\ldr{Z^l}(X_\nu, \theta_\nu)\in\Gamma(\mathsf D_G\cap\ker\TT\s)$.
In particular, its restriction to $P$ is a section of $I^\s(\mathsf D_G)$.
\end{theorem}

The following lemma will be useful for the proof of this theorem.
The proof is easy and shall be omitted.

\begin{lemma}\label{der_of_kers}
Let $G\rr P$ be a Lie groupoid.
Choose $(X,(\tg^*\alpha)\an{P})\in\Gamma((\ker\TT\s)\an{P})$ and 
$Z\in\Gamma(AG)$.
Then we have 
$$\ldr{Z^l}(X^r,\tg^*\alpha)=0.$$
\end{lemma}

\begin{proof}[of Theorem \ref{lie_der_of_xi_section}]
Note first that, in general $\ldr{Z^l}(X_\xi,\theta_\xi)$ is  a section 
of $\mathsf D_G+\ker\TT\tg$: 
for all $\sigma^r=(X_\sigma^r,\tg^*\alpha_\sigma)\in\Gamma(\mathsf D_G\cap\ker\TT\s)$, we have
$$\langle\ldr{Z^l}(X_\xi,\theta_\xi), \sigma^r\rangle
=\ldr{Z^l}\left(\langle(X_\xi,\theta_\xi), (X_\sigma^r,\tg^*\alpha_\sigma)\rangle\right)
-\langle (X_\xi,\theta_\xi), \ldr{Z^l}(X_\sigma^r,\tg^*\alpha_\sigma)\rangle=0$$
 using $\mathsf D_G=\mathsf D_G^\perp$ and  Lemma \ref{der_of_kers}.
This leads to $\ldr{Z^l}(X_\xi,\theta_\xi)\in\Gamma\left((\mathsf
D_G\cap\ker\TT\s)^\perp\right)
=\Gamma(\mathsf D_G+\ker\TT\tg)$.
Choose $g\in G$. Then 
\begin{align*}
T_g\s(\ldr{Z^l}X_\xi)(g)&=T_g\s\left[Z^l, X_\xi\right](g)=\left[\mathsf a(Z), \bar X_\xi\right](\s(g))
\end{align*}
and for any $W\in\Gamma(AG)$
\begin{align*}
\hat\s(\ldr{Z^l}\theta_\xi(g))(W(\s(g)))
&=(\ldr{Z^l}\theta_\xi)(W^l)(g)=(Z^l(\s^*(\bar{\theta_\xi}(W)))-\s^*(\bar{\theta_\xi}([Z,W]_{AG})))(g)\\
&=\bigl((\mathsf a(Z))(\bar{\theta_\xi}(W))-\bar{\theta_\xi}([Z,W]_{AG})\bigr)(\s(g)).
\end{align*}
This shows that $\TT\s\left(\ldr{Z^l}\left(X_\xi,\theta_\xi\right)\right)(g)$
depends only on the values of $Z,\bar X_\xi, \bar{\theta_\xi}$ at $\s(g)$.

Set $$(Y_{\xi,Z}^l,\s^*\alpha_{\xi,Z})(g):=
(0_g, 0_g)\star \Bigl( 
(\ldr{Z^l}X_\xi, \ldr{Z^l}\theta_\xi)(\s(g))-\TT\tg\left((\ldr{Z^l}X_\xi, \ldr{Z^l}\theta_\xi)(\s(g))\right)
\Bigr)$$
 and 
$$(X_{\mathcal L_Z\xi}, \theta_{\mathcal L_Z\xi})(g)
:=\left(\ldr{Z^l}X_\xi, \ldr{Z^l}\theta_\xi\right)(g)
-\left(Y_{\xi,Z}^l,\s^*\alpha_{\xi,Z}\right)(g)$$
for all $g\in G$.
Then $\left(Y_{\xi,Z}^l,\s^*\alpha_{\xi,Z}\right)$ is a smooth 
section of $\ker\TT\tg$ 
satisfying $$\TT\s\left(\left(Y_{\xi,Z}^l,\s^*\alpha_{\xi,Z}\right)(g)\right)
=\TT\s\left(\left(Y_{\xi,Z}^l,\s^*\alpha_{\xi,Z}\right)(\s(g))\right)$$
by construction for all $g\in G$ and
$\mathcal L_Z\xi=(X_{\mathcal L_Z\xi}, \theta_{\mathcal L_Z\xi})$ is consequently
 a star section
if we can show that 
$$\TT\s((X_{\mathcal L_Z\xi}, \theta_{\mathcal L_Z\xi})(g))
=(X_{\mathcal L_Z\xi}, \theta_{\mathcal L_Z\xi})(\s(g))$$
for all $g\in G$.
Using the computations 
above for $\TT\s\left(\ldr{Z^l}(X_\xi,\theta_\xi)\right)$, it is 
 easy to see that, for  $g\in G$, we have 
$$\TT\s\left((X_{\mathcal L_Z\xi}, \theta_{\mathcal L_Z\xi})(g)\right)=\TT\tg\left(\ldr{Z^l}X_\xi,
\ldr{Z^l}\theta_\xi\right)(\s(g)),$$
which, by definition, is equal to $(X_{\mathcal L_Z\xi}, \theta_{\mathcal L_Z\xi})(\s(g))$.

It remains hence to show that $(X_{\mathcal L_Z\xi}, \theta_{\mathcal L_Z\xi})$ is 
a section of $\mathsf D_G$. The equality 
\begin{align*}
\left\langle\sigma^r, (X_{\mathcal L_Z\xi}, \theta_{\mathcal L_Z\xi})\right\rangle
=\left\langle\sigma^r, (\ldr{Z^l}X_\xi,
  \ldr{Z^l}\theta_\xi)\right\rangle
-\left\langle\sigma^r, (Y_{\xi,Z}^l,\s^*\alpha_{\xi,Z})\right\rangle
=0-0
\end{align*}
holds
for all $\sigma^r\in\Gamma(\ker\TT\s\cap\mathsf D_G)$, and 
for all star sections $(X_\eta,\theta_\eta)$ 
of $\mathsf D_G$, we compute
\begin{align*}
&\left\langle\left(X_\eta, \theta_\eta\right), (X_{\mathcal L_Z\xi}, \theta_{\mathcal L_Z\xi})\right\rangle(g)\\
=\,\,&\left\langle\left(X_\eta, \theta_\eta\right)(g), \left(\ldr{Z^l}X_\xi,
  \ldr{Z^l}\theta_\xi\right)(g)-\left(Y_{\xi,Z}^l,\s^*\alpha_{\xi,Z}\right)(g)\right\rangle\\
=\,\,&\left\langle\left(X_\eta, \theta_\eta\right)(g), \left(\ldr{Z^l}X_\xi,
  \ldr{Z^l}\theta_\xi\right)(g)\right\rangle\\
&-\bigl\langle\left(X_\eta, \theta_\eta\right)(g)\star\left(\bar{X_\eta},\bar{\theta_\eta}\right)(\s(g)),
\\ 
&\hspace*{2cm}(0_g,0_g)\star\left( \left(\ldr{Z^l}X_\xi, \ldr{Z^l}\theta_\xi\right)(\s(g))
-\TT\tg\left(\ldr{Z^l}X_\xi, \ldr{Z^l}\theta_\xi\right)(\s(g))
\right)\bigr\rangle\\
\overset{\eqref{key}}=&\left\langle\left(\bar{X_\eta}, \bar{\theta_\eta}\right)(\s(g)), 
\left(\ldr{Z^l}X_\xi,
  \ldr{Z^l}\theta_\xi\right)(\s(g))\right\rangle\\
&-\left\langle\left(\bar{X_\eta},\bar{\theta_\eta}\right)(\s(g)), 
 \left(\ldr{Z^l}X_\xi, \ldr{Z^l}\theta_\xi\right)(\s(g))-
\TT\tg\left(\ldr{Z^l}X_\xi, \ldr{Z^l}\theta_\xi\right)(\s(g))
\right\rangle\\
=\,\,&\left\langle\left(\bar{X_\eta},\bar{\theta_\eta}\right)(\s(g)), 
\TT\tg\left(\ldr{Z^l}X_\xi, \ldr{Z^l}\theta_\xi\right)(\s(g))\right\rangle=0
\end{align*}
since $TP\oplus A^*G=\left(TP\oplus A^*G\right)^\perp$.
Thus, we have shown that 
$(X_{\mathcal L_Z\xi}, \theta_{\mathcal L_Z\xi})\in\Gamma({\mathsf D_G}^\perp)=\Gamma(\mathsf D_G)$.

\medskip

For the proof of the second statement, assume that
$(X_\nu,\theta_\nu)$ is a smooth section of $\mathsf D_G$ that
is a star section $\s$-related to $(\bar X_\nu, \bar{\theta_\nu})=(0,0)$.
For all left invariant sections $(Y^l,\s^*\gamma)$ of $\ker\TT\tg$, 
we have 
\begin{align*}
\left\langle \ldr{Z^l}(X_\nu,\theta_\nu), (Y^l,\s^*\gamma)\right\rangle
&=\ldr{Z^l}\left\langle (X_\nu,\theta_\nu), (Y^l,\s^*\gamma)\right\rangle
-\left\langle (X_\nu,\theta_\nu), \ldr{Z^l}(Y^l,\s^*\gamma)\right\rangle\\
&=\ldr{Z^l}\left(\s^*(\gamma(\bar X_\nu)+\bar{\theta_\nu}(Y))\right)
-\s^*\left(\bar{\theta_\nu}([Z,Y]_{AG})+(\ldr{\mathsf a(Z)}\gamma)(\bar X_\nu)\right)\\
&=0
\end{align*}
since $\bar X_\nu=0$ and $\bar\theta_\nu=0$.
Choose any star section $\xi=(X_\xi,\theta_\xi)$ of $\mathsf D_G$. Then 
\begin{align*}
\left\langle \ldr{Z^l}(X_\nu,\theta_\nu), (X_\xi,\theta_\xi)\right\rangle
&=\ldr{Z^l}\left\langle (X_\nu,\theta_\nu), (X_\xi,\theta_\xi)\right\rangle
-\left\langle (X_\nu,\theta_\nu), \ldr{Z^l}(X_\xi,\theta_\xi)\right\rangle=0
\end{align*}
since $\ldr{Z^l}(X_\xi,\theta_\xi)\in\Gamma\left((\mathsf D_G\cap\ker\TT\s)^\perp\right)$.
We have  also  $\ldr{Z^l}(X_\nu,\theta_\nu)\in\Gamma\left((\mathsf D_G\cap\ker\TT\s)^\perp\right)$
and, because the star sections of $\mathsf D_G$ and the sections of 
$\mathsf D_G\cap\ker\TT\s$ span $\mathsf D_G$, this shows that
$\ldr{Z^l}(X_\nu,\theta_\nu)\in\Gamma\left((\mathsf D_G+\ker\TT\tg)^\perp\right)
=\Gamma(\mathsf D_G\cap\ker\TT\s)$. 
\end{proof}

We have also for any star section 
$\xi$ of $\mathsf D_G$, any section $\sigma\in\Gamma(I^\s(\mathsf D_G))$ 
and $Z\in\Gamma(AG)$:
\begin{align*}
\frac{d}{dt}\left\langle R_{\Exp(tZ)}^*\xi,\sigma^r\right\rangle(g)
&=\frac{d}{dt}\left\langle \xi, R_{\Exp(-tZ)}^*\sigma^r\right\rangle(R_{\Exp(tZ)}(g))\\
&=\frac{d}{dt}\left\langle \xi, \sigma^r\right\rangle(R_{\Exp(tZ)}(g))=0
\end{align*}
since $R_K^*\sigma^r=\sigma^r$ for all bisections $K\in\mathcal B(G)$.
Hence, we get $$\left\langle R_{\Exp(tZ)}^*\xi,\sigma^r\right\rangle(g)
=\left\langle R_{\Exp(0\cdot Z)}^*\xi,\sigma^r\right\rangle(g)
=\left\langle \xi,\sigma^r\right\rangle(g)=0$$ for all 
$g\in G$, $\sigma\in\Gamma(I^\s(\mathsf D_G))$
 and $t\in\R$ where this makes sense  
 and we find consequently
$R_{\Exp(tZ)}^*\xi\in\Gamma(\mathsf D_G+\ker\TT\tg)$.
If the $\s$-fibers of $G\rr P$ are connected,
 the set of bisections 
of $G$ is generated  by the bisections $\Exp(tZ)$, 
$t\in\R$ small enough and $Z\in\Gamma(AG)$
(see  \cite{MaXu00}). 
We know then that 
$R_K^*\xi\in\Gamma(\mathsf D_G+\ker\TT\tg)$ for any bisection $K\in\mathcal B(G)$.

\bigskip

We denote here by $\mathsf S(\mathsf D_G)$ the set 
of star sections  of $\mathsf D_G$.
Note that
 $\mathsf D_G$ is spanned on $G\setminus P$ by the values of the elements 
of $\mathsf S(\mathsf D_G)$, 
since $\mathsf D_G\cap\ker\TT\s$ is spanned there by the values of 
the star sections that vanish on $P$.

\bigskip

Consider 
the vector bundle $\mathsf E:=
\mathsf D_G/(\mathsf D_G\cap\ker\TT\tg)\simeq (\mathsf D_G+\ker\TT\tg)/\ker\TT\tg$ 
over $G$.  
Since the fiber $\mathsf D_G(g)$ over $g$ of the 
 Dirac structure  is spanned for each  $g\in G\setminus P$ by the 
values of the elements of $\mathsf S(\mathsf D_G)$ at $g$
and, for each $p\in P$, the vector space $\mathsf E(p)$ is spanned by the 
classes $\bar\xi(p)+I_p^\tg(\mathsf D_G)$ for all star
sections $\xi$ of $\mathsf D_G$, 
we find that the vector bundle $\mathsf E$ is spanned at each point $g\in G$ by the elements
$\xi(g)+(\mathsf D_G\cap\ker\TT\tg)(g)$
for all $\xi\in\mathsf S(\mathsf D_G)$. To simplify the notation, we  write 
$\tilde\xi$ for the image of the section $\xi\in\mathsf S(\mathsf D_G)$ in $\mathsf E$, 
and $\widetilde{\mathsf S(\mathsf D_G)}$ for the set of these special sections of $\mathsf E$.
By the considerations above,
for any $K\in\mathcal B(G)$ and $\xi\in\mathsf S(\mathsf D_G)$, 
we can define 
$R_K^*\tilde \xi:=\widetilde{R_K^*\xi}$.
If we set in the same manner 
$$\ldr{Z^l}\tilde\xi=\widetilde{\ldr{Z^l}\xi}
\overset{\eqref{der_of_s_section}}=
\widetilde{\mathcal L_Z\xi+{\sigma_{\xi,Z}}^l}=\widetilde{\mathcal L_Z\xi}
\in\widetilde{\mathsf S(\mathsf D_G)}$$

for all $\xi\in\mathsf S(\mathsf D_G)$ and $Z\in\Gamma(AG)$, we find 
for any $g\in G$:
\begin{align*}
\left.\frac{d}{dt}\right\an{t=0}\left(R_{\Exp(tZ)}^*\tilde \xi\right)(g)&=
\left.\frac{d}{dt}\right\an{t=0}\widetilde{R_{\Exp(tZ)}^*\xi}(g)\\
&=\widetilde{\left.\frac{d}{dt}\right\an{t=0}R_{\Exp(tZ)}^*\xi}(g)\\
&=\widetilde{\ldr{Z^l}\xi}(g)=\widetilde{\mathcal L_Z\xi}(g)=\ldr{Z^l}\tilde\xi(g).
\end{align*}

\bigskip

\emph{Assume here that the bracket on sections of $\lie A(\mathsf D_G)$
induced by $\mathsf D_G$ as in Theorem \ref{lie_algebroid_dual} has image 
in $\Gamma(\lie A(\mathsf D_G))$.}
Recall from \cite{Courant90a} that the integrability 
of a Dirac structure is measured 
by the  \emph{Courant $3$-tensor $\mathsf T$}
defined   on sections of $\mathsf D_G$ by
\[\mathsf T\Bigl(\xi,\eta,\zeta
\Bigr)=\left\langle [\xi,\eta],\zeta\right\rangle
\]
for all $\xi,\eta,\zeta\in\Gamma(\mathsf D_G)$.
We show that 
$\mathsf T$ induces a tensor 
$\tilde{\mathsf T}\in\Gamma\left(\bigwedge^3\mathsf E^*\right)$.  
By the considerations above, we can define a $3$-tensor $\tilde{\mathsf T}$ by its values on the elements
of $\widetilde{\mathsf S(\mathsf D_G)}$.
Set $$\tilde{\mathsf T}\left(\tilde \xi, \tilde\eta,\tilde \zeta\right)=\mathsf T(\xi,\eta,\zeta)$$
for all $\xi,\eta,\zeta\in\mathsf S(\mathsf D_G)$. To see that $\tilde{\mathsf T}$ is well-defined, 
choose $g\in G$ and $\sigma_g\in(\mathsf D_G\cap\ker\TT\tg)(g)$.
Then there exists $\sigma\in\Gamma(I^\tg(\mathsf D_G))$ such that
$\sigma^l(g)=\sigma_g$. Then, since 
$[\xi,\eta]$ is a star section $\s$-related to $[\bar\xi,\bar\eta]_\star$, we have:
$$\mathsf T\left(\xi(g),\eta(g),\sigma_g\right)
=\mathsf T\left(\xi,\eta,\sigma^l\right)(g)=
\left\langle[\xi,\eta], \sigma^l\right\rangle(g)=
\left\langle[\bar\xi,\bar\eta]_\star, \sigma\right\rangle(\s(g))=0$$
since $[\bar\xi,\bar\eta]_\star(\s(g))\in \mathsf D_G(\s(g))$ by hypothesis.

\bigskip

For any bisection $K\in\mathcal B(G)$, we can define
the $3$-tensor $R_K^*\tilde{\mathsf T}$ by 
$$\left(R_K^*\tilde{\mathsf T}\right)\left(\tilde\xi,\tilde \eta,\tilde\zeta\right)
=R_K^*\left(\tilde{\mathsf T}
\left(R_{K\inv}^*\tilde\xi,R_{K\inv}^*\tilde \eta,R_{K\inv}^*\tilde\zeta\right)\right)$$
for all $\xi,\eta,\zeta\in\mathsf S(\mathsf D_G)$.
For $Z\in\Gamma(AG)$, we 
can thus define $\ldr{Z^l}\tilde{\mathsf T}$ by
$$\ldr{Z^l}\tilde{\mathsf T}=\left.\frac{d}{dt}\right\an{t=0}
R_{\Exp(tZ)}^*\tilde{\mathsf T}.
$$
We have 
$$R_{\Exp(tZ)}^*\left(\tilde{\mathsf T}\left(\tilde \xi,\tilde \eta,\tilde \zeta\right)\right)
=\left(R_{\Exp(tZ)}^*\tilde{\mathsf T}\right)
\left(R_{\Exp(tZ)}^*\tilde \xi,R_{\Exp(tZ)}^*\tilde \eta,R_{\Exp(tZ)}^*\tilde \zeta\right)$$
for all $\xi,\eta,\zeta\in\mathsf S(\mathsf D_G)$,
which yields easily 
\begin{align}
\ldr{Z^l}\left(\mathsf T\left(\xi,\eta,\zeta\right)\right)
&=\left(\ldr{Z^l}\tilde{\mathsf T}\right)\left(\tilde\xi,\tilde \eta,\tilde\zeta\right)
+{\mathsf T}(\mathcal L_Z\xi,\eta,\zeta)
+{\mathsf T}(\xi,\mathcal L_Z\eta,\zeta)
+{\mathsf T}(\xi,\eta,\mathcal L_Z\zeta)\label{eq}.
\end{align}

We will need the following  lemma for the proof of the main result of this subsection.
\begin{lemma}\label{lem}
Let $(G\rr P,\mathsf D_G)$ be a Dirac groupoid.
Consider three star sections $\xi\sim_\s\bar\xi$, $\eta\sim_\s\bar\eta$ and 
$\zeta\sim_\s\bar\zeta$ of $\mathsf D_G$. Then, 
if $[\bar\xi,\bar\eta]_\star\in\Gamma(\lie A(\mathsf D_G))$, we have
\begin{equation*}
\left[\bar\zeta,\left[\bar\xi,\bar\eta\right]_\star\right]_\star
=[\zeta,[\xi,\eta]]\an{P}
\end{equation*}
where the bracket on the right-hand side is the Courant-Dorfman 
bracket on sections of $TG\oplus T^*G$.
\end{lemma}

\begin{proof}
If $\bar\tau:=\left[\bar\xi,\bar\eta\right]_\star\in\Gamma(\lie A(\mathsf D_G))$, 
then there exists a star section $\tau$ of $\mathsf D_G$ such that
$\tau\sim_\s\bar\tau$.
Since $[\xi,\eta]\an{P}=\left[\bar\xi,\bar\eta\right]_\star=\tau\an{P}$
and $\TT\s\left([\xi,\eta](g)\right)=\left[\bar\xi,\bar\eta\right]_\star(\s(g))
=\TT\s\left(\tau(g)\right)$ for all 
$g\in G$, there exists then 
a section $\chi$ of $\ker\TT\s$ that is vanishing on $P$ such that
$\tau-[\xi,\eta]=\chi$. 
Choose $p\in P$. Then, on a neighborhood $U$ of $p$ in $G$, 
the section $\chi$ of $\ker\TT\s$ can be written 
$\chi=\sum_{i=1}^nf_i\sigma_i^r$ with 
functions $f_1,\ldots,f_n\in C^\infty(U)$ that vanish on $P\cap U$ and basis
sections $\sigma_1,\ldots,\sigma_n$ of $(\ker\TT\s)\an{P}$ on $U\cap P$.
We have then
\begin{align*}
\left[\bar\zeta,\left[\bar\xi,\bar\eta\right]_\star\right]_\star&=[\zeta,\tau]\an{P}=[\zeta,[\xi,\eta]+\chi]\an{P}
=[\zeta,[\xi,\eta]]\an{P}+[\zeta,\chi]\an{P}.
\end{align*}
 If we write $\zeta=(X_\zeta,\omega_\zeta)$, we can compute using \eqref{leibniz}
\begin{align*}
[\zeta,\chi]
=\sum_{i=1}^n\left(f_i\left[\zeta,\sigma_i^r\right]+X_\zeta(f_i)\sigma_i^r\right).
\end{align*}
Since $X_\zeta$ is tangent to $P$ on $P$ and $f_1,\ldots,f_n$ vanish on $P$,
we have $X_\zeta(f_i)\an{P}=0$. This shows that
$[\zeta,\chi](p)=0$ for all $p\in P$.
Hence, we have $\left[\bar\zeta,\left[\bar\xi,\bar\eta\right]_\star\right]_\star(p)
=[\zeta,[\xi,\eta]](p)$.
\end{proof}

Now we can  state the main theorem of this section.
\begin{theorem}\label{int_crit_thm}
Let $(G\rr P,\mathsf D_G)$ be a Dirac groupoid. Assume that $G\rr P$ is $\tg$-connected.
Then the Dirac structure $\mathsf D_G$ is closed if and only if:
\begin{enumerate}
\item 
the induced bracket as in Theorem \ref{lie_algebroid_dual} has image 
in $\Gamma(\lie A(\mathsf D_G))$ and satisfies the Jacobi identity 
$$\left[\bar\zeta,\left[\bar\xi,\bar\eta\right]_\star\right]_\star+
\left[\bar\eta,\left[\bar\zeta,\bar\xi\right]_\star\right]_\star
+\left[\bar\xi,\left[\bar\eta,\bar\zeta\right]_\star\right]_\star=0\qquad \text{ for all } 
\bar\xi,\bar\eta,\bar\zeta\in\Gamma(\lie A(\mathsf D_G))$$
and
\item the induced bracket on sections of $I^\s(\mathsf D_G)$
as in Proposition \ref{other_algebroid}
has image in $\Gamma(I^\s(\mathsf D_G))$.
\end{enumerate}
\end{theorem}

\begin{example}
\begin{enumerate}
\item In the Poisson case, the integrability is ensured by the fact that $A^*G$ is an 
algebroid, 
since it defines then a linear Poisson structure $\pi_A$ on $AG$
such that $\pi_A^\sharp:T^*A\to TA$ is a Lie algebroid morphism, that integrates modulo 
canonical identifications to $\pi_G^\sharp:T^*G\to TG$.
\item In the case of a multiplicative  $2$-form, 
closedness of the  form is ensured by the 
condition of compatibility of the corresponding 
IM-$2$-form with the Lie algebroid bracket \cite{BuCaOr09}. 
This is exactly
the same as the condition on $I^\tg(\mathsf D_G)$ to be a Lie algebroid.
\item In the pair Dirac groupoid case, we have seen that 
$I^\tg(\mathsf D_M\ominus\mathsf D_M)$
and $\lie A(\mathsf D_M\ominus\mathsf D_M)$ are both isomorphic to $\mathsf D_M$.
We know already that $\mathsf D_M\ominus\mathsf D_M$ is closed if and only if 
$\mathsf D_M$ is.
\end{enumerate}
\end{example}

\begin{proof}[of Theorem \ref{int_crit_thm}]
We have shown in Theorem \ref{lie_algebroid_dual} and Proposition
\ref{other_algebroid} that the integrability of 
$\mathsf D_G$ implies 1) and 2).

\medskip

Conversely, assume that 1) and 2) hold. We will show that $\mathsf D_G$ is closed.
First choose $p\in P$. The fiber $\mathsf D_G(p)$
of $\mathsf D_G$ over $p$ is spanned by the values of the sections 
in $I^\s(\mathsf D_G)$
 defined at $p$, and the values at $p$ of the star sections of $\mathsf D_G$.
Since the brackets on sections of $I^\s(\mathsf D_G)$ and 
$\lie A(\mathsf D_G)$ have values in 
$\Gamma(I^\s(\mathsf D_G))$, and respectively $\Gamma(\lie A(\mathsf D_G))$, we find 
for all $\sigma_1,\sigma_2,\sigma_3\in\Gamma(I^\s(\mathsf D_G))$
and $\xi_1,\xi_2,\xi_3\in\mathsf S(\mathsf D_G)$:
\begin{align*}
\mathsf T(\sigma_1^r,\sigma_2^r,\sigma_3^r)(p)
&=\langle [\sigma_1,\sigma_2]^r,\sigma_3^r\rangle(p)
=\langle [\sigma_1,\sigma_2]_{I^\s(\mathsf D_G)},\sigma_3\rangle(p)=0\\
\mathsf T(\sigma_1^r,\sigma_2^r,\xi_3)(p)
&=\langle [\sigma_1,\sigma_2]_{I^\s(\mathsf D_G)},\bar\xi_3\rangle(p)=0\\
\mathsf T(\xi_1,\xi_2,\sigma^r_3)(p)
&=\langle [\bar\xi_1,\bar\xi_2]_\star,\sigma_3\rangle(p)=0\\
\mathsf T(\xi_1,\xi_2,\xi_3)(p)
&=\langle [\bar\xi_1,\bar\xi_2]_\star,\bar\xi_3\rangle(p)=0.
\end{align*}
Hence, $\mathsf T$
vanishes over points in $P$.

\medskip

Consider the $3$-tensor $\tilde{\mathsf T}$ induced on the sections of 
$\mathsf E=\mathsf D_G/(\mathsf D_G\cap\ker\TT\tg)$ by $\mathsf T$
and choose $\xi_1,{\xi_2},{\xi_3}\in\mathsf S(\mathsf D_G)$ and $Z\in\Gamma(AG)$.
We show that $(\ldr{Z^l}\tilde{\mathsf T}) ({\xi_1},{\xi_2},{\xi_3})=0$.
A long but straightforward computation yields
\begin{align*}
Z^l(\mathsf T({\xi_1},{\xi_2},{\xi_3}))
=\,&\mathsf T\left(\mathcal L_Z{\xi_1},{\xi_2},{\xi_3}\right)
+\s^*\left(\alpha_{{\xi_1},Z}\left(\left[\bar X_{\xi_2},\bar X_{\xi_3}\right]\right)\right)
+\s^*\left(\bar\omega_{\xi_1}\left(Y_{{\xi_3},Y_{{\xi_2},Z}}-Y_{{\xi_2},Y_{{\xi_3},Z}}\right)\right)\\
&+\s^*\left(\bar X_{\xi_1}\left(\alpha_{{\xi_2},Z}\left(\bar X_{\xi_3}\right)\right)\right)
+\s^*\left(\bar X_{\xi_1}\left(\bar\omega_{\xi_2}\left(Y_{{\xi_3},Z}\right)\right)\right)\\
&+\s^*\left(\alpha_{{\xi_2},Y_{{\xi_1},Z}}\left(\bar X_{\xi_3}\right)\right)
+\s^*\left(\bar\omega_{\xi_2}\left(Y_{{\xi_3},Y_{{\xi_1},Z}}\right)\right)\\
&+{\rm c.p.}.
\end{align*}
By \eqref{eq}, this leads to
\begin{align*}
(\ldr{Z^l}\tilde{\mathsf T})\left(\tilde{\xi_1},\tilde{\xi_2},\tilde{\xi_3}\right)(g)
=\,&\s^*\left(\alpha_{{\xi_1},Z}\left(\left[\bar X_{\xi_2},\bar X_{\xi_3}\right]\right)\right)
+\s^*\left(\bar\omega_{\xi_1}\left(Y_{{\xi_3},Y_{{\xi_2},Z}}-Y_{{\xi_2},Y_{{\xi_3},Z}}\right)\right)\\
&+\s^*\left(\bar X_{\xi_1}\left(\alpha_{{\xi_2},Z}\left(\bar X_{\xi_3}\right)\right)\right)
+\s^*\left(\bar X_{\xi_1}\left(\bar\omega_{\xi_2}\left(Y_{{\xi_3},Z}\right)\right)\right)\\
&+\s^*\left(\alpha_{{\xi_2},Y_{{\xi_1},Z}}\left(\bar X_{\xi_3}\right)\right)
+\s^*\left(\bar\omega_{\xi_2}\left(Y_{{\xi_3},Y_{{\xi_1},Z}}\right)\right)\\
&+{\rm c.p.}\\
=\,&\left(\ldr{Z^l}\tilde{\mathsf T}\right)\left(\tilde{\xi_1},\tilde{\xi_2},\tilde{\xi_3}\right)(\s(g))
\end{align*} 
for all $g\in G$.
But since $\mathsf T$ vanishes on the units by hypothesis, we find 
by \eqref{eq} that 
$$(\ldr{Z^l}\tilde{\mathsf T})(\tilde{\xi_1},\tilde{\xi_2},\tilde{\xi_3})(\s(g))
=Z^l(\mathsf T({\xi_1},{\xi_2},{\xi_3}))(\s(g)).
$$
A quick computation shows 
that the cotangent part
of $[{\xi_1},[{\xi_2},{\xi_3}]]+[{\xi_2},[{\xi_3},{\xi_1}]]+[{\xi_3},[{\xi_1},{\xi_2}]]$ 
is equal to $\dr(\mathsf T({\xi_1},{\xi_2},{\xi_3}))$, see
also \cite{Courant90a}. 
Using this and Lemma \ref{lem}, we find finally
\begin{align*}
(\ldr{Z^l}\tilde{\mathsf T})(\tilde{\xi_1},\tilde{\xi_2},\tilde{\xi_3})(g)
&=Z^l\left(\mathsf T({\xi_1},{\xi_2},{\xi_3})\right)(\s(g))\\
&=\Bigl\langle
\bigl([{\xi_1},[{\xi_2},{\xi_3}]]+[{\xi_2},[{\xi_3},{\xi_1}]]+[{\xi_3},[{\xi_1},{\xi_2}]]\bigr)(\s(g)),
(Z,0)(\s(g))\Bigr\rangle\\
&=\Bigl\langle\bigl
([\bar{\xi_1},[\bar{\xi_2},\bar{\xi_3}]_\star]_\star
+[\bar{\xi_2},[\bar{\xi_3},\bar{\xi_1}]_\star]_\star
+[\bar{\xi_3},[\bar{\xi_1},\bar{\xi_2}]_\star]_\star\bigr)(\s(g)),
(Z,0)(\s(g))\Bigr\rangle\\
&=0
\end{align*}
since by condition 1), $[\cdot\,,\cdot]_\star$ satisfies the Jacobi identity.

\medskip

Hence, we have shown that $\ldr{Z^l}\tilde{\mathsf T}=0$ for all $Z\in\Gamma(AG)$.
This yields that 
$R_{\Exp(tZ)}^*\tilde{\mathsf T}=\tilde{\mathsf T}$ for all $Z\in\Gamma(AG)$
and $t\in \R$ where this makes sense and hence, since $G$ is $\tg$-connected
and $\tilde{\mathsf T}$ vanishes on the units, we find $\tilde{\mathsf T}=0$.
Thus, $\mathsf T=0$ on $G$
and the proof is finished.
\end{proof}

\begin{remark}
For $Z\in\Gamma(AG)$, define
$\nabla_Z:\Gamma(\mathsf E)\to\Gamma(\mathsf E)$ by 
$\nabla_Z\tilde\xi=\widetilde{\mathcal L_Z\xi}$
for all $\tilde \xi\in\widetilde{\mathsf S(\mathsf D_G)}$,
and $\nabla_Z\left(\sum_{i=1}^nf_i\tilde\xi_i\right)
=\sum_{i=1}^n\left(Z^l(f_i)\tilde\xi_i+f_i\nabla_Z\tilde\xi_i\right)$ 
for all $f_1,\ldots,f_n\in C^\infty(G)$ and $\tilde\xi_1, \ldots,\tilde\xi_n\in
\widetilde{\mathsf S(\mathsf D_G)}$.
Then $\nabla_Z$ is a derivative endomorphism of $\mathsf E$ over $Z^l$.
The map $\Gamma(AG)\to\Gamma(\mathcal D(\mathsf E))$, $Z\to\nabla_Z$
is a \emph{derivative representation of $AG$ on $\mathsf E$
associated to the action 
of $AG$ on $\s:G\to P$, $Z\in\Gamma(AG)\to Z^l$}
(see  \cite{KoMa02}).
\end{remark}

\subsection{The Courant algebroid associated to a closed Dirac groupoid}
\label{Courant}
The dual space of $\mathfrak A(\mathsf D_G)$ can be identified with
$\mathsf P_G\an{P}/ \mathfrak A(\mathsf D_G)^\perp$. 
Since $$\mathfrak A(\mathsf D_G)^\perp=\mathsf D_G\an{P}+ (TP\oplus A^*G)
=I^\tg(\mathsf D_G)\oplus (TP\oplus A^*G)$$ and 
$$\mathsf P_G\an{P}=(TP\oplus A^*G)+\ker\TT\tg\an{P},$$
we have 
$$\left(\mathfrak A(\mathsf D_G)\right)^*\simeq \frac{\ker\TT\tg\an{P}}{I^\tg(\mathsf D_G)}.$$

Since $\mathsf D_G\an{P}\subseteq  \mathfrak A(\mathsf
D_G)\oplus \ker\TT\tg\an{P}$,
we have $I^\s(\mathsf D_G)\subseteq  \mathfrak A(\mathsf
D_G)\oplus \ker\TT\tg\an{P}$ and  the quotient
$$\lie B(\mathsf D_G):=\frac{\mathfrak A(\mathsf D_G)\oplus \ker\TT\tg\an{P}}{I^\s(\mathsf D_G)}$$
is a smooth vector bundle over $P$.
Consider the map
$$\Psi: \ker\TT\tg\an{P}\oplus \mathfrak A(\mathsf D_G)\to \lie B(\mathsf D_G),$$ 
$$\Psi\left(\sigma+\bar\xi\right)=
\sigma+\bar\xi+I^\s(\mathsf D_G)$$
for all $\sigma\in\Gamma(\ker\TT\tg\an{P})$ and
$\bar\xi\in\Gamma(\lie A(\mathsf D_G))$. 
If 
$\Psi(\sigma+\bar\xi)=I^\s(\mathsf D_G)$, 
then we have 
$\sigma+\bar\xi\in\Gamma(\mathsf D_G\an{P})$
and hence $\sigma\in\Gamma(\mathsf D_G\an{P})$
since $\bar\xi\in\Gamma(\mathsf D_G\an{P})$. 
This yields 
$\sigma\in\Gamma(I^\tg(\mathsf D_G))$ 
and the map $\Psi$ factors 
to a vector bundle homomorphism
$$\bar\Psi:(\mathfrak A(\mathsf D_G))^*\oplus \mathfrak A(\mathsf D_G)\to  \lie B(\mathsf D_G)$$
over the identity $\Id_P$.

Set $r=\operatorname{rank} I^\s(\mathsf D_G)$, $n=\dim G$. Then we have also
$r=\operatorname{rank} I^\tg(\mathsf D_G)$
and we can compute
$\operatorname{rank} \lie B(\mathsf D_G)=
\operatorname{rank}(\ker\TT\tg)+\operatorname{rank}(\mathfrak A(\mathsf D_G))
-\operatorname{rank}I^\s(\mathsf D_G)
=n+(n-r)-r=2n-2r$.
We have also
$\operatorname{rank} ((\mathfrak A(\mathsf D_G))^*\oplus \mathfrak A(\mathsf D_G))
=n-r+ n-r=2n-2r$
and since $\bar\Psi$ is surjective, it is 
hence a vector bundle isomorphism.

\medskip

Since $\left(\ker\TT\tg\an{P}\oplus\mathfrak A(\mathsf D_G)\right)^\perp
=\left(\ker\TT\tg\an{P}+\mathsf D_G\an{P}\right)^\perp
=I^\s(\mathsf D_G)$, 
the bracket $\langle\cdot\,,\cdot\rangle$
restricts to  a non degenerate symmetric bracket 
on $\lie B(\mathsf D_G)$, that will also be written $\langle\cdot\,,\cdot\rangle$
in the following.

\medskip

Recall from Example \ref{ex_dual_of_lie_algebroid_poisson}
 that if $(G\rr P,\mathsf D_{\pi_G})$ is a Poisson groupoid, the bundle
$\mathfrak A(\mathsf D_{\pi_G})$ is equal to
$\graph(\pi_G^\sharp\an{A^*G})\simeq A^*G$, 
$\mathsf a_\star(\xi)=\pi_G^\sharp(\xi)$ for all $\xi\in\Gamma(A^*G)$ and the bracket 
on 
sections of $\mathfrak A(\mathsf D_G)$ is the bracket induced by the Poisson structure.
In the same manner, we have 
$(\mathfrak A(\mathsf D_G))^*=\ker\TT\tg\an{P}/I^\tg(\mathsf D_G)
= \ker\TT\tg\an{P}/\graph\left(\pi^\sharp_G\an{(T^\s_PG)^\circ}\right)$
which is isomorphic as a vector bundle  to 
$AG$.
The vector bundle $\lie B(\mathsf D_{\pi_G})$ is thus the vector bundle underlying 
the Courant algebroid 
associated to $(G\rr P, \pi)$
We will study this example in more detail in Example \ref{bialgebroid_of_Poisson}, where we will
show that $\lie B(\mathsf D_{\pi_G})$ carries a natural Courant algebroid structure 
that makes it isomorphic as a Courant algebroid to $AG\oplus A^*G$.

We show here  that if the Dirac  groupoid 
$(G\rr P,\mathsf D_G)$ is closed, 
the vector bundle $\lie B(\mathsf D_G)\to P$
always inherits the structure of a Courant algebroid 
from the ambient standard Courant algebroid 
structure of $\mathsf P_G$.

Because of the special case of Poisson groupoids, we have chosen the notation 
$\lie B(\mathsf D_G)$: this Courant algebroid will play the role 
of  the ``Lie bialgebroid of the Dirac groupoid
$(G\rr P,\mathsf D_G)$''.

\begin{theorem}\label{theorem_courant_algebroid}
Let $(G\rr P, \mathsf D_G)$ be a closed Dirac groupoid and $$\lie B(\mathsf
D_G)=\frac{\mathfrak A(\mathsf D_G)\oplus \ker\TT\tg\an{P}}{I^\s(\mathsf D_G)}\to P$$
the associated vector bundle over $P$.
Set $\mathsf b:\lie B(\mathsf D_G)\to TP$, 
$\mathsf b(v_p,\alpha_p)=T_p\s v_p$.
Define
$$[\cdot\,,\cdot]:\Gamma(\lie B(\mathsf D_G))\times \Gamma(\lie B(\mathsf D_G))
\to \Gamma(\lie B(\mathsf D_G))$$ by 
\begin{align*}
\left[\bar\xi+\sigma+I^\s(\mathsf D_G), \bar\eta+\tau+I^\s(\mathsf D_G)\right]
=\left.\left[\xi+\sigma^l, 
\eta+\tau^l\right]\right\an{P}
+I^\s(\mathsf D_G)
\end{align*}
for all $\sigma,\tau\in\Gamma\left(\ker\TT\tg\an{P}\right)$,
$\bar\xi,\bar\eta\in\Gamma(\lie A(\mathsf D_G))$
and star sections $\xi\sim_\s\bar\xi$, $\eta\sim_\s\bar\eta$ of $\mathsf D_G$, 
where the bracket on the right-hand side of this equation
is the Courant bracket on sections of the Courant algebroid 
$\mathsf P_G$.
This bracket is well-defined
and  $(\lie B(\mathsf D_G), \mathsf b, [\cdot\,,\cdot], \langle\cdot\,,\cdot\rangle)$
is a Courant algebroid.
\end{theorem}

\begin{proof}
This proof can be found with more detailed computations in \cite{thesis}.
The  map $\mathsf b$ is well-defined since $T_p\s v_p=0_p$ for all
$(v_p,\alpha_p)\in I^\s(\mathsf D_G)$.
We show that the bracket on sections of $\lie B(\mathsf D_G)$
is well-defined, that is, that it has image in $\Gamma(\lie B(\mathsf D_G))$ 
and does not depend on the choice
of the sections  $\bar\xi+\sigma$ 
and $\bar\eta+\tau$ representing
$\bar\xi+\sigma+I^\s(\mathsf D_G)$
and $\bar\eta+\tau+I^\s(\mathsf D_G)$.
We have, writing 
$\sigma^l=\left(X^l,\s^*\alpha\right)$ and 
$\tau^l=\left(Y^l,\s^*\beta\right)$,
\begin{align}
&\left[(X_\xi+X^l, \theta_\xi+\s^*\alpha), (X_\eta+Y^l, \theta_\eta+\s^*\beta)\right]\nonumber\\
=&\left[(X_\xi, \theta_\xi), (X_\eta,\theta_\eta)\right]
+\ldr{X^l}\left(X_\eta,\theta_\eta\right)-\ldr{Y^l}\left(X_\xi,\theta_\xi\right)\nonumber\\
&\quad +\Biggl([X,Y]_{AG}^l, \s^*\left(\ldr{\bar{X_\xi}+\mathsf a(X)}\beta-
\ldr{\bar{X_\eta}+\mathsf a(Y)}\alpha\right)\label{important_bracket}\\
&\hspace*{3.5cm}+\s^*\left(\frac{1}{2}\dr
\bigl(\alpha(\bar{X_\eta}+\mathsf a(Y))+\bar{\theta_\xi}(Y)
-\beta(\bar{X_\xi}+\mathsf a(X))-\bar{\theta_\eta}(X) \bigr)\right)\Biggr).\nonumber
\end{align}
By Theorems \ref{lie_algebroid_dual} and  \ref{lie_der_of_xi_section}, the restriction
of this to $P$ is a section of $\mathfrak A(\mathsf D_G)\oplus\ker\TT\tg\an{P}$
and depends on the choice of the star sections 
$(X_\xi,\theta_\xi)$, $(X_\eta,\theta_\eta)$
only by sections of $I^\s(\mathsf D_G)$.

Choose $\sigma\in\Gamma(I^\s(\mathsf D_G))$. Then we have 
for all $(Y^l, \s^*\beta)
\in\Gamma(\ker\TT\tg)$:
\begin{align*}
&\left[\sigma^r, (Y^l, \s^*\beta)\right]
=\left(0, \ldr{X_\sigma^r}(\s^*\beta)-\ldr{Y^l}(\tg^*\alpha_\sigma)
+\frac{1}{2}\dr\left((\tg^*\alpha_\sigma)(Y^l)-(\s^*\beta)(X_\sigma^r)\right)\right)=(0, 0).
\end{align*}
We have used Lemma \ref{der_of_kers}.
If $(X_\nu,\theta_\nu)$ is a section of $\mathsf D_G$ that 
is a star section $\s$-related to 
$(\bar X_\nu,\bar \theta_\nu)=(0,0)$, then we have 
$(X_\nu,\theta_\nu)\in\Gamma(\mathsf D_G\cap\ker\TT\s)$
and we find smooth sections $\sigma_1^r, \ldots,
\sigma_k^r\in\Gamma(\mathsf D_G\cap\ker\TT\s)$
and functions $f_1,\ldots,f_k\in C^\infty(G)$
such that 
$(X_\nu,\theta_\nu)=\sum_{i=1}^kf_i\sigma_i^r$. Then we get easily
for all $(Y^l, \s^*\beta)
\in\Gamma(\ker\TT\tg)$
\begin{align*}
\left[(X_\nu,\theta_\nu), \left(Y^l, \s^*\beta\right)\right]
&=-\sum_{i=1}^k
Y^l(f_i)\sigma_i^r,
\end{align*}
which is a section of $\mathsf D_G\cap \ker\TT\s$.
Hence, the restriction to $P$
of $\left[\left( X_\nu,\theta_\nu\right), \left(Y^l, \s^*\beta\right)\right]$
is a section of $I^\s(\mathsf D_G)$.
In the same manner,  for  $i=1,\ldots,k$ and any star section
$(X_\xi,\theta_\xi)\sim_\s(\bar X_\xi,\bar{\theta_\xi})$,
\begin{align*}
\left[\sigma_i^r, (X_\xi,\theta_\xi)\right]&=
\left(\left[X_{\sigma_i}^r,X_\xi\right], \ldr{X_{\sigma_i}^r}\theta_\xi-\ip{X_\xi}\dr (\tg^*\alpha_{\sigma_i})\right)
\end{align*}
 is a section of $\mathsf D_G$ since $\mathsf D_G$ is closed.
Since $X_\xi\sim_\s\bar{X_\xi}$ and 
$X_{\sigma_i}^r\sim_\s 0$, we have 
$\left[X_{\sigma_i}^r,X_\xi\right]\sim_\s [0, \bar{X_\xi}]=0$ and we compute
for any $Y\in\Gamma(AG)$, 
using the equality $ (\tg^*\alpha_{\sigma_i})(X_\xi)=-\theta_\xi(X_{\sigma_i}^r)$:
\begin{align*}
 \left(\ldr{X_{\sigma_i}^r}\theta_\xi-\ip{X_\xi}\dr (\tg^*\alpha_{\sigma_i})\right)\left(Y^l\right)
=&\,X_{\sigma_i}^r\left(\theta_\xi\left(Y^l\right)\right)-\theta_\xi\left(\left[X_{\sigma_i}^r,Y^l\right]\right)
-X_\xi\left(\left(\tg^*\alpha_{\sigma_i}\right)\left(Y^l\right)\right) \\
&\hspace*{2cm}+Y^l\left((\tg^*\alpha_{\sigma_i})(X_\xi)\right)
+(\tg^*\alpha_{\sigma_i})\left(\left[X_\xi, Y^l\right]\right)\\
=&\,X_{\sigma_i}^r(\s^*(\bar\theta_\xi(Y)))+(\ldr{Y^l}(\tg^*\alpha_{\sigma_i}))(X_\xi)=0.
\end{align*}
This shows that 
$\left[\left(X_{\sigma_i}^r,\tg^*\alpha_{\sigma_i}\right), \left(X_\xi,\theta_\xi\right)\right]$
is a section of $\ker\TT\s\cap\mathsf D_G$ for $i=1,\ldots,k$.
Then we get as above 
\begin{align*}
[( X_\nu,\theta_\nu), (X_\xi, \theta_\xi)]
&=\sum_{i=1}^k\left(f_i\left[\sigma_i^r, (X_\xi, \theta_\xi)\right]
-X_\xi(f_i)\sigma_i^r
+\frac{1}{2}\langle \sigma_i^r, (X_\xi, \theta_\xi)\rangle
(0,\dr f_i)\right)\\
&=\sum_{i=1}^k
\left(f_i\left[\sigma_i^r, (X_\xi, \theta_\xi)\right]-X_\xi(f_i)\sigma_i^r
\right),
\end{align*}
which is a section of $\mathsf D_G\cap\ker\TT\s$ by the considerations
above. Hence, its restriction 
to $P$ is a section of $I^\s(\mathsf D_G)$.
If $\bar\xi+\sigma\in\Gamma(I^\s(\mathsf D_G))$, then as above,
we find that $\sigma\in\Gamma(I^\tg(\mathsf D_G))$. The section
$\xi+\sigma^l-(\bar\xi+\sigma)^r$ is then a section of $\mathsf D_G$ 
that is  a star section $\s$-related to $0$. Since by the considerations above, we know 
that 
$$\left.\left[  \xi+\sigma^l-(\bar\xi+\sigma)^r, \eta+\tau^l\right]\right\an{P}
\in\Gamma(I^\s(\mathsf D_G))$$
and 
$$\left.\left[  (\bar\xi+\sigma)^r, \eta+\tau^l\right]\right\an{P}
\in\Gamma(I^\s(\mathsf D_G))$$
for all star sections
 $\eta\sim_\s\bar\eta$ of $\mathsf D_G$ and $\tau\in\Gamma((\ker(\TT\tg))\an{P})$, 
we have shown that
 the bracket
does not depend on the choice of the representatives for
$(\bar X_\xi,\bar\theta_\xi)+(X,(\s^*\alpha)\an{P})+I^\s(\mathsf D_G)$ and 
$(\bar X_\eta,\bar\theta_\eta)+(Y,(\s^*\beta)\an{P})+I^\s(\mathsf D_G)$.
\medskip

We show now that 
$(\lie B(\mathsf D_G), \mathsf b, [\cdot\,,\cdot], \langle\cdot\,,\cdot\rangle)$
is a Courant algebroid. The map $$\mathcal D: C^\infty(P)\to \Gamma(\lie B(\mathsf D_G))$$
is simply given by 
$$\mathcal D f=\frac{1}{2}(0,\s^*\dr f)+I^\s(\mathsf D_G)$$
since $$\left\langle \mathcal D f, \overline{(v_p,\alpha_p)}\right\rangle
=\frac{1}{2}\mathsf b\left(\overline{(v_p,\alpha_p)}\right)(f)
=\frac{1}{2}T_p\s v_p(f)$$ for all $\overline{(v_p,\alpha_p)}\in \lie B_p(\mathsf D_G)$.
We  check all the Courant algebroid axioms.
Choose  $$(\bar{X_\xi}+X, \bar\theta_\xi+(\s^*\alpha)\an{P})+I^\s(\mathsf D_G),\qquad 
(\bar{X_\eta}+Y, \bar\theta_\eta+(\s^*\beta)\an{P})+I^\s(\mathsf D_G)$$
and 
$$(\bar{X_\tau}+Z, \bar\theta_\tau+(\s^*\gamma)\an{P})
+I^\s(\mathsf D_G)\in\Gamma(\lie B(\mathsf D_G))$$
and let $f$ be  an  arbitrary element of $C^\infty(P)$. 
\begin{enumerate}
\item By \eqref{important_bracket}, the bracket 
$$\left[(X_\xi+X^l, \theta_\xi+\s^*\alpha), (X_\eta+Y^l, \theta_\eta+\s^*\beta)\right]$$
can be taken as the section extending 
$$\left[(\bar{X_\xi}+X, \bar\theta_\xi+(\s^*\alpha)\an{P})+I^\s(\mathsf D_G), 
(\bar{X_\eta}+Y, \bar\theta_\eta+(\s^*\beta)\an{P})+I^\s(\mathsf D_G)\right]$$
to compute its bracket with $(\bar{X_\tau}+Z, \bar\theta_\tau+(\s^*\gamma)\an{P})
+I^\s(\mathsf D_G)$.
Since $\mathsf P_G$ is a Courant algebroid, we have 
\begin{align*}
&\left[\left[(X_\xi+X^l, \theta_\xi+\s^*\alpha), (X_\eta+Y^l, \theta_\eta+\s^*\beta)\right], 
(X_\tau+Z^l, \theta_\tau+\s^*\gamma)\right]+\rm{c.p.}\\
=&\frac{1}{6}\left(0, \dr\left(\left\langle 
\left[\left(X_\xi+X^l, \theta_\xi
+\s^*\alpha\right), \left(X_\eta+Y^l, \theta_\eta+\s^*\beta\right)\right], 
\left(X_\tau+Z^l, \theta_\tau+\s^*\gamma\right)
\right\rangle\right)\right)\\
&\hspace*{12cm}+\rm{c.p.}.
\end{align*}
If we write $e_{\xi,X,\alpha}$
for 
$(\bar{X_\xi}+X, \bar\theta_\xi+(\s^*\alpha)\an{P})+I^\s(\mathsf D_G)$, etc,
this can be checked to restrict to
\begin{align*}
&[e_{\xi,X,\alpha}, [e_{\eta,Y,\beta}, 
e_{\tau,Z,\gamma}]]
+[e_{\eta,Y,\beta}, [e_{\tau,Z,\gamma}, 
e_{\xi,X,\alpha}]]
+[e_{\tau,Z,\gamma}, [e_{\xi,X,\alpha}, 
e_{\eta,Y,\beta}]]\\
=&\frac{1}{3}\mathcal D\left(\left\langle[e_{\xi,X,\alpha}, e_{\eta,Y,\beta}], 
e_{\tau,Z,\gamma}
\right\rangle+
\left\langle[e_{\eta,Y,\beta}, e_{\tau,Z,\gamma}], 
e_{\xi,X,\alpha}\right\rangle
+\left\langle[e_{\tau,Z,\gamma}, e_{\xi,X,\alpha}], 
e_{\eta,Y,\beta}\right\rangle
\right)
\end{align*} on $P$.
\item We have 
\begin{align*}
\mathsf b[e_{\xi,X,\alpha}, e_{\eta,Y,\beta}]
&=\left.T\s\left[X_\xi+X^l, X_\eta+Y^l\right]\right\an{P}
=\left[T\s\left(X_\xi+X^l\right), T\s\left(X_\eta+Y^l\right)\right]\\
&=\left[\bar{X_\xi}+\mathsf a(X), \bar{X_\eta}+\mathsf a(Y)\right]
=\left[\mathsf b(e_{\xi,X,\alpha}), \mathsf b(e_{\eta,Y,\beta})\right].
\end{align*}
\item We compute 
\begin{align*}
&\left[e_{\xi,X,\alpha}, f\cdot e_{\eta,Y,\beta}\right]\\
=\,\,&\left.\left[(X_\xi+X^l, \theta_\xi+\s^*\alpha), 
(\s^*f)\cdot (X_\eta+Y^l, \theta_\eta+\s^*\beta)\right]\right\an{P}
+I^\s(\mathsf D_G)\\
=\,\,&\Bigl((\s^*f)\left[(X_\xi+X^l, \theta_\xi+\s^*\alpha), 
 (X_\eta+Y^l, \theta_\eta+\s^*\beta)\right]\\
&\qquad+(X_\xi+X^l)(\s^*f)\cdot \left(X_\eta+Y^l, \theta_\eta+\s^*\beta\right)\\
&\qquad\left. -\left\langle(X_\xi+X^l, \theta_\xi+\s^*\alpha), 
 (X_\eta+Y^l, \theta_\eta+\s^*\beta)\right\rangle\cdot \frac{1}{2}(0,\dr(\s^*f))
\Bigr)\right\an{P}
+I^\s(\mathsf D_G)\\
=\,\,&f[e_{\xi,X,\alpha}, e_{\eta,Y,\beta}]
+\mathsf b(e_{\xi,X,\alpha})(f) \cdot e_{\eta,Y,\beta}\\
&\qquad\qquad\qquad -\langle\left(\bar{X_\xi}+X, \bar\theta_\xi+\s^*\alpha\right), 
 \left(\bar{X_\eta}+Y, \bar\theta_\eta+\s^*\beta\right)\rangle\mathcal D f\\
=\,\,&f\left[e_{\xi,X,\alpha}, e_{\eta,Y,\beta}\right]
+\mathsf b\left(e_{\xi,X,\alpha}\right)(f) \cdot e_{\eta,Y,\beta}
-\left\langle e_{\xi,X,\alpha}, e_{\eta,Y,\beta}\right\rangle\mathcal D f.
\end{align*}
\item We have obviously $\mathsf b\circ \mathcal D=0$.
\item Finally, since $\mathsf P_G$ is a Courant algebroid, 
the corresponding equality for sections of $\mathsf P_G$ 
yields easily, with the same computations as in the previous points
\begin{align*}
\mathsf b(e_{\xi,X,\alpha})\langle e_{\eta,Y,\beta}, e_{\tau, Z,\gamma}\rangle
=&\bigl\langle [e_{\xi,X,\alpha}, e_{\eta,Y,\beta}]+
\mathcal D\langle e_{\xi,X,\alpha}, e_{\eta,Y,\beta}\rangle,
e_{\tau, Z,\gamma}\bigr\rangle\\
&\qquad +\bigl\langle e_{\eta,Y,\beta},
[e_{\xi,X,\alpha}, e_{\tau, Z,\gamma}]+
\mathcal D\langle e_{\xi,X,\alpha}, e_{\tau, Z,\gamma}\rangle\bigr\rangle.
\end{align*}
\end{enumerate}
\end{proof}

\begin{example}\label{bialgebroid_of_Poisson}
We see in this example that in the special case of a Poisson groupoid
$(G\rr P, \mathsf D_{\pi_G})$, the obtained
Courant algebroid is isomorphic to the Courant algebroid defined 
by the Lie bialgebroid associated to $(G\rr P,\pi_G)$, 
see \cite{LiWeXu97}, \cite{LiWeXu98}.
This shows how the Courant algebroid structure 
on $AG\oplus A^*G$ induced by the Lie bialgebroid 
of the Poisson groupoid $(G\rr P, \pi_G)$
 can be related to the standard Courant algebroid 
structure on $\mathsf P_G=TG\oplus T^*G$.


Recall that the Courant algebroid $\mathsf E_{\pi_G}=AG\oplus A^*G$
associated to the Lie bialgebroid $(AG, A^*G)$ of 
$(G\rr P,\pi_G)$ is endowed with the anchor 
$\rho:AG\oplus A^*G\to TP$ defined by $\rho(v_p,\alpha_p)=\mathsf a(v_p)+\pi_G^\sharp(\alpha_p)$
for all $p\in P$ and $(v_p,\alpha_p)\in A_pG\times A^*_pG$ and
the symmetric bracket $\langle\cdot\,,\cdot\rangle$ 
defined by $\langle(v_p,\alpha_p), (w_p,\beta_p)\rangle
=\alpha_p(w_p)+\beta_p(v_p)$
for all  $p\in P$ and $(v_p,\alpha_p), (w_p,\beta_p)\in A_pG\times A^*_pG$.
Its Courant bracket is given by 
\begin{align}\label{Courant_bracket_bialgebroid}
[(X,\xi),(Y,\eta)]=&\Bigl([X,Y]_{AG}+\mathfrak L^*_{\xi}Y-\mathfrak L^*_{\eta}X-\frac{1}{2}\mathsf d_*(\xi(Y)-\eta(X)),
\nonumber\\
&\hspace*{3cm}[\xi,\eta]_\star+\mathfrak L_{X}\eta-\mathfrak L_{Y}\xi+\frac{1}{2}\mathsf d (\xi(Y)-\eta(X))
\Bigr),
\end{align}
where, if $\bar X_\xi:=\pi_G^\sharp(\xi)\in\Gamma(TP)$ for $\xi\in\Gamma(A^*G)$,
\begin{align*}
\mathfrak L^*_{\xi}Y\in\Gamma(AG),
&\quad \tau(\mathfrak L^*_{\xi}Y)=\bar X_\xi(\tau(Y))-[\xi,\tau](Y) \qquad \forall \tau\in\Gamma(A^*G)\\
\mathfrak L_{X}\eta\in\Gamma(A^*G),
&\quad (\mathfrak L_{X}\eta)(Z)=\mathsf a(X)(\eta(Z))-\eta([X,Z]_{AG}) \qquad \forall Z\in\Gamma(AG)
\end{align*}
and, for any $f\in C^\infty(P)$, 
\begin{align*}
(\mathsf d_*f)(\tau)&=\bar X_\tau (f)=X_\tau (\s^*f)\an{P}=-\tau(X_{\s^*f}), 
\qquad  \text{ hence }  \mathsf d_*f=-X_{\s^*f}\an{P}\\
(\mathsf d f)(Z)&=\mathsf a(Z)(f)=Z(\s^*(f)), \qquad  \text{ hence }  \mathsf d f=\dr(\s^*f)\an{AG}=\hat\s(\dr(\s^*f)).
\end{align*}

The isomorphism $\Psi:AG\oplus A^*G\to \lie B(\mathsf D_{\pi_G})$
is given by 
$$\Psi(X(p),\xi(p))=(X+\bar X_\xi, \xi)(p)+I_p^\s(\mathsf D_{\pi_G}),$$ 
with inverse 
$$\Psi\inv((v_p,\alpha_p)+I_p^\s(\mathsf D_{\pi_G}))
=(v_p-\pi_G^\sharp(\alpha_p), \hat\s(\alpha_p)).
$$
The verification of the equalities $\Psi\inv\circ\Psi=\Id_{\mathsf E_{\pi_G}}$ and 
$\Psi\circ\Psi\inv=\Id_{\lie B(\mathsf D_{\pi_G})}$ is easy, in the same spirit
as the corresponding equalities in the following example. 
In \cite{thesis}, it is shown that
\begin{align*}
&\Psi\inv\left([\Psi(X,\xi), \Psi(Y,\eta)]_{\lie B(\mathsf D_{\pi_G})}\right)
=\,[(X,\xi), (Y,\eta)]
\end{align*}
for $(X,\xi),(Y,\eta)\in\Gamma(AG\oplus A^*G)$. Since the computations are long, but
straightforward, we omit them here.
\end{example}

\begin{example}\label{ex_presymplectic_courant}
Consider a  Lie groupoid $G\rr P$  endowed with a closed
multiplicative $2$-form $\omega_G\in\Omega^2(G)$.
The Courant algebroid $\lie B(\mathsf D_{\omega_G})$
is given here by 
$$\lie B(\mathsf D_{\omega_G})=\left(\graph(\omega_G^\flat\an{TP}:TP\to A^*G)
+\ker\TT\tg\an{P}\right)/\graph\left(\omega_G^\flat\an{T^\s_PG}:T^\s_PG\to (T_P^\tg G)^\circ\right).$$
We show that it is isomorphic as a Courant algebroid 
to the standard Courant algebroid $\mathsf P_P=TP\oplus T^*P$.
For this, consider the maps
$$\Lambda:\lie B(\mathsf D_{\omega_G})\to TP\oplus T^*P, \qquad 
\Lambda\left(\overline{(v_p,\alpha_p)}\right)
=(T_p\s v_p, \beta_p),$$ where $(T_p\s)^*\beta_p=\alpha_p-\omega_G^\flat(v_p)$,
and 
$$\Lambda\inv:TP\oplus T^*P\to \lie B(\mathsf D_{\omega_G}), \qquad 
\Lambda\inv\left(v_p,\alpha_p\right)
=\overline{(\epsilon(v_p), (T_p\s)^*\alpha_p+\omega_G^\flat(\epsilon(v_p)))}.
$$
Note that $\Lambda$ is well-defined: 
if $(v_p,\alpha_p)\in \lie A_p(\mathsf D_{\omega_G})+\ker\TT\tg$, 
then $\hat\tg(\alpha_p)=\omega_G^\flat(T_p\tg v_p)$
and hence 
$\hat\tg(\alpha_p-\omega_G^\flat(v_p))=\hat\tg(\alpha_p)-\omega_G^\flat(T_p\tg v_p)=0_p$.
Thus, the covector $\alpha_p-\omega_G^\flat(v_p)$ can be written
$(T_p\s)^*\beta_p$ with some $\beta_p\in T_p^*P$.
(For simplicity, we will identify elements $\beta_p$ of 
$T_p^*P$  with $(T_p\s)^*\beta_p\in (T_p^\s G)^\circ\subseteq T_p^*G$
and $v_p\in T_pP$ with $\epsilon(v_p)\in T_pP\subseteq T_pG$
in the following.)
Furthermore, if $u_p\in T_p^\s G$, we have 
$\Lambda\left(\overline{(u_p,  \omega_G^\flat(u_p))}\right)
=(T_p\s u_p, \omega_G^\flat(u_p)-\omega_G^\flat(u_p))=(0_p, 0_p)$.
The map $\Lambda\inv$ has image in $\lie B(\mathsf D_{\omega_G})$
because 
for any $(v_p,\alpha_p)\in\mathsf P_P(p)$, we have 
\begin{align*}
\left(v_p, (T_p\s)^*\alpha_p+\omega_G^\flat(v_p)\right)
=&\TT\tg\left(v_p, \omega_G^\flat(v_p)\right)+(0, (T_p\s)^*\alpha_p)
 \in 
\lie A_p(\mathsf D_{\omega_G})+(\ker\TT\tg)_p.
\end{align*}
Choose now $p\in P$, 
$\overline{(v_p,\alpha_p)}\in\lie B_p(\mathsf D_{\omega_G})$
and compute
\begin{align*}
(\Lambda\inv\circ \Lambda)\left(\overline{(v_p,\alpha_p)}\right)
&=\Lambda\inv(T_p\s v_p, \alpha_p-\omega_G^\flat(v_p))\\
&=\overline{(T_p\s v_p, \alpha_p-\omega_G^\flat(v_p)+\omega_G^\flat(T_p\s v_p))}\\
&=\overline{(v_p,\alpha_p)+(T_p\s v_p-v_p, \omega_G^\flat(T_p\s v_p-v_p))}
=\overline{(v_p,\alpha_p)}.
\end{align*}
In the same manner, if $(v_p,\alpha_p)\in\mathsf P_P(p)$, we have 
$$(\Lambda\circ \Lambda\inv)\left(v_p,\alpha_p\right)
=\Lambda\left(\overline{(v_p,\alpha_p+\omega_G^\flat(v_p))}\right)
=(T_p\s v_p,\alpha_p+\omega_G^\flat(v_p)-\omega_G^\flat(v_p))
=(v_p,\alpha_p)$$ since $v_p\in T_pP$.
The equality $\mathsf b\circ \Lambda\inv=\pr_{TP}$ is immediate.

Now if $(\bar X,\bar \alpha), (\bar Y,\bar \beta)\in\Gamma(\mathsf P_P)$, 
we choose $X,Y\in \mx(G)$ such that
$X\sim_\s \bar X$, $X\an{P}=\bar X$,  $Y\sim_\s \bar Y$,
$Y\an{P}=\bar Y$ and we compute
\begin{align*}
&\Lambda\left[\Lambda\inv(\bar X,\bar \alpha), \Lambda\inv(\bar Y,\bar \beta)\right]\\
=&\,\Lambda\left(\left.\left([X,Y], 
\omega_G^\flat([X,Y])+\ldr{X}(\s^*\bar\beta)-\ldr{Y}(\s^*\bar\alpha)
+\frac{1}{2}\dr((\s^*\bar\alpha)(Y)-(\s^*\bar\beta)(X))
\right)\right\an{P}\!\!\!\!+I^\s(\mathsf D_{\omega_G})\right)\\
=&\,\Lambda\left(\overline{\left([\bar X, \bar Y], 
\omega_G^\flat([\bar X, \bar Y])+\s^*\left(\ldr{\bar X}\bar\beta-\ldr{\bar Y}\bar\alpha
+\frac{1}{2}\dr(\bar\alpha(\bar Y)-\bar\beta(\bar X))\right)\right)}\right)\\
=&\left([\bar X, \bar Y], 
\ldr{\bar X}\bar\beta-\ldr{\bar Y}\bar\alpha
+\frac{1}{2}\dr(\bar\alpha(\bar Y)-\bar\beta(\bar X))\right),
\end{align*}
and we recover the standard Courant bracket 
of the two sections
$(\bar X,\bar \alpha), (\bar Y,\bar \beta)\in\Gamma(\mathsf P_P)$.
\end{example}

\begin{example}\label{iso_pair_case}
Consider the  pair Dirac  groupoid
$(M\times M\,\rr\, M, \mathsf D_M\ominus\mathsf D_M)$ associated
to a closed Dirac manifold $(M,\mathsf D_M)$ (see Example \ref{ex_pair_Dirac_Lie}).
The vector bundle 
$\lie B(\mathsf D_M\ominus\mathsf D_M)\to \Delta_M$ 
is defined here by
$$\lie B_{(m,m)}(\mathsf D_M\ominus\mathsf D_M)
=\frac{\lie A_{(m,m)}(\mathsf D_M\ominus\mathsf D_M)+\{0\}\times T_mM\times\{0\}\times T^*_mM}
{\{(v_m,0_m,\alpha_m,0_m)\mid (v_m,\alpha_m)\in\mathsf D_M(m)\}}
$$
for all $m\in M$ (recall that we have computed
$\lie A(\mathsf D_M\ominus\mathsf D_M)$ in Example
\ref{algebroid_of_pair_Dirac}).
Hence, we get an isomorphism 
\begin{equation}\label{iso_Courant_pair}
\Pi:\lie B(\mathsf D_M\ominus\mathsf D_M)\to TM\times_M T^*M, \quad 
\overline{(v_m,w_m,\alpha_m,\beta_m)}\mapsto (w_m,\beta_m)
\end{equation}
over $\pr_1:\Delta_M\to M$,
with inverse 
$$\Pi\inv: TM\times_M T^*M \to \lie B(\mathsf D_M\ominus\mathsf D_M), \quad 
(w_m,\beta_m)\mapsto \overline{(0_m,w_m,0_m,\beta_m)}.$$

The Courant bracket on $\lie B(\mathsf D_M\ominus\mathsf D_M)$ is easily seen
to correspond via this isomorphism to the standard Courant bracket 
on $\mathsf P_M=TM\times_M T^*M$ (and hence,  does not depend on $\mathsf D_M$).
\end{example}

\begin{remark}
Consider a Dirac groupoid as in Theorem 
\ref{main}. Set $N:=G/\mathsf{G_0}$ and $Q=:P/(TP\cap\mathsf{G_0})$. Then the Courant algebroid
$TG\oplus T^*G$ projects under $\pr$ to the Courant algebroid $TN\times_N T^*N$.
We have a map $\TT\pr:TG\oplus\mathsf{P_1}\to TN\times_N T^*N$,
$(v_g,\alpha_g)\mapsto (T_g\pr(v_g), \alpha_{[g]})$, 
where $\alpha_{[g]}$ is such that
$\alpha_g=(T_g\pr)^*\alpha_{[g]}$. By definition of the reduced Dirac structure $\pr(\mathsf D_G)=D_\pi$,
the restriction
of this map to $\left(\lie A(\mathsf D_G)\oplus\ker(\TT\tg)\an{P}\right)\cap(TG\oplus\mathsf{P_1})$
has image $\lie A(\mathsf D_\pi)+(\ker\TT\tg_N)\an{Q}$. Furthermore,
we find that $(v_p,(T_p\pr)^*\alpha_{[p]})\in I^\s_p(\mathsf D_G)+\mathsf{G_0}(p)\times 0_p$
if and only if $(T_p\pr v_p,\alpha_{[p]})\in I^{\s_N}_{[p]}(\mathsf D_\pi)$.

Hence, the map $\TT\pr$ factors to
a map $\lie B(\mathsf D_G)\to \lie B(\mathsf D_\pi)$. It is straightforward to check that
this is a morphism of Courant algebroids.
\end{remark}

\subsection{Induced action of the group of bisections on $\lie B(\mathsf D_G)$}
\label{subsec_action}
We show here how  the action of $G$ on $\lie g/\lie g_0\times \lie p_1$ found in \cite{Jotz11a}
in the Lie group can be generalized to the setting of Dirac groupoids.
In this section, the Dirac groupoids that we consider
are not necessarily closed. Hence, the vector bundle 
$\lie B(\mathsf D_G)$ exists, but does not necessarily have a Courant 
algebroid structure.  

We begin with a lemma, which will also be useful in the following section 
about Dirac homogeneous spaces.
\begin{lemma}\label{sections_of_D'}
Let $(G\rr P, \mathsf D_G)$ be a Dirac groupoid and $(v_p,\alpha_p)\in
\lie A_p(\mathsf D_G)\oplus (\ker\TT\tg)\an{P}$ for some $p\in
P$. If  $\,\TT\tg(v_p,\alpha_p)=(\bar{X_\xi}(p), \bar\theta_\xi(p))\in\lie A_p(\mathsf D_G)$, then 
$(v_p,\alpha_p)=(\bar{X_\xi}(p), \bar\theta_\xi(p))+(u_p,(T_p\s)^*\gamma_p)$ with some
$u_p\in A_pG$ and $\gamma_p\in T_p^*P$ and 
$$ \left((X_\xi,\theta_\xi)(g)\right)\star (v_p,\alpha_p)
=(X_\xi(g)+T_{p}L_gu_{p}, \theta_\xi(g)+(T_g\s)^*\gamma_p)
$$
for any $g\in\s\inv(p)$ and $(X_\xi,\theta_\xi)\in\Gamma(\mathsf D_G)$ such that
$(X_\xi,\theta_\xi)\sim_\s(\bar X_\xi,\bar\theta_\xi)$.
\end{lemma}
\begin{proof}
If $(v_p,\alpha_p)\in
\lie A_p(\mathsf D_G)\oplus(\ker\TT\tg)\an{P}$, 
then $\TT\tg(v_p,\alpha_p)\in \lie A_p(\mathsf D_G)$ and hence 
$\TT\tg(v_p,\alpha_p)=(\bar{X_\xi},\bar\theta_\xi)(p)$ for some 
section $(\bar{X_\xi},\bar\theta_\xi)\in \Gamma(\lie A_p(\mathsf D_G))$. The difference
$(v_p,\alpha_p)-(\bar{X_\xi},\bar\theta_\xi)(p)=(v_p,\alpha_p)-\TT\tg(v_p,\alpha_p)$
is then an element of $(\ker\TT\tg)\an{P}$ and there exists 
$\gamma_p\in T_p^*P$ such that $(v_p,\alpha_p)-(\bar{X_\xi},\bar\theta_\xi)(p)
=(u_p, (T_p\s)^*\gamma_p)$ if we set 
$u_p=v_p-T_p\tg v_p$.

Since the section $(X_\xi,\theta_\xi)\in\Gamma(\mathsf D_G)$ is a pair
that is  a star section $\s$-related to $(\bar{X_\xi}, \bar\theta_\xi)$,
the product 
$(X_\xi,\theta_\xi)(g)\star (v_p,\alpha_p)$ is defined for any $g\in\s\inv(p)$.

We compute, using a bisection $K$ through $g$,
\[X_\xi(g)\star v_p= X_\xi(g)+T_pL_Kv_p-T_pL_K(T_p\tg v_p)
=X_\xi(g)+T_pL_K(u_p)=X_\xi(g)+T_pL_g(u_p).
\]
For the first equality, we have used the formula proved in \cite{Xu95}, see
also
\cite{Mackenzie05}.

We have also, for any $v_g=v_g\star(T_g\s v_g)\in T_gG$ 
\begin{align*}
(\theta_\xi(g)\star \alpha_p)(v_g)&=(\theta_\xi(g)\star \alpha_p)(v_g\star(T_g\s v_g))
=\theta_\xi(g)(v_g)+\alpha_p(T_g\s v_g)\\
&=\theta_\xi(g)(v_g)+(\alpha_p-\hat\tg(\alpha_p))(T_g\s v_g)
=\theta_\xi(g)(v_g)+(T_p\s^*\gamma_p)(T_g\s v_g)\\
&=\left(\theta_\xi(g)+(T_g\s^*\gamma_p)\right)(v_g).
\end{align*}
\end{proof}

\begin{theorem}\label{action}
Let $(G\rr P, \mathsf D_G)$ be a Dirac groupoid.
Choose a bisection $K\in\mathcal B(G)$ and consider
\[
r_K: \mathfrak A(\mathsf D_G)\oplus\ker\TT\tg\an{P}\to\lie B(\mathsf D_G)
\]
\[r_K(v_p,\alpha_p)=\Bigl(T_{K(p)\inv}R_K(v_{K(p)\inv}\star v_p), 
\left(T_{(\s\circ K)(p)}R_K\inv\right)^*
(\alpha_{K(p)\inv}\star\alpha_p)\Bigr)+I^\s_{(\s\circ K)(p)}(\mathsf D_G),\]
where 
$(v_{K(p)\inv},\alpha_{K(p)\inv})\in\mathsf D_G(K(p)\inv)$
is such that $$\TT\s\left(v_{K(p)\inv},\alpha_{K(p)\inv}\right)=\TT\tg(v_p,\alpha_p).$$
The map $r_K$ is well-defined and induces the right translation by $K$,
\[\begin{array}{cccc}
\rho_K: &\lie B(\mathsf D_G)&\to& \lie B(\mathsf D_G)\\
&(v_p,\alpha_p)+I_p^\s(\mathsf D_G)&\mapsto &r_K(v_p,\alpha_p).
\end{array}\]
The map $\rho:\mathcal B(G)\times \Gamma(\lie B(\mathsf D_G))\to
\Gamma(\lie B(\mathsf D_G))$ is a right action.
\end{theorem} 

For the proof of this theorem, we will need the following lemma, which proof is straighforward.
\begin{lemma}
Let $G\rr P$ be a Lie groupoid. 
Choose $g,h\in G$ and  $K\in\mathcal B(G)$. 
Choose $(v_h,\alpha_h)\in \mathsf P_G(h)$, 
$\left(v_{g},\alpha_{g}\right)\in \mathsf P_G(g)$ such that 
$\TT\s\left(v_{g},\alpha_{g}\right)=\TT\tg\left(v_h,\alpha_h\right)$.
Then  
\begin{equation}\label{util_tangent}
T_{g\star h}R_K(v_{g}\star v_h)
=v_{g}\star \left(T_{h}
R_{K}v_h\right)
\end{equation}
and 
\begin{equation}\label{util_cotangent}
 \alpha_{g}\star \left(\left(T_{R_K(h)}R_{K}\inv\right)^*\alpha_h\right)
=\left(T_{R_K(g\star h)}R_{K}\inv\right)^*
(\alpha_{g}\star \alpha_h).
\end{equation}
\end{lemma}

\begin{proof}[of Theorem \ref{action}]
First, we check that the map $r_K$ is well-defined, that is, 
that it
has image in $\lie B(\mathsf D_G)$ and does not depend on the choices made.

Choose $p\in P$, $(v_p,\alpha_p)\in 
\lie A_p(\mathsf D_G)\oplus(\ker\TT\tg)_p$ and $K\in\mathcal B(G)$.
Set $K(p)=g$.
Since the map $r_K$ is linear in every fiber of 
$\lie A(\mathsf D_G)\oplus(\ker\TT\tg)\an{P}$,
 it suffices to show that the image
of $(0_p,0_p)$ is $I^\s_{\s(g)}(\mathsf D_G)$
for any choice of  $(v_{g\inv},\alpha_{g\inv})\in\mathsf D_G(g\inv)$
such that $\TT\s(v_{g\inv},\alpha_{g\inv})=(0_p, 0_p)$
to prove that it is well-defined.
Using \eqref{util_tangent} and \eqref{util_cotangent}, we get
\begin{align*}
T_{g\inv}R_K (v_{g\inv}\star 0_p)&
=v_{g\inv}\star\left(T_pR_K 0_p\right)
=v_{g\inv}\star 0_{g}\\
(\alpha_{g\inv}\star 0_p)\circ T_{\s(g)}R_{K\inv}&=
\alpha_{g\inv}\star\left(0_p\circ T_{g}R_{K\inv}\right)
=\alpha_{g\inv}\star0_g.
\end{align*}
Thus, we have shown that
$$r_K(0_p, 0_p)
=(v_{g\inv},\alpha_{g\inv})\star(0_g, 0_g) \in\mathsf
D_G(\s(g))\cap\ker\TT\s=I_{\s(g)}^\s(\mathsf D_G).$$

\medskip

Choose next
$(v_p,\alpha_p)\in 
\mathfrak A(\mathsf D_G)\oplus(\ker\TT\tg)\an{P}$ such that 
 $(v_p,\alpha_p)\in I^\s_p(\mathsf D_G)$, that 
is, such that $\overline{(v_p,\alpha_p)}=0$ in $\lie B_p(\mathsf D_G)$. Choose 
$(v_{g\inv},\alpha_{g\inv})\in
\mathsf D_G(g\inv)$ such that
$\TT\s(v_{g\inv},\alpha_{g\inv})=\TT\tg(v_p,\alpha_p)$.
Then we have 
$T_{g\inv}R_K(v_{g\inv}\star v_p)=T_{g\inv}R_K(v_{g\inv}\star v_p\star 0_p)=v_{g\inv}\star v_p\star 0_g$,
since $T_p\s v_p=0$. We have also 
$\hat\s(\alpha_p)=0$, and by \eqref{util_cotangent}:
$$(T_{\s(g)}R_{K}\inv)^*(\alpha_{g\inv}\star\alpha_p)
=(T_{\s(g)}R_{K}\inv)^*(\alpha_{g\inv}\star\alpha_p\star 0_p)
=\alpha_{g\inv}\star\alpha_p\star 0_g.
$$
Thus, 
$r_K(v_p,\alpha_p)
=(v_{g\inv},\alpha_{g\inv})\star (v_p,\alpha_p)
\star(0_g,0_g)
\in I_{\s(K(p))}^\s(\mathsf D_G)$.
The map $\rho_K:\lie B(\mathsf D_G)\to \lie B(\mathsf D_G)$ is consequently
well-defined.

\medskip

We show now that $\rho: \mathcal B(G)\times \lie B(\mathsf D_G)\to  \lie B(\mathsf D_G)$
defines an action of the group of bisections of $G\rr P$ on $\Gamma(\lie B(\mathsf D_G))$.

Choose $K,L\in\mathcal B(G)$, $p\in P$ 
and  $\overline{(v_p,\alpha_p)}$
in $\lie B_p(\mathsf D_G)$.
Set $K(p)=g$ and choose  a pair $(v_{g\inv},\alpha_{g\inv})\in\mathsf D_G(g\inv)$
such that $\TT\s(v_{g\inv},\alpha_{g\inv})=\TT\tg(v_p,\alpha_p)$.
Set also $h:=L(\s(g))$
and choose  $(w_{h\inv},\beta_{h\inv})\in\mathsf D_G(h\inv)$
such that $\TT\s(w_{h\inv},\beta_{h\inv})
=\TT\tg\left(v_{g\inv},\alpha_{g\inv}\right)$. Then we have 
 $(K\star L)(p)=g\star h$ and we compute, using
\eqref{util_tangent} and \eqref{util_cotangent}:
\begin{align*}
&\rho_L\left(\rho_K\left(\overline{(v_p,\alpha_p)}\right)\right)\\
=\,&\rho_L\left( \left(T_{g\inv}R_K,\left(T_{\s(g)}R_{K}\inv\right)^*\right) 
\left(\left(v_{g\inv},\alpha_{g\inv}\right)\star
(v_p,\alpha_p)\right)\right)\\
=&\left(T_{(g\star h)\inv}R_{K\star L},(T_{\s(g\star h)}R_{K\star L}\inv)^*\right) 
\left((w_{h\inv},\beta_{h\inv})
\star \left(v_{g\inv},\alpha_{g\inv}\right)\star(v_p,\alpha_p)\right)+I^\s_{\s(h)}(\mathsf D_G)\\
=\,& \rho_{K\star L}\left(\overline{(v_p,\alpha_p)}\right)
\end{align*}
since $(w_{h\inv},\beta_{h\inv})\star\left(v_{g\inv},\alpha_{g\inv}\right)$
is an element of $\mathsf D_G((g\star h)\inv)$ satisfying
$$\TT\s\left((w_{h\inv},\beta_{h\inv})\star\left(v_{g\inv},\alpha_{g\inv}\right)
\right)
=\TT\s\left(v_{g\inv},\alpha_{g\inv}\right)
=(v_p,\alpha_p).$$
\phantom{bla}
\end{proof}

\begin{example}
Consider a Poisson groupoid $(G\rr P, \pi_G)$. We will compute 
the action of the bisections $\mathcal B(G)$ 
on $\lie B(\mathsf D_{\pi_G})\simeq AG\oplus A^*G$.

Choose $K\in\mathcal B(G)$,
 $p\in P$ and $\Psi(X(p),\xi(p))\in \lie B_p(\mathsf D_{\pi_G})$,
$(X(p),\xi(p))\in A_pG\times A^*_pG$
(recall that $\Psi$ has been defined in Example \ref{bialgebroid_of_Poisson}).
If $K(p)=g$ and $\theta_\xi\in\Omega^1(G)$ is a star section that is $\s$-related
to $\xi$, then  $\rho_K(\Psi(X(p),\xi(p)))$ 
is given by
\begin{align*}
\rho_K(\Psi(X(p),\xi(p)))&=\left(T_{g\inv}R_K(\pi_G^\sharp(\theta_\xi(g\inv))
+X^l(g\inv)),
((R_K\inv)^*\theta_\xi)(\s(g))\right)
\end{align*}
and corresponds to
\begin{align*}
&(\Psi\inv\circ\rho_K\circ\Psi)(X(p),\xi(p))\\
=\,&\left(T_{g\inv}R_K(\pi_G^\sharp(\theta_\xi(g\inv))+X^l(g\inv))-\pi_G^\sharp((R_K\inv)^*\theta_\xi(\s(g))),
\hat\s\left(((R_K\inv)^*\theta_\xi)(\s(g))\right)\right)\\
=\,&\Bigl( T_{p}(L_{g\inv}\circ R_K)X(p)+
T_{g\inv}R_K(\pi_G^\sharp(\theta_\xi(g\inv)))-\pi_G^\sharp((R_K\inv)^*\theta_\xi(\s(g))),\\
&\hspace*{5cm}(T_{\s(g)}(L_g\circ R_{K\inv}))^*\xi(p)\Bigr).
\end{align*}
Note that by the general theorem, 
this does not depend on the choice of 
$\theta_\xi$. 
In the case of a Poisson Lie group, we recover the action 
of $G$ on the Lie bialgebroid, see \cite{Drinfeld93}.
In the case of a trivial Poisson groupoid, i.e., with $\pi_G=0$, this 
is simply the pair of maps on $AG$ and $A^*G$ generalizing
$\Ad$ and $\Ad^*$ in the Lie group case. 
\end{example}

\begin{example}
Consider a  Lie groupoid $G\rr P$  endowed with a closed
multiplicative $2$-form $\omega_G\in\Omega^2(G)$.
 We will compute 
the action of the bisections $\mathcal B(G)$ 
on $\lie B(\mathsf D_{\omega_G})\simeq TP\oplus T^*P$.

Choose a bisection $K\in\mathcal B(G)$,  a vector 
$\Lambda\inv(v_p,\alpha_p)
=\overline{(v_p, (T_p\s)^*\alpha_p+\omega_G^\flat(v_p))}
\in\lie B_p(\mathsf D_{\omega_G})$,
 $(v_p,\alpha_p)\in\mathsf P_P(p)$
(recall Example \ref{ex_presymplectic_courant})
and set $K(p)=g\in G$. Then 
we have 
$$\TT\tg(v_p,(T_p\s)^*\alpha_p+\omega_G^\flat(v_p))
=(v_p, \omega_G^\flat(v_p))\in\lie A_p(\mathsf D_{\omega_G}).$$
Hence, any vector $v_{g\inv}\in T_{g\inv}G$ such that
$T_{g\inv}\s v_{g\inv}=v_p$  leads to 
$(v_{g\inv},\omega_G^\flat(v_{g\inv}))\in\mathsf D_{\omega_G}(v_{g\inv})$
and $\TT\s(v_{g\inv},\pi_G^\flat(v_{g\inv}))=\TT\tg(v_p, (T_p\s)^*\alpha_p+\omega_G^\flat(v_p))$.

The vector $\rho_K\left(\Lambda\inv(v_p,\alpha_p)\right)
\in\lie B_{\s(g)}(\mathsf D_{\omega_G})$
is then given by
\begin{align*}
&\rho_K\left(\Lambda\inv(v_p,\alpha_p)\right)\\
=&\left(T_{g\inv}R_K(v_{g\inv}\star v_p), 
(T_{\s(g)}R_K\inv)^*\left(\omega_G^\flat(v_{g\inv})\star
\left((T_p\s)^*\alpha_p+\omega_G^\flat(v_p)\right)\right)\right)\\
=&\left(T_{g\inv}R_Kv_{g\inv}, 
(T_{\s(g)}R_K\inv)^*\left(\omega_G^\flat(v_{g\inv})+(T_{g\inv}\s)^*
\alpha_p\right)\right).
\end{align*}
Thus, we get 
\begin{align*}
&\,(\Lambda\circ\rho_K\circ\Lambda\inv)(v_p,\alpha_p)\\
=\,&\left(T_{g\inv}(\s\circ R_K)v_{g\inv}, (T_{\s(g)}R_K\inv)^*\left(\omega_G^\flat(v_{g\inv})+(T_{g\inv}\s)^*
\alpha_p\right)-\omega_G^\flat\left(T_{g\inv}R_Kv_{g\inv}\right)
\right)\\
=\,&\left(T_p(s\circ K)v_p, \left(T_{\s(g)}(\s\circ K)\inv\right)^*\alpha_p
+(T_{\s(g)}R_K\inv)^*\omega_G^\flat(v_{g\inv})-\omega_G^\flat\left(T_{g\inv}R_Kv_{g\inv}\right)
\right).
\end{align*}
Again, by the general theorem, 
this does not depend on the choice of 
$v_{g\inv}$. 
In the trivial case $\omega_G=0$, 
$\rho_K$ is simply the map
$(T(\s\circ K), (T(\s\circ K)\inv)^*)$ 
induced by the diffeomorphism $s\circ K$ on $\mathsf P_P$. 
\end{example}

\begin{example}\label{action_pair_case}Let $(M,\mathsf D)$ be a smooth Dirac 
manifold 
and consider the pair Dirac  groupoid 
$(M\times M\,\rr\,M, \mathsf D\ominus\mathsf D)$ associated to it.
The set of bisections of $M\times M\,\rr\, M$ is equal to
$\mathcal B(M\times M)=\{\Id_M\}\times\operatorname{Diff}(M)$.
Choose $K=(\Id_M, \phi_K)\in\mathcal B(M\times M)$, 
$p:=(m,m)\in\Delta_M$ and 
$\overline{(v_m,w_m,\alpha_m,\beta_m)}\in\lie B_p(\mathsf D_M\ominus \mathsf D_M)$.
Then we have 
$\TT\tg(v_m,w_m,\alpha_m,\beta_m)=(v_m,v_m,\alpha_m,-\alpha_m)\in\lie A_p(\mathsf D_M\ominus\mathsf D_M)$.
Set $n:=\phi_K(m)$. Then we have 
$K(p)\inv=(n,m)$ and 
$(0_n,v_m,0_n,-\alpha_m)\in(\mathsf D_M\ominus\mathsf D_M)(n,m)$
is such that 
$\TT\s(0_n,v_m,0_n,-\alpha_m)=(v_m,v_m,\alpha_m,-\alpha_m)=
\TT\tg(v_m,w_m,\alpha_m,\beta_m)$.

The vector $\rho_K\left(\overline{(v_m,w_m,\alpha_m,\beta_m)}\right)$
is thus given by 
\begin{align*}
\rho_K\left(\overline{(v_m,w_m,\alpha_m,\beta_m)}\right)
&=\overline{(T_{(n,m)}R_K(0_n,w_m), (T_{(n,n)}R_K\inv)^*(0_n,\beta_m))}\\
&=\overline{(0_n, T_m\phi_Kw_m, 0_n, (T_n\phi_K\inv)^*\beta_m)}.
\end{align*}
Recall that $\lie B(\mathsf D_M\ominus\mathsf D_M)$ is isomorphic to 
$\mathsf P_M$ via 
\eqref{iso_Courant_pair}. It is easy to see 
that the action of $\mathcal B(M\times M)$ on 
$\lie B(\mathsf D_M\ominus\mathsf D_M)$
corresponds via this identification to the action
of $\operatorname{Diff}(M)$ on $\mathsf P_M$ 
given by 
$\phi\cdot (v_m,\alpha_m)=(T_m\phi v_m, (T_{\phi(m)}\phi\inv)^*\alpha_m)$
for all $\phi\in \operatorname{Diff}(M)$
and $(v_m,\alpha_m)\in\mathsf P_M(m)$.
\end{example}

\section{Dirac homogeneous spaces and the classification}\label{homogeneous}
In this section, we show that 
the Courant algebroid found in Section \ref{section_geometry_dirac}
is the right ambient Courant algebroid for the classification of the Dirac
homogeneous spaces of a Dirac groupoid.

We prove our main theorem (Theorem \ref{Drinfeld})
about the correspondence between (closed) Dirac homogeneous
spaces of a (closed) Dirac  groupoid  and Lagrangian subspaces (subalgebroids)
of the vector bundle (Courant algebroid) $\lie B(\mathsf D_G)$.
This result generalizes 
the result of \cite{Drinfeld93}
about the Poisson homogeneous spaces of Poisson Lie groups, of
\cite{LiWeXu98} about Poisson homogeneous spaces of Poisson groupoids
and the result in \cite{Jotz11a}
about the Dirac homogeneous spaces of Dirac Lie groups.
To be able to define the notion of a homogeneous Dirac structure on 
a homogeneous space of a Lie groupoid, 
we have to prove the following proposition. The proof is straightforward and will be omitted here.
For more details, see \cite{thesis}.
 
\begin{proposition}\label{action_hom_cot}
Let $G\rr P$ be a Lie groupoid acting on a smooth manifold $M$ with momentum map $\mathsf J:M\to P$. Then there
is an induced action of $TG\,\rr\, TP$ on $T\mathsf J:TM\to TP$.

Assume that $M\simeq G/H$ is a smooth homogeneous space of $G$
and let $q:G\to G/H$ be the projection.
The map 
$\widehat{\mathsf J}:T^*(G/H)\to A^*G$,
 $\widehat{\mathsf J}(\alpha_{gH})=\hat \tg\left((T_gq)^*\alpha_{gH}\right)$
for all $gH\in G/H$ is well-defined and
 $\widehat\Phi:T^*G\times_{A^*G}T^*(G/H)\to T^*(G/H)$ given by
\[\left(\widehat\Phi(\alpha_{g'}, \alpha_{gH})\right)
\left(T_{(g',gH)}\Phi(v_{g'},v_{gH})\right)=
\alpha_{g'}(v_{g'})+\alpha_{gH}(v_{gH})\]
defines an
action of $T^*G\,\rr\, A^*G$ on $\widehat{\mathsf J}:T^*(G/H)\to A^*G$.
\end{proposition}
In the following, we   write $\alpha_g\cdot \alpha_{g'H}$
for $\widehat\Phi(\alpha_g,\alpha_{g'H})$.

\begin{corollary}
If $G/H$ is a smooth homogeneous space of $G\rr P$, there is an induced action
$\TT\Phi=(T\Phi,\widehat\Phi)$
of \[(TG\oplus T^*G)\,\rr \,(TP\oplus A^*G)\]
 on
\[\TT\mathsf J:=T\mathsf J\times \widehat{\mathsf J}:( T(G/H)\times_{G/H} T^*(G/H))
\to (TP\oplus A^*G).\]
\end{corollary}

We will show that the following definition generalizes in a natural manner the notion of Poisson homogeneous
space of a Poisson groupoid.
\begin{definition}\label{def_hom_space}
Let $(G\rr P,\mathsf D_G)$ be a Dirac groupoid, and $G/H$ a smooth homogeneous 
space of $G\rr P$ endowed with a Dirac structure $\mathsf D_{G/H}$.
The pair $(G/H,\mathsf D_{G/H})$ is a \emph{Dirac homogeneous space}
of the Dirac  groupoid $(G\rr P,\mathsf D_G)$ if the induced 
action of 
$(TG\oplus T^*G)\,\rr\, (TP\oplus A^*G)$
on $\,\TT\mathsf  J:( T(G/H)\times_{G/H}  T^*(G/H))
\to (TP\oplus A^*G)$
restricts to an action 
of \[\mathsf D_G\,\rr\, \lie A(\mathsf D_G)
\qquad\text{ 
on }
\qquad 
\TT\mathsf J\an{\mathsf D_{G/H}}:\mathsf D_{G/H}\to \lie A(\mathsf D_G).\]
\end{definition}

\medskip

Let $(G,\mathsf{D}_G)$ be a Dirac Lie group and $H$ a closed connected 
Lie subgroup of $G$. Let 
$G/H=\{gH\mid g\in G\}$ be the homogeneous space defined as the quotient space 
by the right action of $H$ on $G$. Let $q:G\to G/H$ be the quotient map.
For $g\in G$, let $\sigma_g:G/H\to G/H$ be the map
defined by $\sigma_g(g'H)=gg'H$. 

\begin{proposition}\label{Dirac_hom_of_group}
Let $(G,\mathsf{D}_G)$ be a Dirac Lie group and $H$ a closed connected 
Lie subgroup of $G$. Let 
$G/H$ be  endowed with a Dirac structure
$\mathsf{D}_{G/H}$.
The pair $(G/H, \mathsf D_{G/H})$ is a 
\emph{Dirac homogeneous space of $(G,\mathsf{D}_G)$} if and only if
the left action
\[\sigma:G\times G/H\to G/H, \quad \sigma_g(g'H)=gg'H\]
is a forward Dirac map, where $G\times G/H$
is endowed with the product Dirac structure 
$\mathsf D_G\oplus\mathsf D_{G/H}$.
\end{proposition}

\begin{example}
Consider a Poisson homogeneous space $(G/H,\pi)$ of a 
Poisson groupoid $(G\rr P, \pi_G)$, i.e., 
the graph 
$\operatorname{Graph}(\Phi)\subseteq G\times G/H\times \overline{G/H}$
is a coisotropic submanifold (see \cite{LiWeXu98}). 

Consider the Dirac groupoid $(G\rr P, \mathsf D_{\pi_G})$ defined 
by $(G\rr P, \pi_G)$
and the Dirac manifold $(G/H,\mathsf D_{G/H})$, defined by
$\mathsf D_{G/H}=\operatorname{Graph}(\pi^\sharp:T^*(G/H)\to T(G/H))$.
The  verification of the fact 
that $(G/H,\mathsf D_{G/H})$ is a Dirac homogeneous space of the Dirac 
groupoid $(G\rr P,\mathsf D_{\pi_G})$.

Conversely,
we show in a similar manner that if $\TT\Phi$  restricts
to an action of $\mathsf D_{\pi_G}$ on $\mathsf D_\pi$,  
then the graph of the left action of $G$ on $G/H$ is coisotropic.
\end{example}

\begin{example}
Let $(G\rr P, \omega_G)$ be a presymplectic groupoid 
and $H$ a wide subgroupoid of $G\rr P$.
Assume that $G/H$ has a smooth manifold structure such that the projection
$q:G\to G/H$ is a surjective submersion. Let $\omega$ be a closed 
$2$-form on $G/H$ such that
the action $\Phi:G\oplus (G/H)\to G/H$ is a presymplectic groupoid action, i.e.,
$\Phi^*\omega=\pr_{G/H}^*\omega+\pr_G^*\omega_G$ \cite{BuCr05}.
Let $\mathsf D_\omega$ be the graph of the vector bundle map
$\omega^\flat:T(G/H)\to T^*(G/H)$ associated to $\omega$.
It is easy to check that the pair 
$(G/H,\mathsf D_\omega)$ is  a closed Dirac homogeneous space of the closed Dirac
 groupoid $(G\rr P,\mathsf D_{\omega_G})$, 
see Example \ref{ex_presymplectic_groupoid}.
\end{example}

\begin{example}\label{easy_example}
Let $(G\rr P,\mathsf D_G)$ be a Dirac groupoid. 
Then $(\tg:G\to P,\mathsf D_G)$ is a Dirac homogeneous space of $(G\rr P,\mathsf D_G)$.
\end{example}

\subsection{The moment map $\mathsf J$}
In the Poisson case,
if $(G/H,\mathsf D_{G/H})$ is a Poisson
homogeneous space of a Poisson groupoid 
$(G\rr P,\pi_G)$, then the  map $\mathsf J:G/H\to P$ is a Poisson map
(see \cite{LiWeXu98}). This is also
true here under some regularity conditions on the characteristic distributions
of the involved Dirac structures. 

\begin{theorem}
Let $(G\rr P,\mathsf D_G)$ be a  Dirac groupoid
such that Theorem \ref{induced_Dirac_on_P} holds and
$(G/H,\mathsf D_{G/H})$ a Dirac homogeneous space of $(G\rr P,\mathsf D_G)$.
Assume  that the map 
$\mathsf J\an{\mathsf{G_0}^{G/H}}:\mathsf{G_0}^{G/H}\to \mathsf{G_0}\cap TP$
is surjective in every fiber, where 
$\mathsf{G_0}^{G/H}$ is the characteristic distribution defined
on $G/H$ by $ \mathsf D_{G/H}$.
Then the  map $\mathsf J:(G/H,\mathsf D_{G/H})\mapsto (P,\mathsf D_P)$
is a forward Dirac map.
\end{theorem}

\begin{proof}
Choose $(v_p,\alpha_p)\in\mathsf D_P(p)$, for some $p\in P$ and  $gH\in G/H$ such that $\tg(g)=p$.
Then there exists $w_p\in T_p^\s G$ and $u_p\in \mathsf{G_0}(p)\cap T_pP$
such that $(w_p, (T_p\tg)^*\alpha_p)\in I^\s_p(\mathsf D_G)$ and $v_p=u_p+T_p\tg w_p$.
Since $\TT\s(w_p, (T_p\tg)^*\alpha_p)=(0_p,0_p)$, 
and $\mathsf D_G$ acts on $\mathsf D_{G/H}$, we find that
$(w_p, (T_p\tg)^*\alpha_p)\cdot (0_{gH},0_{gH})\in\mathsf D_{G/H}(gH)$.
We have 
$w_p\cdot 0_{gH}=T_p(q\circ R_g)w_p$
and, for all $ v_{gH}\in T_{gH}(G/H)$:
\begin{align*}
\widehat\Phi\left( (T_p\tg)^*\alpha_p, 0_{gH}\right)(v_{gH})
&=\widehat\Phi\left( (T_p\tg)^*\alpha_p, 0_{gH}\right)(T_{gH}\mathsf Jv_{gH}\cdot v_{gH})\\
&=((T_p\tg)^*\alpha_p)(T_{gH}\mathsf J(v_{gH}))+0_{gH}(v_{gH})\\
&=\alpha_p(T_{gH}(\tg\circ \mathsf J)v_{gH})
=\alpha_p(T_{gH}\mathsf J\,v_{gH})=((T_{gH}\mathsf J)^*\alpha_p)(v_{gH}).
\end{align*}
Thus, we have shown that 
$(T_p(q\circ R_g)w_p, (T_{gH}\mathsf J)^*\alpha_p)\in\mathsf D_{G/H}(gH)$.
We have $T_{gH}\mathsf J(T_p(q\circ R_g)w_p)=T_p(\mathsf J\circ q\circ R_g)w_p
=T_p(\tg\circ R_g)w_p=T_p\tg w_p=v_p-u_p$.
Choose $u_{gH}\in \mathsf{G_0}^{G/H}(gH)$ such that 
$T_{gH}\mathsf J(u_{gH})=u_p$.
Then the pair $(T_p(q\circ R_g)w_p+u_{gH}, (T_{gH}\mathsf J)^*\alpha_p)\in\mathsf D_{G/H}(gH)$
is such that $T_{gH}\mathsf J\left(T_p(q\circ R_g)w_p+u_{gH}\right)=v_p$.
\end{proof}

\subsection{The homogeneous Dirac structure on the classes of the units}
Let $G\rr P$ be a  Lie groupoid  and 
$G/H$ a smooth homogeneous space of $G\rr P$ endowed 
with a Dirac structure $\mathsf D_{G/H}$.
Consider the Dirac bundle $\lie D= q^*(\mathsf D_{G/H})\an{P}\subseteq
\mathsf P_G\an{P}$ over the units $P$. More explicitly, we have 
\begin{equation}\label{lieD}
\lie D(p)=\left\{(v_p,\alpha_p)\in  T_pG\times T_p^*G
\left| \begin{array}{c}
\exists (v_{pH},\alpha_{pH})\in\mathsf 
D_{G/H}(pH) \quad \text{ such that }\\
\alpha_p=(T_pq)^*\alpha_{pH} \text{ and } 
T_pqv_p=v_{pH}
\end{array}
\right.\right\}
\end{equation}
for all $p\in P$.

\begin{proposition}
Let $(G\rr P,\mathsf D_G)$ be a Dirac groupoid and $(G/H,\mathsf D_{G/H})$
 a Dirac homogeneous space of $(G\rr P,\mathsf D_G)$. 
Then  $\lie D\subseteq \mathsf P_G\an{P}$ defined as in \eqref{lieD}
satisfies
\begin{equation}
\label{lieDequation}
I^\s(\mathsf D_G)\subseteq \lie D
\subseteq 
 \lie A(\mathsf D_G)\oplus(\ker\TT\tg)\an{P}.
\end{equation} Thus,
the quotient $\bar{\lie D}=\lie D/I^\s(\mathsf D_G)$
is a smooth subbundle of $\lie B(\mathsf D_G)$.
We have  by definition $AH\oplus\{0\}\subseteq \lie D$.
\end{proposition}

\begin{proof}
Choose $p\in P$ and $(v_p,\alpha_p)\in I^\s_p(\mathsf D_G)=
\mathsf D_G(p)\cap\ker\TT\s$.
Then $\TT\s(v_p,\alpha_p)=(0_p,0_p)$ and the product 
$(v_p,\alpha_p)\cdot (0_{pH}, 0_{pH})$ makes sense.
Since $( 0_{pH}, 0_{pH})\in\mathsf D_{G/H}(pH)$,
we have then
$(T_pqv_p,\alpha_p\cdot 0_{pH})
=(v_p,\alpha_p)\cdot (0_{pH}, 0_{pH})\in\mathsf D_{G/H}(pH)$.
But $\alpha_p\cdot 0_{pH}$ is such that
$(T_pq)^*(\alpha_p\cdot 0_{pH})=\alpha_p\star ((T_pq)^*0_{pH})=\alpha_p$, and
we have  hence $(v_p,\alpha_p)\in\lie D(p)$ by definition of $\lie D$.

The inclusion $I^\s(\mathsf D_G)\subseteq\lie D$
yields immediately $\lie D=\lie D^\perp\subseteq (I^\s(\mathsf D_G))^\perp
=\mathsf D_G\an{P}+(\ker\TT\tg)\an{P}=\mathfrak A(\mathsf D_G)\oplus(\ker\TT\tg)\an{P}$.
\end{proof}

\begin{theorem}\label{LieD}
Let $(G\rr P,\mathsf D_G)$ be a Dirac groupoid
and $\lie D$  a Dirac subspace of $\mathsf P_G\an{P}$
satisfying \eqref{lieDequation}. Then the set
$\mathsf D=\mathsf D_G\cdot \lie D \subseteq \mathsf P_G$ defined
by 
\[\mathsf D(g)=\left\{(v_g,\alpha_g)\star(v_{\s(g)},\alpha_{\s(g)})\left|
\begin{array}{c}
(v_g,\alpha_g)\in\mathsf D_G(g), \\ (v_{\s(g)},\alpha_{\s(g)})\in\lie D(\s(g)),\\
\TT\s(v_g,\alpha_g)=\TT\tg(v_{\s(g)},\alpha_{\s(g)})
\end{array}
\right.\right\}
\]
is a Dirac structure on $G$ and $(G,\mathsf D)$ is a Dirac homogeneous
space of $(G\rr P,\mathsf D_G)$.
\end{theorem}

Note that $\mathsf D_G\cap\ker\TT\s\subseteq \mathsf D$ by construction:
for all $(v_g,\alpha_g)\in\mathsf D_G(g)\cap\ker\TT\s$, 
we have $\TT\s(v_g,\alpha_g)=(0_{\s(g)}, 0_{\s(g)})\in\lie D(\s(g))$
and hence $(v_g,\alpha_g)=(v_g,\alpha_g)\star (0_{\s(g)}, 0_{\s(g)})
\in\mathsf D(g)$. 

\begin{proof}
By Lemma \ref{sections_of_D'}, 
$\mathsf D$ is spanned by sections 
$\xi+\sigma^l$ 
such that 
$\bar\xi+\sigma$
is a section of $\lie D$ (with 
$\bar\xi\in\Gamma(\lie A(\mathsf D_G))$
and $\sigma\in\Gamma((\ker\TT\tg)\an{P})$ )
 and all the sections of $\mathsf D_G\cap\ker\TT\s$. 
This shows that $\mathsf D$ is smooth.

Choose $(v_g,\alpha_g)\star (v_{\s(g)},\alpha_{\s(g)})$
and 
$(w_g,\beta_g)\star (w_{\s(g)},\beta_{\s(g)})
\in\mathsf D(g)$, that is, with\linebreak
$(v_g,\alpha_g), (w_g,\beta_g)\in\mathsf D_G(g)$ and 
$(v_{\s(g)},\alpha_{\s(g)}), (w_{\s(g)},\beta_{\s(g)})\in\lie D(\s(g))$.
We have then 
\begin{align*}
&\bigl\langle (v_g,\alpha_g)\star (v_{\s(g)},\alpha_{\s(g)}),
(w_g,\beta_g)\star (w_{\s(g)},\beta_{\s(g)})\bigr\rangle\\
=\,&\alpha_g(w_g)+\alpha_{\s(g)}(w_{\s(g)})+
\beta_g(v_g)+\beta_{\s(g)}(v_{\s(g)})\\
=\,&\langle (v_g, \alpha_g), (w_g, \beta_g)\rangle
+\langle (v_{\s(g)},\alpha_{\s(g)}), (w_{\s(g)},\beta_{\s(g)})\rangle=0.
\end{align*}
This shows $\mathsf D\subseteq \mathsf D^\perp$.

For the converse inclusion, choose $(w_g,\beta_g)\in\mathsf D(g)^\perp$.
Then $$(w_g,\beta_g)\in(\mathsf D_G(g)\cap\ker\TT\s)^\perp
=(\mathsf D_G+ \ker\TT\tg)(g)$$
and consequently, we get the fact that 
$\TT\tg(w_g,\beta_g)\in\TT\tg(\mathsf D_G(g))=\lie A_{\tg(g)}(\mathsf D_G)$.
We write $\tg(g)=p$ and $\TT\tg(w_g,\beta_g)=\bar\xi(p)$
for some section $\bar\xi\in\Gamma(\mathfrak A(\mathsf D_G))$.
Consider a section $\xi\in\Gamma(\mathsf D_G)$
such that $\xi\sim_\s\bar\xi$. Then 
we have for all $(v_{\s(g)},\alpha_{\s(g)})\in\lie D(\s(g))$
and $(v_g,\alpha_g)\in\mathsf D_G(g)$ such that
$\TT\tg(v_{\s(g)},\alpha_{\s(g)})=\TT\s(v_g,\alpha_g)$:
\begin{align*}
\bigl\langle \xi(g\inv)\star (w_g, \beta_g),
\left(v_{\s(g)},\alpha_{\s(g)}\right)
\bigr\rangle
=\,&\bigl\langle \xi(g\inv)\star (w_g, \beta_g),
(v_g,\alpha_g)\inv\star(v_g,\alpha_g)\star (v_{\s(g)},\alpha_{\s(g)})
\bigr\rangle\\
=\,& \langle (w_g, \beta_g), (v_g,\alpha_g)\star (v_{\s(g)},\alpha_{\s(g)})\rangle
+\langle \xi(g\inv), (v_g,\alpha_g)\inv\rangle\\
=\,&0,
\end{align*}
since $(v_g,\alpha_g)\star (v_{\s(g)},\alpha_{\s(g)})\in\mathsf D(g)$
and $(v_g,\alpha_g)\inv\in\mathsf D_G(g\inv)$.
This proves that 
$$\xi(g\inv)\star (w_g, \beta_g)\in\lie D(\s(g))^\perp=\lie D(\s(g)),$$
and hence, if we write 
$\xi(g\inv)\star (w_g, \beta_g)=(w_{\s(g)},\beta_{\s(g)})
\in\lie D(\s(g))$,
$$(w_g, \beta_g)=\left(\xi(g\inv)\right)\inv
\star (w_{\s(g)},\beta_{\s(g)})\in\mathsf D(g).$$

\medskip

The second claim is obvious since the restriction to 
$\mathsf D$ of the map $\TT\mathsf J$ has image in $\TT\tg(\mathsf D_G)
=\lie A(\mathsf D_G)$ and, by
construction
of $\mathsf D$, the map $\mathsf D_G\times_{\mathfrak A(\mathsf D_G)}\mathsf D$,
$((v_g,\alpha_g),(v_h,\alpha_h))\mapsto (v_g,\alpha_g)\star (v_h,\alpha_h)$
is a well-defined Lie groupoid action. 
 \end{proof}

\begin{theorem}\label{unicity}
In the situation of the preceding theorem, if 
$\lie D$ is the restriction to $P$ of the pullback  $q^*(\mathsf D_{G/H})$ (as in \eqref{lieD})
for some Dirac homogeneous space 
$(G/H, \mathsf D_{G/H})$ of $(G\rr P,\mathsf D_G)$, then 
$\mathsf D=q^*(\mathsf D_{G/H})$.
\end{theorem}
\begin{proof}
Choose $(v_g,\alpha_g)\in q^*\left(\mathsf D_{G/H}\right)(g)$. Then $\alpha_g$
is equal to $(T_gq)^*\alpha_{gH}$ for some
$\alpha_{gH}\in T^*_{gH}(G/H)$ such that $(T_gqv_g,\alpha_{gH})\in\mathsf
D_{G/H}(gH)$.
Then $\TT\mathsf J(T_gqv_g,\alpha_{gH})=\TT\tg(v_g,\alpha_g)\in \lie A_{\tg(g)}(\mathsf
D_G)$ and 
there exists $(w_{g\inv},\beta_{g\inv})\in\mathsf D_G(g\inv)$ such that
$$\TT\s(w_{g\inv},\beta_{g\inv})=\TT\mathsf J(T_gqv_g,\alpha_{gH}).$$
Set $p=\s(g)$ and consider 
$(u_{pH},\gamma_{pH}):=(w_{g\inv},\beta_{g\inv})\cdot (T_gqv_g,\alpha_{gH})\in\mathsf D_{G/H}(pH)$.
Then we have $$(T_pq)^*\gamma_{pH}=\beta_{g\inv}\star((T_gq)^*\alpha_{gH})=
\beta_{g\inv}\star\alpha_g$$ by Proposition \ref{action_hom_cot}
about the action of $T^*G\rr A^*G$ on $\widehat{\mathsf J}:T^*(G/H)\to A^*G$, 
and 
$$u_{pH}=w_{g\inv}\cdot (T_gqv_g)=T_{(g\inv, gH)}\Phi(w_{g\inv}, T_gqv_g)
=T_pq(w_{g\inv}\star v_g).$$
Thus, $(u_p,\gamma_p):=(w_{g\inv},\beta_{g\inv})\star (v_g,\alpha_g)$ is an element 
of $\lie D(p)$, and we have 
$(v_g,\alpha_g)=(w_{g\inv},\beta_{g\inv})\inv\star(u_p,\gamma_p)$.
Since $\mathsf D_G$ is multiplicative and $(w_{g\inv},\beta_{g\inv})\in\mathsf D_G(g\inv)$, the pair
$(w_{g\inv},\beta_{g\inv})\inv$ is an element of $\mathsf D_G(g)$ 
and we have shown that 
$(v_g, \alpha_g)\in\mathsf D(g)$.

Since $q^*(\mathsf D_{G/H})\subseteq\mathsf D$ is an inclusion of Dirac structures, we have then equality.
\end{proof}

\begin{remark}
Note that 
 Theorem \ref{unicity} shows 
that if $(G\rr P,\mathsf D_G)$ is a Dirac groupoid, 
a $\mathsf D_G$-homogeneous Dirac structure on $G/H$ is uniquely
determined by its restriction to $q(P)\subseteq G/H$. 
\end{remark}

\begin{example}
We have seen in Example \ref{easy_example}
that if $(G\rr P, \mathsf D_G)$ is a Dirac groupoid, then
 $(\tg:G\to P, \mathsf D_G)$ is
a Dirac homogeneous space of $( G\rr P, \mathsf D_G)$. 

The space $\lie D$ is here the direct sum $I^\s(\mathsf D_G)\oplus \mathfrak A(\mathsf
D_G)$.
The corresponding Dirac structure $\mathsf D$ is equal to $\mathsf D_G$ by the
last theorem. This can also be seen directly from the definition of $\mathsf
D$, since $\mathsf D$ is spanned by the sections
$(X_\xi,\theta_\xi)$ for $(\bar{X_\xi},\bar\theta_\xi)\in\Gamma(\mathfrak A(\mathsf D_G))$ and the
sections  $\sigma^r$
for all $\sigma\in\Gamma( I^\s(\mathsf D_G))$, 
which are spanning sections for $\mathsf D_G$.
\end{example}

\subsection{The Theorem of Drinfel$'$d}
Recall that if $(G\rr P, \mathsf D_G)$ is a Dirac groupoid, then 
there is an induced action of the set of bisections $\mathcal B(G)$ of $G$
on the vector bundle $\lie B(\mathsf D_G)$ associated to $\mathsf D_G$
(see Theorem \ref{action}).
If $H$ is a wide Lie subgroupoid of $G\rr P$, this action restricts to an action
of $\mathcal B(H)$ on $\lie B(\mathsf D_G)$.
We use this action to characterize 
$\mathsf D_G$-homogeneous Dirac structures on $G/H$.

\begin{theorem}\label{drinfeld1}
Let $(G\rr P,\mathsf D_G)$ be a Dirac  groupoid,
$H$ a $\tg$-connected wide subgroupoid of $G$ such that 
the homogeneous space $G/H$ has a smooth manifold structure 
and  $q:G\to G/H$ is a smooth surjective submersion.
Let $\lie D$ be a Dirac subspace of $\mathsf P_G\an{P}$
satisfying \eqref{lieDequation}
and such that $AH\oplus\{0\}\subseteq \lie D$. 
Then the following are equivalent:
\begin{enumerate}
\item $\lie D$ is the pullback  $q^*(\mathsf D_{G/H})\an{P}$ as in \eqref{lieD},
where $\mathsf D_{G/H}$ is some  $\mathsf D_G$-homogeneous Dirac structure 
 on $G/H$.
\item $\bar{\lie D}=\lie D/I^\s(\mathsf D_G)\subseteq\lie B(\mathsf D_G)$
is invariant under the induced action of $\mathcal B(H)$
on $\lie B(\mathsf D_G)$.
\item The $\mathsf D_G$-homogeneous Dirac structure 
$\mathsf D=\mathsf D_G\cdot\lie D \subseteq \mathsf P_G$ as in Theorem
\ref{LieD}
pushes-forward to a ($\mathsf D_G$-homogeneous) 
Dirac structure on  the quotient $G/H$.
 \end{enumerate}
\end{theorem}

Note that, together with Theorem \ref{unicity}, this shows that 
a Dirac structure $\mathsf D_{G/H}$ on $G/H$ is $\mathsf D_G$-homogeneous
if and only if 
$I^\s(\mathsf D_G)\subseteq (q^*\mathsf D_{G/H})\an{P}$
and
$q^*\mathsf D_{G/H}=\mathsf D_G\cdot (q^*\mathsf D_{G/H})\an{P}$, that is, 
$(G/H,\mathsf D_{G/H})$ is $(G\rr P,\mathsf D_G)$-homogeneous
if and only if $(G,q^*\mathsf D_{G/H})$ is.

For the proof of Theorem \ref{drinfeld1}, we will need the following
Lemma.
\begin{lemma}\label{lemma}
In the situation of Theorem \ref{LieD}, 
we have $\lie D=\mathsf D\an{P}$.
\end{lemma} 

\begin{proof}
Choose $p\in P$ and $(v_p,\alpha_p)\in \lie D(p)$. Then 
$\TT\tg(v_p, \alpha_p)\in\lie A_p(\mathsf D_G)\subseteq \mathsf D_G(p)$
and $(v_p, \alpha_p)=\TT\tg(v_p, \alpha_p)\star(v_p, \alpha_p)\in\mathsf D(p)$.
This shows $\lie D\subseteq \mathsf D\an{P}$ and we are done since
both vector bundles have the same rank.
\end{proof}

\begin{proof}[of Theorem \ref{drinfeld1}]
Assume first that $\lie D=q^*(\mathsf D_{G/H})\an{P}$ for some\linebreak
 $\mathsf D_G$-homogeneous Dirac structure 
 $\mathsf D_{G/H}$ on $G/H$ and choose $K\in\mathcal B(H)$ and 
$(v_p,\alpha_p)\in \lie D(p)$, $p\in P$. Then there exists $\alpha_{pH}\in
T_{pH}^*(G/H)$
such that $\alpha_p=(T_pq)^*\alpha_{pH}$
and $(T_pqv_p,\alpha_{pH})\in\mathsf D_{G/H}(pH)$.
If we set $K(p)=:h\in H$ and  write
$\overline{(v_p,\alpha_p)}$ 
for $(v_p,\alpha_p)+I_p^\s(\mathsf D_G)\in \bar{\lie D}(p)\subseteq \lie B_p(\mathsf D_G)$, 
we have 
$$\rho_K\left(\overline{(v_p,\alpha_p)}\right)
=\left(T_{h\inv}R_K(v_{h\inv}\star v_p), (T_{\s(h)}R_K\inv)^*(\alpha_{h\inv}\star \alpha_p)\right)
+I_{\s(h)}^\s(\mathsf D_G)$$
for any $(v_{h\inv},\alpha_{h\inv})\in\mathsf D_G(h\inv)$
satisfying $\TT\s(v_{h\inv},\alpha_{h\inv})=\TT\tg(v_p,\alpha_p)$.
Since $$\TT\mathsf J(T_pqv_p, \alpha_{pH})=\TT\tg(v_p,\alpha_p)
=\TT\s(v_{h\inv}, \alpha_{h\inv}),$$ the product
$(v_{h\inv}, \alpha_{h\inv})\cdot (T_pqv_p, \alpha_{pH})$
makes sense and is an element 
of $\mathsf D_{G/H}(\s(h)H)$.
Note that since $K\in\mathcal B(H)$, we have $q\circ R_K=q$.
The pair  
$(T_{h\inv}R_K(v_{h\inv}\star v_p), (T_{\s(h)}R_K\inv)^*(\alpha_{h\inv}\star \alpha_p))$
satisfies 
$T_{\s(h)}q(T_{h\inv}R_K(v_{h\inv}\star v_p))\in T_{\s(h)H}(G/H)$,
$$T_{\s(h)}q(T_{h\inv}R_K(v_{h\inv}\star v_p))
=T_{h\inv}(q\circ R_K)(v_{h\inv}\star v_p)
=T_{h\inv}q(v_{h\inv}\star v_p)
=v_{h\inv}\cdot (T_pqv_p)$$
and 
\begin{align*}
(T_{\s(h)}R_K\inv)^*(\alpha_{h\inv}\star \alpha_p)
&=(T_{\s(h)}R_K\inv)^*(\alpha_{h\inv}\star (T_pq)^*\alpha_{pH})\\
&=(T_{\s(h)}R_K\inv)^*\left((T_{h\inv}q)^*(\alpha_{h\inv}\cdot \alpha_{pH})\right)
=(T_{\s(h)}q)^*(\alpha_{h\inv}\cdot\alpha_{pH}).
\end{align*}
Thus, 
$(T_{h\inv}R_K(v_{h\inv}\star v_p), (T_{\s(h)}R_K\inv)^*(\alpha_{h\inv}\star \alpha_p))$
is an element of $\lie D(\s(h))$ and 
$\rho_K\left(\overline{(v_p,\alpha_p)}\right)$
is an element of $\bar{\lie D}(\s(h))$.
This shows $(1)\Rightarrow (2)$.

\medskip

Assume now that $\bar{\lie D}$ is invariant under the action 
of $\mathcal B(H)$ on $\lie B(\mathsf D_G)$.
Recall the backgrounds about Dirac reduction
in Section \ref{background2}.
Set $\mathcal K=\mathcal H\oplus0_{T^*G}$, 
and hence $\mathcal K^\perp=TG\oplus \mathcal H^\circ$.

We have  $AH\oplus\{0\}\subseteq \lie D$ by hypothesis.
By definition of $\mathsf D$ and $\mathcal H$, 
this yields immediately 
$\mathcal K=\mathcal H\oplus\{0\}\subseteq\mathsf D$, 
hence 
 $\mathsf D\subseteq \mathcal K^\perp$ and 
$\mathsf D\cap\mathcal K^\perp=\mathsf D$ has constant rank on $G$.
By \eqref{condition_reduction}, we
 have  to show that $\mathsf D$ is invariant under the right action 
of $\mathcal B(H)$
on $G$. 
We will use the fact that $\mathsf D$ is spanned 
by the sections $\sigma^r\in\Gamma(\mathsf D_G\cap\ker\TT\s)$ 
for all $\sigma\in\Gamma(I^\s(\mathsf D_G))$ 
and  
$(X_\xi,\theta_\xi)+(X^l,\s^*\alpha)$
for all sections 
$(\bar X_\xi,\bar\theta_\xi)+(X,(\s^*\alpha)\an{P})\in\Gamma(\lie D)\subseteq\Gamma(\lie A(\mathsf D_G)
\oplus(\ker\TT\tg)\an{P})$.

Choose $K\in\mathcal B(H)$.
It is easy to verify that
$$(R_K^*Z^r, R_K^*(\tg^*\gamma))=(Z^r,\tg^*\gamma)\quad \text{ 
for all }\quad (Z,(\tg^*\gamma)\an{P})\in\Gamma(\ker\TT\s\an{P}).$$
Choose a section $(X_\xi,\theta_\xi)+(X^l,\s^*\alpha)$
of $\mathsf D$. We want to show that $(R_K^*(X_\xi+X^l), R_K^*(\theta_\xi+\s^*\alpha))$
is then also a section of $\mathsf D$.
Choose $g\in G$ and set for simplicity $h=K(\s(g))\in H$, $p=\s(h)$, $q=\tg(h)=\s(g)$
and $(\bar{X_\xi}+X,\bar\theta_\xi+\s^*\alpha)(p)=:(u_p,\gamma_p)\in\lie D(p)$. Then 
$\left((X_\xi,\theta_\xi)+(X^l,\s^*\alpha)\right)(g\star h)=
(X_\xi,\theta_\xi)(g\star h)\star (u_p,\gamma_p)$ and we can compute
\begin{align*}
&\left(R_K^*(X_\xi+X^l), R_K^*(\theta_\xi+\s^*\alpha)\right)(g)\\
=&\left(T_{g\star h}R_K\inv(X_\xi+X^l)(g\star h), 
(T_{g}R_K)^*((\theta_\xi+\s^*\alpha)(g\star h))\right)\\
=&\left(T_{g\star h}R_K\inv \left(X_\xi(gh)\star u_p\right),
(T_{g}R_K)^*\left(\theta_\xi(gh)\star\gamma_p\right) \right).
\end{align*}
Choose $(v_g,\alpha_g)\in\mathsf D_G(g)$
such that $\TT\tg(v_g,\alpha_g)=\TT\tg(X_\xi(gh), \theta_\xi(gh))$. Then 
the product $(w_h,\alpha_h):=(v_g,\alpha_g)\inv\star (X_\xi(gh), \theta_\xi(gh))$
is an element of $\mathsf D_G(h)$ such that
$\TT\s(w_h,\alpha_h)=(\bar{X_\xi},\bar\theta_\xi)(p)$ and
we have 
\begin{align*}
&\left(R_K^*(X_\xi+X^l), R_K^*(\theta_\xi+\s^*\alpha)\right)(g)\\
=&(v_g,\alpha_g)\star\left(T_{h}R_K\inv \left(v_{g}\inv
\star X_\xi(gh)\star u_p\right),
(T_{g}R_K)^*\left(\alpha_g\inv\star\theta_\xi(gh)\star\gamma_p\right) \right)\\
=&(v_g,\alpha_g)\star\left(T_{h}R_K\inv 
\left(w_h\star u_p\right),
(T_{q}R_K)^*\left(\beta_h\star\gamma_p\right) \right).
\end{align*}
But since $\bar{\lie D}$ is invariant under the action 
of $\mathcal B(H)$ on $\lie B(\mathsf D_G)$ and 
$\overline{(u_p,\gamma_p)}=
(u_p,\gamma_p)+I_p^\s(\mathsf D_G)$ is an element of $\bar{\lie D}(p)$, 
we have $$
\left(T_{ h}R_K\inv 
\left(w_h\star u_p\right),
(T_{q}R_K)^*\left(\beta_h\star \gamma_p\right)\right)
+I_{q}^\s(\mathsf D_G)=\rho_{K\inv}\left(\overline{(u_p,\gamma_p)}\right)\in\bar{\lie D}(q).$$
Because $I_{q}^\s(\mathsf D_G)\subseteq \lie D(q)$, we have 
consequently $$\left(T_{h}R_K\inv 
\left(w_h\star u_p\right),
(T_{q}R_K)^*\left(\beta_h\star\gamma_p\right) \right)
\in\lie D(q)$$
and hence $$\left(R_K^*(X_\xi+X^l), R_K^*(\theta_\xi+\s^*\alpha)\right)(g)
=(v_g,\alpha_g)\star\left(T_{h}R_K\inv 
\left(w_h\star u_p\right),
(T_{q}R_K)^*\left(\beta_h\star\gamma_p\right) \right)
\in\mathsf D(g)$$
since $(v_g,\alpha_g)\in\mathsf D_G(g)$.

We show then that the push-forward $q(\mathsf D)$ is 
 a $\mathsf D_G$-homogeneous Dirac structure on $G/H$. 
By definition of $\TT\mathsf J$, we have $\TT\mathsf  J(q(\mathsf D))=\TT\tg(\mathsf D)
\subseteq \TT\tg(\mathsf D_G)
= \mathfrak A(\mathsf D_G)$.
Choose $(v_{gH},\alpha_{gH})\in q(\mathsf D)(gH)$
and $(w_{g'},\beta_{g'})\in\mathsf D_G(g')$
such that $\TT\s(w_{g'},\beta_{g'})
=\TT\mathsf  J(v_{gH},\alpha_{gH})$. Then there exists 
$v_g\in T_gG$ such that 
$T_gqv_g=v_{gH}$ and $(v_g, (T_gq)^*\alpha_{gH})\in\mathsf D(g)$.
The pair $(v_g, (T_gq)^*\alpha_{gH})$
satisfies then $\TT\tg(v_g, (T_gq)^*\alpha_{gH})=\TT\mathsf  J(v_{gH},\alpha_{gH})
=\TT\s(w_{g'},\beta_{g'})$ and since $(G,\mathsf D)$ is a Dirac homogeneous
space of $(G\rr P,\mathsf D_G)$, we have 
$(w_{g'},\beta_{g'})\star (v_g, (T_gq)^*\alpha_{gH})\in\mathsf D(g'\star g)$
and the identities 
$(T_{g'\star g}q)^*(\beta_{g'}\cdot\alpha_{gH})=\beta_{g'}\star (T_gq)^*\alpha_{gH}$
and $T_{g'\star g}q(w_{g'}\star v_g)=w_{g'}\cdot (T_gqv_g)=
w_{g'}\cdot v_{gH}$. Thus, the pair 
$(w_{g'},\beta_{g'})\cdot (v_{gH},\alpha_{gH})$
is an element of $q(\mathsf D)(gg'H)$
and $q(\mathsf D)$ is shown to be $\mathsf D_G$-homogeneous.
Hence, we have shown $(2)\Rightarrow (3)$.

\medskip

To show that $(3)$ implies $(1)$, we have then just 
to show that the vector bundle 
$\lie D\to P$ is the
restriction to $P$ of the
 pullback  $q^*(q(\mathsf D))$.
Since $\mathsf D\an{P}=\lie D$ by Lemma \ref{lemma}, we can show that
$\mathsf D=q^*(q(\mathsf D))$. This follows from the inclusion
$\mathcal H\oplus 0_{T^*G}\subseteq \mathsf D$.
Choose $(v_g,\alpha_g)\in\mathsf D(g)$. Then 
$\alpha_g\in \mathcal H(g)^\circ$. Thus, 
there exists  $\alpha_{gH}\in T_{gH}^*(G/H)$
such that $\alpha_g=(T_gq)^*\alpha_{gH}$ and,  
by definition of $q(\mathsf D)$, the pair  $(T_gqv_g, \alpha_{gH})$
is an element of $q(\mathsf D)(gH)$. But then 
$(v_g,\alpha_g)\in q^*(q(\mathsf D))(g)$. Conversely, if 
$(v_g,\alpha_g)\in q^*(q(\mathsf D))(g)$, then 
$\alpha_g=(T_gq)^*\alpha_{gH}$
for some $\alpha_{gH}\in T_{gH}^*(G/H)$ satisfying 
$(T_gq v_g,\alpha_{gH})\in q(\mathsf D)(gH)$. By definition
of $q(\mathsf D)(gH)$, there exists then 
$u_g\in T_gG$ such that $(u_g,(T_gq)^*\alpha_{gH})=(u_g,\alpha_g)\in\mathsf D(g)$
and $T_gq u_g=T_gq v_g$. But this yields that $v_g-u_g\in \mathcal H(g)$ and hence 
$(v_g,\alpha_g)=(u_g,\alpha_g)+(v_g-u_g,0_g)\in \mathsf D(g)+(\mathcal H(g)\times\{0_g\})
=\mathsf D(g)$ since $(\mathcal H(g)\times\{0_g\})\subseteq\mathsf D(g)$.
\end{proof}

\begin{theorem}\label{drinfeld2}
Let $(G\rr P, \mathsf D_G)$ be a closed
Dirac groupoid. 
In the situation of the previous theorem, the 
following are equivalent
\begin{enumerate}
\item The Dirac structure $q(\mathsf D)=\mathsf D_{G/H}$
is closed.
\item The Dirac structure
$\mathsf D$ is closed.
\item The set of sections of $\bar{\lie D}\subseteq \lie B(\mathsf D_G)$ is closed under the bracket on 
the sections of the Courant algebroid $\lie B(\mathsf D_G)$.
\end{enumerate}
\end{theorem}

\begin{proof}
If $\mathsf D$ is closed, then $q(\mathsf D)$ is closed
by a Theorem in \cite{Zambon08} 
about  Dirac reduction by foliations
(see the generalities about 
Dirac reduction in Section \ref{background2}). 
Conversely, assume that $q(\mathsf D)$ is closed. 
Since $\mathsf D\subseteq TG\oplus\mathcal H^\circ$
and by the proof of Theorem \ref{drinfeld1},
the Dirac structure $\mathsf D$ is spanned by $q$-descending sections, 
that is, sections $(X,\alpha)$ such that $\alpha\in\Gamma(\mathcal H^\circ)$
and 
$R_K^*(X,\alpha)=(X,\alpha)$ for all $K\in\mathcal B(H)$.
Choose two descending sections $(X,\alpha), (Y,\beta)$ of 
$\mathsf D$. Choose $(\bar X,\bar \alpha), (\bar Y,\bar \beta)\in\Gamma(q(\mathsf D))$
such that $(X,\alpha)\sim_q(\bar X,\bar \alpha)$
 and $(Y,\beta)\sim_q(\bar Y,\bar \beta)$. Then the bracket
$[(X,\alpha), (Y,\beta)]$ descends to $[(\bar X,\bar \alpha), (\bar Y,\bar \beta)]$
which is a section of  $q(\mathsf D)$ since $(G/H, q(\mathsf D))$ is closed.
But since $\mathcal H\oplus 0_{T^*G}\subseteq \mathsf D$, we have 
$\mathsf D=q^*(q(\mathsf D))$
(recall the proof of  Theorem \ref{drinfeld1}). Since $[(X,\alpha), (Y,\beta)]$ is a section of 
$q^*(q(\mathsf D))$, we have shown that 
$[(X,\alpha), (Y,\beta)]\in\Gamma(\mathsf D)$.
This proves $(1)\iff (2)$.

\medskip

Assume that $(G,\mathsf D)$ is closed and choose two sections 
$e_{\xi,X,\alpha}=(\bar{X_\xi}+X,\bar\theta_\xi+\s^*\alpha)+I^\s(\mathsf D_G), 
e_{\eta,Y,\beta}=(\bar{X_\eta}+Y,\bar\theta_\eta+\s^*\beta)+I^\s(\mathsf D_G)$
of $\bar{\lie D}\subseteq\lie B(\mathsf D_G)$. Then the two
pairs $(X_\xi+X^l,\theta_\xi+\s^*\alpha), (X_\eta+Y^l,\theta_\eta+\s^*\beta)$
are smooth sections of $\mathsf D$ by construction and 
since $(G,\mathsf D)$ is closed, we have 
$[(X_\xi+X^l,\theta_\xi+\s^*\alpha), (X_\eta+Y^l,\theta_\eta+\s^*\beta)]\in\Gamma(\mathsf D)$.
But since $\lie D=\mathsf D\an{P}$ and 
 $[e_{\xi,X,\alpha}, e_{\eta,Y,\beta}]
=[(X_\xi+X^l,\theta_\xi+\s^*\alpha), (X_\eta+Y^l,\theta_\eta+\s^*\beta)]\an{P}
+I^\s(\mathsf D_G)$, this yields
 $[e_{\xi,X,\alpha}, e_{\eta,Y,\beta}]\in\Gamma(\bar{\lie D})$.

Conversely, assume  that $\Gamma(\bar{\lie D})$ is closed under the Courant bracket on sections
of $\lie B(\mathsf D_G)$ and choose 
two spanning sections  $(X_\xi+X^l,\theta_\xi+\s^*\alpha), (X_\eta+Y^l,\theta_\eta+\s^*\beta)$
of $\mathsf D$ corresponding to 
$(\bar{X_\xi}+X,\bar\theta_\xi+\s^*\alpha\an{P})$ and 
$(\bar{X_\eta}+Y,\bar\theta_\eta+\s^*\beta\an{P})
\in\Gamma(\lie D)\subseteq\Gamma(\lie A(\mathsf D_G)\oplus(\ker\TT\tg)\an{P})$.
Since $[e_{\xi,X,\alpha},e_{\eta, Y,\beta}]$ is then an element 
of $\Gamma(\bar{\lie D})$ and $I^\s(\mathsf D_G)\subseteq \lie D$, we have 
$$[(X_\xi+X^l,\theta_\xi+\s^*\alpha), (X_\eta+Y^l,\theta_\eta+\s^*\beta)]\an{P}
\in\Gamma(\lie D)$$
by definition of the bracket on the sections of $\lie B(\mathsf D_G)$.
By Theorem \ref{lie_der_of_xi_section}, Lemma \ref{sections_of_D'}
and \eqref{important_bracket},
the value of 
 $[(X_\xi+X^l,\theta_\xi+\s^*\alpha), (X_\eta+Y^l,\theta_\eta+\s^*\beta)]$
at $g\in G$ 
equals 
\begin{align*}
&\Bigl(\left([(X_\xi,\theta_\xi),(X_\eta,\theta_\eta)]+\mathcal L_X\eta
-\mathcal L_Y\xi\right)(g)\Bigr)\star\\
&\hspace*{4cm}\left([(X_\xi+X^l,\theta_\xi+\s^*\alpha), (X_\eta+Y^l,\theta_\eta+\s^*\beta)](\s(g))\right)
\end{align*}
and we find that 
$[(X_\xi+X^l,\theta_\xi+\s^*\alpha), (X_\eta+Y^l,\theta_\eta+\s^*\beta)]$
is a section of $\mathsf D$,
since the first factor is an element 
of $\mathsf D_G(g)$ and the second an element 
of $\lie D(\s(g))$. 
Recall also that, by the proof 
of Theorem \ref{theorem_courant_algebroid},
we know that 
$[(X_\xi+X^l,\theta_\xi+\s^*\alpha), \sigma^r]\in\Gamma(\mathsf D_G\cap\ker\TT\s)$ for 
all $\sigma\in\Gamma(I^\s(\mathsf D_G))$. Finally, since 
$\mathsf D_G$ is closed, we know that
$[\sigma_1^r,\sigma_2^r]\in\Gamma(\mathsf D_G)$
for all $\sigma_1, \sigma_2 \in\Gamma(\mathsf D_G\cap\ker\TT\s)$.
Thus, 
by the Leibniz identity for the restriction
to $\Gamma(\mathsf D)$ of the Courant bracket on $\mathsf P_G$, we have 
shown that $(G,\mathsf D)$ is closed.
\end{proof}

As a corollary of the Theorems \ref{LieD}, \ref{unicity},
\ref{drinfeld1}, and \ref{drinfeld2} we get our main result,
that generalizes the correspondence theorems  in 
\cite{Drinfeld93},
\cite{LiWeXu98} and  \cite{Jotz11a}.

\begin{theorem}\label{Drinfeld}
Let $(G\rr P, \mathsf D_G)$ be a Dirac groupoid.
Let $H$ be a wide Lie subgroupoid of $G$ such that
the quotient $G/H$ is a smooth manifold and the map $q:G\to G/H$ a smooth surjective submersion.
There is a one-one correspondence 
between
 $\mathsf D_G$-homogeneous
Dirac structures on $G/H$ and 
Dirac 
subspaces $\lie D$ of $\mathsf P_G\an{P}$ such that
$AH\times\{0\}+ I^\s(\mathsf D_G)
\subseteq \lie D\subseteq \lie A(\mathsf D_G)\oplus(\ker\TT\tg)\an{P}$
and $\bar{\lie D}:=\lie D/I^\s(\mathsf D_G)$
is a $\mathcal B(H)$-invariant Dirac 
subspace
 of $\lie B(\mathsf D_G)$.

If $(G\rr P, \mathsf D_G)$ is closed, then 
closed $\mathsf D_G$-homogeneous Dirac structures on $G/H$
correspond in this way to Lagrangian subalgebroids $\bar{\lie D}$ of
$\lie B(\mathsf D_G)$.
\end{theorem}

\begin{remark}\label{remark_action_of_H_when_closed}
Assume that $(G\rr P, \mathsf D_G)$ is a closed Dirac  groupoid,
 $\lie D\subseteq \mathsf P_G\an{P}$
a Dirac subspace satisfying \eqref{lieDequation}
 and $AH\times\{0\}\subseteq \lie D$
for some $\tg$-connected wide Lie subgroupoid $H$ of $G\rr P$, and such that
$\lie D/I^\s(\mathsf D_G)\subseteq \lie B(\mathsf D_G)$
is closed under the bracket on $\lie B(\mathsf D_G)$.
It is easy to check (as in the proof of Theorem \ref{drinfeld2}) that we have 
then 
$\left[(X^l,0), (X,\alpha)\right]\in \Gamma(\mathsf D_G\cdot \lie D)$
for all $(X,\alpha)\in \Gamma(\mathsf D_G\cdot \lie D)$
and $X\in\Gamma(AH)$. Since $H$ is $\tg$-connected, we get then the fact
that $(R_K^*X, R_K^*\alpha)\in \Gamma(\mathsf D_G\cdot \lie D)$
for all bisections $K\in\mathcal B(H)$
and the Dirac structure $\mathsf D_G\cdot \lie D$
projects to a Dirac structure on $G/H$, that is 
$\mathsf D_G$-homogeneous. 
The quotient $\lie D/ I^\s(\mathsf D_G)$ is then
automatically invariant
under the induced action of the
bisections $\mathcal B(H)$ on $\lie B(\mathsf D_G)$
and this shows that the condition 2 of Theorem \ref{drinfeld1}
is always satisfied if $\mathsf D_G$ is closed,
 $\lie D/I^\s(\mathsf D_G)$ is closed
under the Courant bracket on sections of $\lie B(\mathsf D_G)$
and $H$ is $\tg$-connected.
\end{remark}

\begin{example}
In \cite{LiWeXu98}, it is shown  that 
for a Poisson groupoid $(G\rr P, \pi_G)$, there is a one to one 
 correspondence between 
$\pi_G$-homogeneous Poisson structures on smooth homogeneous spaces $G/H$ and regular closed
Dirac structures $L$ of the Courant algebroid
$AG\oplus A^*G$,
such that $H$ is  the $\tg$-connected  subgroupoid of 
$G$ corresponding to the subalgebroid
$L\cap (AG\times0_{A^*G})$.
Since pullbacks to $G$ of Poisson 
structures on $G/H$ correspond
to closed Dirac structures 
on $G$ with characteristic distribution  $\mathcal H$,
we recover this result as a special case 
of Theorem \ref{Drinfeld}, 
using 
Remark \ref{remark_action_of_H_when_closed}
and the isomorphism in
Example \ref{bialgebroid_of_Poisson}.

Note that in this particular situation of a Poisson  groupoid, 
Theorem \ref{Drinfeld} 
classifies not only the Poisson homogeneous spaces of $(G\rr P, \pi_G)$,
but all its (not necessarily closed) Dirac homogeneous spaces.
\end{example}

\begin{example}
Let $(G\rr P, \pi_G)$ be a Poisson groupoid and $H$ a wide subgroupoid of $G$. 
Assume that the Poisson structure
descends to the quotient $G/H$, i.e., that $\pi_G$ is invariant under the 
action of the bisections of $H$. 
Let $\pi$ be the induced structure on $G/H$. We show that $(G,q^*\mathsf D_{\pi})$
is a Dirac homogeneous space of $(G\rr P, \pi_G)$. This is equivalent to the fact that
$(G/H,\pi)$ is a Poisson homogeneous space of $(G\rr P, \pi_G)$.

The Dirac structure $q^*\mathsf D_{\pi}$ is equal
to $(\mathcal H\oplus0_{T^*G})\oplus 
\graph\left(\left.\pi_G^\sharp\right\an{\mathcal H^\circ}:\mathcal H^\circ\to TG\right)$.
Since $\mathcal H\subseteq T^\tg G$, the inclusion
$\TT\tg\left(q^*\mathsf D_{\pi}\right)\subseteq \lie A(\mathsf D_{\pi_G})$
is obvious. Choose $(v_g,\alpha_g)\in\left(q^*\mathsf D_{\pi}\right)(g)$
and 
$\alpha_h\in T_h^*G$
such that $\TT\s\left(\pi_G^\sharp(\alpha_h),\alpha_h\right)=\TT\tg(v_g,\alpha_g)$.
Then we have 
$(v_g,\alpha_g)=\left(u_g+\pi_G^\sharp(\alpha_g),\alpha_g\right)$
with some $u_g\in \mathcal H(g)$ and the product
$\left(\pi_G^\sharp(\alpha_h),\alpha_h\right)\star (v_g,\alpha_g)$
is equal to
\begin{align*}
(\pi_G^\sharp(\alpha_h),\alpha_h)\star (u_g+\pi_G^\sharp(\alpha_g),\alpha_g)
&=\left(\pi_G^\sharp(\alpha_h)\star \pi_G^\sharp(\alpha_g)+0_h\star u_g, \alpha_g\star \alpha_h\right)\\
&=\left(\pi_G^\sharp(\alpha_g\star \alpha_h)+ T_gL_h u_g,\alpha_g\star \alpha_h\right)
\end{align*}
since $\pi_G$ is multiplicative.  The vector $T_gL_h u_g$
is an element of $\mathcal H$ by definition and consequently,
$(\pi_G^\sharp(\alpha_h),\alpha_h)\star (u_g+\pi_G^\sharp(\alpha_g),\alpha_g)$
is an element of $q^*(\mathsf D_\pi)$, which is shown to be 
$\pi_G$-homogeneous. It corresponds to the Lagrangian
subalgebroid $(AH\times 0_{T^*P})
\oplus \graph\left(\left.\pi_G^\sharp\right\an{AH^\circ}:AH^\circ\to TP\right)+I^\s(D_{\pi_G})$
of $\lie B(\mathsf D_{\pi_G})$, 
or  more simply, to the Lagrangian subalgebroid $AH\oplus AH^\circ$
in the Courant algebroid $AG\oplus A^*G$.

Thus, Theorem \ref{Drinfeld} together with the isomorphism in Example 
\ref{bialgebroid_of_Poisson}
shows that the multiplicative Poisson structure on $G$ descends to $G/H$ 
if and only if the Lagrangian subspace
$AH\oplus AH^\circ$ is a subalgebroid of the Courant algebroid $AG\oplus A^*G$.

The Poisson homogeneous space that corresponds in this way to the Lagrangian subalgebroid
$AG\oplus0_{A^*G}$ is the Poisson manifold $(P,\pi_P)$, where 
$\pi_P$ is the Poisson structure induced on $P$ by $\pi_G$, see
\cite{Weinstein88b} and also Theorem \ref{induced_Dirac_on_P}. 
 Note that the other trivial Dirac structure 
$0_{AG}\oplus A^*G$ corresponds to $(G,\pi_G)$ seen as a 
Poisson homogeneous space  of $(G\rr P,\pi_G)$ (see Example \ref{extreme_homogeneous}).

In the same manner, we can show that if a Dirac groupoid $(G\rr P, \mathsf D_G)$ 
is invariant under the 
action of a wide subgroupoid $H$, and the Dirac structure descends to the 
quotient $G/H$, then
$(G/H,q(\mathsf D_G))$ is $(G\rr P, \mathsf D_G)$-homogeneous. 
For that, we use the formula
$q^*(q(\mathsf D_G))=\mathcal K_H+ \mathsf D_G\cap\mathcal K_H^\perp$.
In  particular, the 
Dirac structure on $P$ obtained under some regularity conditions in Theorem 
\ref{induced_Dirac_on_P}
is $\mathsf D_G$-homogeneous. As in the Poisson case, we find hence that the 
Dirac structure
descends to $G/H$
if and only if $$\overline{AH\oplus0_{T^*P}\oplus 
\lie A(\mathsf D_G)\cap(TP\oplus AH^\circ)}
\subseteq \lie B(\mathsf D_G)$$
is invariant under the induced action of $\mathcal B(H)$.
\end{example}

\begin{example}
Let $(M,\mathsf D_M)$ be a smooth Dirac manifold and 
$(M\times M\,\rr \,M, \mathsf D_M\ominus \mathsf D_M)$
the pair Dirac  groupoid associated to it.

The wide Lie subgroupoids of $M\times M\,\rr\, M$
are the equivalence relations $R\subseteq M\times M$, 
and the corresponding homogeneous 
spaces are the products $M\times M/R$.
For instance, if $\Phi:G\times M\to M$ is an action of a  Lie group $G$
(with Lie algebra $\lie g$)
 on $M$, the subset $R_G=\{(m, \Phi_g(m))\mid m\in M, g\in G\}$
is a wide subgroupoid of $M\times M$, and 
$(M\times M)/ R_G$ is easily seen to equal $M\times M/G$.
Hence, if the action is free and proper, the 
homogeneous space $(M\times M)/ R_G$
has a smooth manifold structure such that the projection
$q:M\times M\to M\times M/G$ is a smooth surjective submersion.

In this case, the bisections of $R_G$ are the diffeomorphisms 
of $M$ that leave the orbits of $G$ invariant. For instance,
for every $g\in G$, the map $K_g:\Delta_M\simeq M\to M\times M$,
$m\to (m,\Phi_g(m))$ is a bisection of $R_G$. 

Choose $(m,m)\in\Delta_M$. A vector $(v_m, w_m)\in T_{(m,m)}(M\times M)$
is an element of $T_{(m,m)}^\tg R_G$ if and only if 
$v_m=0_m$ and $(0_m,w_m)\in T_{(m,m)} R_G$, 
that is if and only if $v_m=0_m$ and $w_m\in T_m(G\cdot m)$.
Thus, if $\mathcal V$ is the vertical space of the action, 
we find that $AR_G=(\{0\}\oplus\V)\an{\Delta_M}$.

By Theorem \ref{drinfeld1}, $\mathsf D_M\ominus\mathsf D_M$-homogeneous  Dirac 
structures on $M\times M/G$ are in one-one correspondence 
with Lagrangian subspaces 
$\lie D$ of $\mathsf P_G\an{P}$ such that
$I^\s(\mathsf D_M\ominus\mathsf D_M)+(A R_G\times\{0\})\subseteq \lie D$ 
and such that $\lie D/I^\s(\mathsf D_M\ominus\mathsf D_M)$ 
is invariant under the induced action 
of $\mathcal B(R_G)$ on $\lie B(\mathsf D_M\ominus\mathsf D_M)$.
But since $I^\s_{(m,m)}(\mathsf D_M\ominus\mathsf D_M)+A_{(m,m)} R_G=
\{(v_m, \xi_M(m), \alpha_m, 0_m)\mid (v_m,\alpha_m)\in\mathsf D_M(m), \, \xi\in\lie g\}$,
we find, using the isomorphism in Example \ref{iso_pair_case} and the considerations in Example
\ref{action_pair_case}, that 
$(\mathsf D_M\ominus\mathsf D_M)\cdot \lie D$ 
is a product of Dirac structures 
$\mathsf D_M\oplus \mathsf D$ such that $\V\oplus0_{T^*G}\subseteq \mathsf D$ and 
$(\Phi_g^*X,\Phi_g^*\alpha)\in\Gamma(\mathsf D)$ for all $g\in G$ and 
$(X,\alpha)\in\Gamma(\mathsf D)$.
But Dirac structures $\mathsf D$ satisfying these conditions
 are exactly the pullbacks to $M$ of Dirac structures on $M/G$ 
and we find that the $\mathsf D_M\ominus\mathsf D_M$-homogeneous Dirac structures on $M\times M/G$
are of the form $\mathsf D_M\oplus \bar{\mathsf D}:=\mathsf D_M\oplus q_G(\mathsf D)$, 
where $q_G:M\to M/G$ is the canonical projection.

With Example \ref{iso_pair_case} and Theorem \ref{drinfeld2}, we get hence that $\bar{\mathsf D}$
is closed if and only if $\mathsf D$ is closed.
\end{example}

\begin{example}
The left invariant Dirac structures on a Lie group $G$ are the homogeneous 
structures relative to the trivial Poisson bracket on $G$ (\cite{Jotz11a}).
Hence, if we consider this example in the groupoid situation, we should recover the ``right'' definition
for left invariant Dirac structures on a Lie groupoid. We say that a  Dirac structure $\mathsf D$
on a Lie groupoid $G\rr P$ is \emph{left-invariant} 
if the action $\TT\Phi$ of $TG\oplus T^*G$ on 
$\TT\mathsf\tg:TG\oplus T^*G\to TP\oplus A^*G$ restricts 
to an action of $0_{TG}\oplus T^*G$ on
$\mathsf D$, i.e., $$\left(0_{TG}\oplus T^*G\right)\cdot\mathsf D=\mathsf D.$$

 In \cite{LiWeXu98}, a Dirac structure
on a Lie groupoid $G\rr P$ is said to be left-invariant if it is the pullback
under the map
\begin{equation*}
\begin{array}{cccc}
\Phi:T^\tg G\oplus T^*G&\to& AG\oplus A^*G& \\
(v_g,\alpha_g)&\mapsto&(T_gL_{g\inv}v_g,\hat \s(\alpha_g))& \in A_{\s(g)}G\times A^*_{\s(g)}G
\end{array}
\end{equation*}
of a Dirac structure in $AG\oplus A^*G$. These two definitions are easily seen to be equivalent,
the inclusion $0_{TG}\oplus(T^\tg G)^\circ\subseteq\mathsf D$ is immediate
and it is easy to check that $\mathsf D$ is invariant under the lifted
right actions of the bisections 
if and only if the corresponding Dirac structure in
$\lie B(T^*G)$ is invariant under the induced action of $\mathcal B(G)$ on 
 $\lie B(T^*G)$ (compare with Proposition 6.2 in \cite{LiWeXu98}).

The result in Theorem \ref{drinfeld2} implies that a left-invariant Dirac structure $\mathsf D$
is closed if and only if the corresponding Dirac structure $\Phi(\mathsf D\an{P})\subseteq 
AG\oplus A^*G$ is a subalgebroid. 
\end{example}

\def\cprime{$'$} \def\polhk#1{\setbox0=\hbox{#1}{\ooalign{\hidewidth
  \lower1.5ex\hbox{`}\hidewidth\crcr\unhbox0}}} \def\cprime{$'$}
  \def\cprime{$'$} \def\cprime{$'$} \def\cprime{$'$} \def\cprime{$'$}
  \def\cprime{$'$} \def\cprime{$'$}
  \def\polhk#1{\setbox0=\hbox{#1}{\ooalign{\hidewidth
  \lower1.5ex\hbox{`}\hidewidth\crcr\unhbox0}}}
  \def\polhk#1{\setbox0=\hbox{#1}{\ooalign{\hidewidth
  \lower1.5ex\hbox{`}\hidewidth\crcr\unhbox0}}}
  \def\polhk#1{\setbox0=\hbox{#1}{\ooalign{\hidewidth
  \lower1.5ex\hbox{`}\hidewidth\crcr\unhbox0}}}
  \def\polhk#1{\setbox0=\hbox{#1}{\ooalign{\hidewidth
  \lower1.5ex\hbox{`}\hidewidth\crcr\unhbox0}}} \def\cprime{$'$}
  \def\polhk#1{\setbox0=\hbox{#1}{\ooalign{\hidewidth
  \lower1.5ex\hbox{`}\hidewidth\crcr\unhbox0}}}
  \def\polhk#1{\setbox0=\hbox{#1}{\ooalign{\hidewidth
  \lower1.5ex\hbox{`}\hidewidth\crcr\unhbox0}}}
  \def\polhk#1{\setbox0=\hbox{#1}{\ooalign{\hidewidth
  \lower1.5ex\hbox{`}\hidewidth\crcr\unhbox0}}}
  \def\polhk#1{\setbox0=\hbox{#1}{\ooalign{\hidewidth
  \lower1.5ex\hbox{`}\hidewidth\crcr\unhbox0}}}
\providecommand{\bysame}{\leavevmode\hbox to3em{\hrulefill}\thinspace}
\providecommand{\MR}{\relax\ifhmode\unskip\space\fi MR }
\providecommand{\MRhref}[2]{%
  \href{http://www.ams.org/mathscinet-getitem?mr=#1}{#2}
}
\providecommand{\href}[2]{#2}

\bigskip

\noindent
\textbf{M. Jotz}\\
Section de Math{\'e}matiques\\ 
Ecole Polytechnique
  F{\'e}d{\'e}rale de Lausanne\\ 
CH-1015 Lausanne\\ Switzerland\\
\texttt{madeleine.jotz@epfl.ch}\\
Partially supported by Swiss NSF grant 200021-121512

\end{document}